\documentclass[11pt]{article}

\usepackage{amsmath,amssymb,amsthm,amsfonts,amstext,amsbsy,amscd}
\usepackage[utf8]{inputenc}
\usepackage{dsfont}
\usepackage{color}
\usepackage[colorlinks,citecolor=red,urlcolor=green,linkcolor = blue,hypertexnames=false]{hyperref}
\usepackage{graphicx}
\usepackage{subcaption}
\newcommand{\1}{\mathds{1}}
\newcommand{\paren}[1]{\left( #1 \right)}
\newcommand{\croch}[1]{\left[\, #1 \,\right]}
\newcommand{\acc}[1]{\left\{ #1 \right\}}

\newcommand{\scalH}[1]{\langle #1 \rangle_{\H}}

\renewcommand{\H}[1]{ \mathcal{H}\left( #1 \right)}

\newcommand{\R}{\mathbb{R}}

\renewcommand{\P}{\mathbb{P}}

\newcommand{\X}{\mathcal{X}}

\renewcommand{\H}{\mathcal{H}}

\newcommand{\norm}[1]{\left \lVert #1 \right\rVert}

\providecommand{\normH}[1]{\norm{#1}_{\H}}

\renewcommand{\P}{\mathbb{P}}

\DeclareMathOperator{\tr}{tr}

\renewcommand{\H}{\mathcal{H}}
\renewcommand{\P}{\mathbb{P}}

\newcommand{\Ybf}{\mathbf{Y}}

\newcommand{\bP}{\mathbf{P_\epsilon}}
\newcommand{\f}{\mathbf{f}}
\newcommand{\tauDP}{\tau_{DP}}
\newcommand{\tauSDP}{\tau_{SDP}}

\newcommand{\effdim}{\mathcal{N}}
\newcommand{\geffdim}{\mathcal{N}^g}
\newcommand{\tildengeffdim}{\widetilde{\mathcal{N}}^g_n}
\newcommand{\tildetstar}{\widetilde{\tstar}}
\newcommand{\Ltstar}{\widetilde{t}^*_n}

\usepackage{natbib}
\numberwithin{equation}{section}
\theoremstyle{plain}
\newtheorem{thm}{Theorem}

\newtheorem{proposition}{Proposition}

\newtheorem{lemma}{Lemma}
\newtheorem{assumption}{Assumption}

\theoremstyle{definition}
\newtheorem{definition}{Definition}

\theoremstyle{remark}
\newtheorem{remark}{Remark}

\newtheorem{example}{Example}

\newcommand{\AY}{\tilde{\mathbf{Y}}}
\newcommand{\Af}{\tilde{\mathbf{f}}}
\newcommand{\Ae}{\tilde{\boldsymbol{\epsilon}}}

\newcommand{\SDP}{\tau_{SDP}}
\newcommand{\DP}{\tau_{DP}}

\newcommand{\ED}{L}
\newcommand{\subgauss}{\sigma}
\newcommand{\tstar}{t_n^*}

\allowdisplaybreaks[4]

\title{Analyzing the discrepancy principle for kernelized spectral filter learning algorithms\thanks{The research of Martin Wahl has been partially funded by Deutsche Forschungsgemeinschaft (DFG) - SFB1294/1 - 318763901.}}
\author{Alain Celisse\thanks{CNRS-Université de Lille, Inria - Modal Project-team, Lille, France. E-Mail: alain.celisse@math.univ-lille1.fr} \qquad Martin Wahl\thanks{Humboldt-Universit\"{a}t zu Berlin, Germany. E-Mail:
martin.wahl@math.hu-berlin.de}}

\date{}
\begin{document}

\maketitle

\begin{abstract}
We investigate the construction of early stopping rules in the nonparametric regression problem where iterative learning algorithms are used and the optimal iteration number is unknown. More precisely, we study the discrepancy principle, as well as modifications based on smoothed residuals, for kernelized spectral filter learning algorithms including gradient descent. Our main theoretical bounds are oracle inequalities established for the empirical estimation error (fixed design), and for the prediction error (random design). From these finite-sample bounds it follows that the classical discrepancy principle is statistically adaptive for slow rates occurring in the hard learning scenario, while the smoothed discrepancy principles are adaptive over ranges of faster rates (resp. higher smoothness parameters). Our approach relies on deviation inequalities for the stopping rules in the fixed design setting, combined with change-of-norm arguments to deal with the random design setting.
\end{abstract}


\section{Introduction}

\subsection{State-of-the-art}

The present work addresses the problem of estimating a regression function in a nonparametric framework by means of iterative learning algorithms, which is an ubiquitous problem in the statistical and machine learning literature.     
Since it is out of the scope of the present introduction to review all of them, let us only mention a few contributions in machine learning such as the boosting strategies aiming at estimating a regression function from a set of weak learners by iteratively re-weighting them \citep{duffy2002boosting,BuhlmannYu2003boosting}, or the more recent use of deep neural networks \citep{AnthonyBartlett2009neural,goodfellow2016deep}, where the iterative stochastic gradient descent algorithm is extensively applied \citep{jastrzebski2018relation,li2018learning}.
Nonparametric regression is the topic of several monographs such as \citep[][]{MR1920390}, \citep[][]{MR2724359}, or the more recent \citep[][]{GineNickl2016mathematical} that provides a detailed account of classical techniques for the theoretical analysis of nonparametric models. 

    \medskip
    
Our theoretical analysis applies to learning algorithms embedded in a reproducing kernel Hilbert space (RKHS) associated with a reproducing kernel \citep{aronszajn1950theory}.
Their use in machine learning traces back to \citep{10.1145/130385.130401} (SVMs), and there is now an extensive literature on this topic. Among others, \citep{CuckerSmale2002mathematical} and \citep{MR2450103} describe the mathematical foundation of learning with reproducing kernels. 
\cite{caponnetto2007optimal} derive optimal convergence rates for the prediction error of the kernelized Tykhonov algorithm, while \cite{Jacot2018neural} connect the properties of a deep neural network during the training to a particular reproducing kernel called the neural tangent kernel (see \cite{10.5555/559923} and \cite{shawe2004kernel} for more applications of reproducing kernels).
    
    \medskip
    
    The class of spectral filter algorithms \citep{MR2297015,MR3833647,lin2018optimal} that is under consideration in the present work can be seen as a subset of the broader family of iterative algorithms. 
    Iterative algorithms become ubiquitous in situations where some regularization is needed \citep{MR3190843}, or if no closed-form expressions are available for the estimator of interest. This typically arises for most of M-estimators \citep{vanderVaart1996weak} for which optimization algorithms such as gradient descent, coordinate descent, or Newton's method are used among others \citep{boyd2004convex}.
    In practice using such iterative algorithms requires the knowledge of the best iteration number at which one should interrupt the process. This optimal iteration number actually reaches a crucial trade-off between the statistical precision output after some iterations and the computational resources induced by them. For instance, interrupting the process too early provides a poor statistical precision, whereas waiting for more iterations induces a higher computational price (and typically even worse performances) \citep[][Fig.~1]{MR3190843}. 
    
    The main focus here is given to the so-called \emph{early stopping rules}, which are data-driven estimators of this usually unknown best iteration number. Designing such rules is all the more important as they are designed to output an efficient estimator while saving the computational resources. For instance, unlike Lepskii's method and similar model selection procedures \citep{MR2657949,blanchard2019lepskii}, early stopping rules avoid all pairwise comparisons between models, which turns out to be highly time consuming.
    The design and study of early stopping rules have received a lot of attention which can be traced back to the empirical work of \cite{prechelt1998early} in the context of neural networks.
    A first line of research leads to \emph{deterministic} stopping rules that only depend on the data through the sample size $n$ and some smoothness parameters (see \cite{zhang2005boosting} for the boosting, followed by \cite{yao2007early} and \cite{lin2018optimal} with spectral filter algorithms).
    A second strategy has been initiated by \cite{MR3190843} and then by \cite{wei2019early}, which mainly relies on upper bounding with high probability the estimation error by means of the Rademacher complexity. The resulting stopping rules enjoy good convergence rates from an asymptotic perspective, but only depend on the data through the points of the design which limits their practical application.
    More recently, a new promising idea has been investigated by \cite{MR3859376,MR3829522} in the context of the Gaussian sequence model where a stopping rule is suggested and analyzed which relies on the one hand on the discrepancy principle, and on the other hand on the estimation (rather than an upper bound) of the approximation error. 
    While the resulting stopping rules still have some drawbacks compared to classical model selection procedures (such as Lepskii's method \citep{BMM19}) in terms of statistical optimality, they achieve good oracle properties in a computationally efficient way.

    \subsection{Contributions}
From a practical perspective, our main contribution is the description of \emph{data-driven} early stopping rules based on the discrepancy principle.
        Unlike previous approaches, the dependence of our stopping rules with respect to the data is not limited to the sample size \citep{yao2007early} nor to the design points \citep{MR3190843,wei2019early}.
        By contrast, the present work rather extends the results of \cite{MR3859376,MR3829522} for inverse problems in the Gaussian sequence setting to the context of reproducing kernels and kernelized spectral filter estimators.

From a theoretical perspective our contributions are two-fold. %
On the one hand, we derive the first non-asymptotic theoretical analysis of these stopping rules applied to spectral filter algorithms combined with reproducing kernels. 
Firstly, this analysis relies on several new concentration inequalities in the fixed-design setting which lead to (non-asymptotic) oracle inequalities for two stopping rules based on the discrepancy principle. Secondly, we use a new change-of-norm argument 
which allow us to transfer these oracle inequalities to the random design setting.
On the other hand, these finite-sample bounds from the random design case lead to establish that: $(i)$ the classical discrepancy principle is statistically adaptive for slow rates occurring in the hard learning scenario (called outer case), and $(ii)$ the smoothing-based discrepancy principles are adaptive over ranges of higher smoothness parameters (called inner case).

\subsection{Outline}
The remainder of the paper is organized as follows.
Next Section \ref{sec.DP.ESR} introduces the main notions used along the papers. It starts by describing the statistical model, the spectral filter learning algorithms, and reviewing previous works on optimal rates in the context of the present paper. 
The early stopping rule based on the discrepancy principle (DP) is then introduced and motivated in Section~\ref{sec.motivation.DP}. 

Our first main theoretical results are discussed in Section~\ref{sec.DP.Oracle.inequality.fixed.design} which focuses on the DP stopping rule in the fixed-design setting. In particular, the main ingredients of the derivation are detailed in Section~\ref{sec.preliminary.results}.
The improved early stopping rule based on the smoothing of the residuals is then introduced and analyzed in Section~\ref{sec.smoothing.residuals} for the fixed-design case, while the random design framework is addressed in Section~\ref{SecRD}.
A short illustration of the behaviour of the different stopping rules is provided in Section~\ref{sec.simulations.experiments} by means of empirical simulations from synthetic data.

Finally, we provide proofs based on a unified analysis for both early stopping rules in Section~\ref{secProofFixedDesign} in the fixed-design, while proofs for the random design case are detailed in Section~\ref{SecProofRandomDesign}.
The appendix collects some background material.

\section{Spectral filters and discrepancy principle}
\label{sec.DP.ESR}

\subsection{Regression model and reproducing kernel}\label{sec.Model}
Let $(X,Y)$ be a pair of random variables satisfying the regression equation
\begin{equation}\label{def.model}
Y=f(X)+\epsilon,
\end{equation}
where $X$ is a random variable taking values in $\mathcal{X}\subseteq \mathbb{R}^d$, $f:\mathcal{X}\rightarrow \mathbb{R}$ is an unknown regression function, and $\epsilon$ is a real-valued random variable such that $\mathbb{E}(\epsilon|X)=0$ and $\mathbb{E}(\epsilon^2|X)=\sigma^2$, with \emph{$\sigma^2>0$ assumed to be known} as in \citep{MR3190843} for instance.
Additionally, we suppose that $\epsilon$ is sub-Gaussian conditional on $X$, cf. \citep{MR3837109}.
\begin{assumption}\label{AssSubGaussError}
There is a constant $A\geq 1$ such that
\begin{align}
    \forall q\geq 1,\qquad q^{-1/2}(\mathbb{E}( |\epsilon|^q|X))^{1/q}\leq A\sigma. \tag{{\bf SubGN}}\label{SubGN}
\end{align}
\end{assumption}
Let $k(\cdot,\cdot)$ be a continuous and positive kernel on $\mathcal{X}\subseteq \mathbb{R}^d$ and $\mathcal{H}$ be the reproducing kernel Hilbert space of $k$. Let also $\langle \cdot,\cdot\rangle_{\mathcal{H}}$ and $\|\cdot\|_{\mathcal{H}}$ denote the inner product in $\mathcal{H}$ and its corresponding norm. In the following, $\rho$ represents the distribution function of $X$, and $L_\rho:L^2(\rho)\rightarrow L^2(\rho), L_\rho g(x)=\int k(x,y)g(y)\,d\rho(y)$ states for the integral operator associated with $k$ and $\rho$. Let $\langle \cdot,\cdot\rangle_{\rho}$ and $\|\cdot\|_\rho$ denote the inner product in $L^2(\rho)$ and its corresponding norm. We also define the $\mathcal{H}$ valued random variable $k_X=k(X,\cdot)$ with values in $\mathcal{H}$ for which we make the following assumption. 
\begin{assumption}\label{AssBoundedk} There is a constant $M>0$ such that 
\begin{align}
 \normH{k_X} \leq M   \quad a.s. \label{BdK} \tag{{\bf BdK}}
\end{align} 
\end{assumption}
For instance, \eqref{BdK} holds true if $\sup_{x\in\X}k(x,x)\leq M^2$ (from the reproducing property). This arises with any kernel and a bounded domain $\X$, or with a bounded kernel and $\X$ unbounded (Gaussian kernel). 

In particular, we can define the covariance operator
\begin{align*}
    \Sigma=\mathbb{E} \croch{k_X\otimes k_X},
\end{align*}
where $a\otimes b \in\mathcal{L}(\H)$ denotes the tensor product between elements $a,b\in \H$ such that $(a\otimes b) u = a \scalH{b,u}$, for every $u\in \H$.
Under Assumption \eqref{BdK} we know that both, $L_\rho$ and $\Sigma$ are positive self-adjoint trace-class operators. Moreover, both operators $L_\rho$ and $\Sigma$ are intimately related, which can be seen by introducing the inclusion operator $S_\rho:\mathcal{H}\rightarrow L^2(\rho)$, mapping $h\in\mathcal{H}$ to its equivalence class in $L^2(\rho)$ ($S_\rho$ is well-defined, because under Assumption \eqref{BdK} every $h\in\mathcal{H}$ is bounded).
Then it is well-known that 
\begin{align*}
    S_\rho S_\rho^* = L_\rho \in \mathcal{L}(L^2(\rho)),\qquad S_\rho^*S_\rho=\Sigma \in \mathcal{L}(\mathcal{H}),
\end{align*}
where $S_\rho^*$ is the adjoint operator of $S_\rho$.
For these and more information on the learning with kernels setting see e.g. \citep{CuckerSmale2002mathematical} and \citep{MR2249842}.
By the spectral theorem, there exists a sequence $\lambda_1\geq \lambda_2\geq\dots> 0$ of positive eigenvalues (which is either finite or converges to zero), together with an orthonormal system $u_1,u_2,\dots$ of eigenvectors of the range of $L_\rho$ such that $\Sigma =\sum_{j\ge 1}\lambda_i u_j\otimes u_j$.

We will assume that $f$ satisfies a polynomial source condition (see Chap.~4 in \cite{MR3114700}) that is, 
\begin{assumption}
For some $r\geq 0$ and $R>0$, we have
\begin{align*}
f = L_\rho^r g ,\quad \mbox{with } g\in L^2(\rho)\ \mbox{and } \norm{g}_\rho\leq R. \tag{{\bf SC}(r,R)} \label{Assume.SC}
\end{align*}
\end{assumption}
Note that such source conditions are often written as $\|L_\rho^{-r}f\|_\rho\leq R$; see e.g. \cite{smale2007learning}.
\begin{remark}[Inner and Outer cases]
On the one hand, if $r\geq 1/2$, then
\begin{align}\label{EqInnerCaseRepr}
    f = L_\rho^r g = S_\rho \Sigma^{r-1/2}  \Sigma^{-1/2}S_\rho^*  g = S_\rho f_\H ,
\end{align}
where $f_\H = \Sigma^{r-1/2} (\Sigma^{-1/2}S_\rho^* g) \in\H$. This means that $f$ (resp. its equivalence class) can be represented (through the inclusion operator $S_\rho$) as a function in $\H$. This case is then called \emph{the inner case}.
Let us mention that one also recovers an alternative formulation of the source condition when $r\geq 1/2$ that is,
\begin{align*}
    f_\H=\Sigma^sh,\quad \mbox{where } h\in\mathcal{H}\ \mbox{and } \norm{h}_{\mathcal{H}}\leq R, 
\end{align*}
with $s=r-1/2\geq 0$ and $ h = \Sigma^{-1/2}S_\rho^* g \in\H$, where $ \norm{h}_\H = \|\Sigma^{-1/2}S_\rho^*  g\|_\H = \norm{g}_\rho \leq R$. These results can be found in \cite{CuckerSmale2002mathematical}, where it is shown how to characterize $\mathcal{H}$ through the eigenvalues of $L_\rho$.

On the other hand, if $r<1/2$, then $f$ can not be represented as a function in $\H$ in general, which justifies referring to this situation as \emph{the outer case}. 

In what follows, the outer and inner cases are respectively considered in Section~\ref{sec.outer.case} and  Section~\ref{sec.SDP.inner.case}. 
\end{remark}

      \medskip
 
We suppose that we observe $n$ independent copies $(X_1,Y_1),\dots,(X_n,Y_n)$ of $(X,Y)$. Let $K_n\in \mathbb{R}^{n\times n}$ be the kernel matrix defined by $(K_n)_{ij}=k(X_i,X_j)/n$ and $\Sigma_n$ be the empirical covariance operator defined by
\begin{align*}
    \Sigma_n=\frac{1}{n}\sum_{i=1}^nk_{X_i}\otimes k_{X_i}.
\end{align*}
Both operators $K_n$ and $\Sigma_n$ are strongly related, as can be seen by introducing the sampling operator $S_n$ defined by $S_n:\mathcal{H}\rightarrow \mathbb{R}^n,h\mapsto (h(X_i))_{i=1}^n$ and its adjoint operator $S_n^*$, where $\mathbb{R}^n$ is endowed with the empirical inner product $\langle\cdot,\cdot\rangle_n$ and its corresponding \emph{empirical norm} $\|\cdot\|_n$ such that $\langle a,b\rangle_n=(1/n) \sum_{i=1}^na_ib_i$ and $\norm{a}_n=\sqrt{\langle a,b \rangle_n}$ for every $a,b\in\R^n$. Then we have
\begin{align*}
    S_nS_n^*=K_n,\qquad S_n^*S_n=\Sigma_n.
\end{align*}
By the spectral theorem, there exists a sequence $\hat\lambda_1\geq \hat\lambda_2\geq\dots\geq \hat \lambda_n$ of non-negative eigenvalues, together with an orthonormal system $\hat u_1,\hat u_2,\dots,\hat u_n$ in $\mathcal{H}$ and an orthonormal basis $\hat v_1,\dots, \hat v_n$ of $(\mathbb{R}^n,\langle\cdot,\cdot\rangle_n)$ such that
\begin{align}\label{EqSVDSamplingOp}
    S_n=\sum_{j=1}^n\hat\lambda_j^{1/2}\hat v_j\otimes \hat u_j.
\end{align}
In particular, we have $\Sigma_n =\sum_{j=1}^n\hat\lambda_i \hat u_j\otimes \hat u_j$ and $K_n=\sum_{j=1}^n \hat\lambda_j \hat v_j\hat v_j^T$. We write $\mathbf{Y}=(Y_1,\dots,Y_n)^T$ and $\boldsymbol{\epsilon}=(\epsilon_1,\dots,\epsilon_n)^T$. Moreover, for a function $f:\mathcal{X}\rightarrow \mathbb{R}$, we abbreviate $\mathbf{f}=(f(X_1),\dots,f(X_n))^T$. In particular, for $h\in\mathcal{H}$, we have $\mathbf{h}=S_nh\in\mathbb{R}^n$.

\subsection{Spectral filter learning algorithms}\label{SecSpectralFilter}
Let us consider the problem of estimating $f$ by means of spectral filter learning algorithms, see e.g. \cite{MR2297015}, \citep{MR3114700}, \citep{MR3833647} and \citep{lin2018optimal}. 
For a function $g:(0,M^2]\times [0,\infty)\rightarrow \mathbb{R}$, let us write $g_t(\lambda)=g(\lambda,t)$. 
\begin{definition}[Regularizer] \label{Def}
A function $g:(0,M^2]\times [0,\infty)\rightarrow \mathbb{R}$ is called a regularizer if $(\lambda,t)\mapsto\lambda g_t(\lambda)$ is non-decreasing in $t$ and $\lambda$, continuous in $t$, with $g_0(\lambda)=0$ and $\lim_{t\to +\infty}\lambda g_t(\lambda)=1$, and if there is a constant $B>0$ such that
\begin{align}
(i) & \mbox{ For all $(\lambda,t)\in(0,M^2]\times [0,\infty)$, we have } 0\leq \lambda g_t(\lambda)\leq 1 , \tag{{\bf BdF}}\label{BdF}\\
(ii) & \mbox{ For all $(\lambda,t)\in(0,M^2]\times [0,\infty)$, we have } g_t(\lambda)\leq Bt.\tag{{\bf LFU}}\label{LFU}
\end{align}
\end{definition}
The above definition is slightly stronger than e.g. Definition 1 in \cite{MR2297015} because we assume the continuity in $t$. This excludes e.g. the spectral cut-off algorithm (corresponding to the choice $g_t(\lambda)=\1_{\paren{\lambda t\geq 1}}/\lambda$) from the present study. 
\begin{definition}[Spectral filter estimators]\label{DefFilter}
For a given regularizer $g:(0,M^2]\times [0,\infty)\rightarrow \mathbb{R}$, a \emph{spectral filter estimator} is an estimator given by
\begin{align*}
    \hat f^{(t)} = g_t(\Sigma_n)S_n^{*}\mathbf{Y},\quad t\geq 0.
\end{align*}
\end{definition}
By \eqref{BdK}, we have that $\max(\lambda_1,\hat\lambda_1)\leq M^2$ almost surely. This implies that the estimators $\hat f^{(t)}$ are indeed well-defined. The following examples provide several choices of spectral filter algorithms and  regularizers.
\begin{example} The choice $g_t(\lambda)=(\lambda+t^{-1})^{-1}$ corresponds to Tikhonov regularization and Definition~\ref{Def} holds with $B=1$.
\end{example}
\begin{example} Gradient descent with constant step size $\eta\in(0,1/M^2)$ (also called Landweber) corresponds to the sequence of iterations
\begin{align*}
 \hat f^{(0)}=0,\quad \hat f^{(t)}=\hat f^{(t-1)}+\eta S_n^*(\mathbf{Y}-S_n\hat f^{(t-1)}),\quad t=1,2,\dots.
\end{align*}
It has the closed-form expression $\hat f^{(t)}=g_t(\Sigma_n)S_n^*\mathbf{Y}$ with $g_t(\lambda)=\lambda^{-1}(1-(1-\eta \lambda)^t)$. Interpolating, we may consider $g_t(\lambda)=\lambda^{-1}(1-(1-\eta \lambda)^t)$ for $t\geq 1$, and $g_t(\lambda)=\eta t$ for $t< 1$. In this case, Definition \ref{Def} holds with $B=\eta$.
\end{example}
\begin{example}
The choice $g_t(\lambda)=\lambda^{-1}(1-e^{-t\lambda })$ corresponds to Showalter's method. In this case, Definition \ref{Def} holds with $B=1$.
\end{example}

At some places, an additional assumption will turn to be useful in the analysis of spectral filter algorithms. It lower bounds the regularizer.
\begin{assumption} There is a constant $b>0$ such that
\begin{align}
& \mbox{ for all $(\lambda,t)\in(0,M^2]\times [0,\infty)$, we have } \lambda g_t(\lambda)\geq b(1\wedge \lambda t) . \tag{{\bf LFL}} \label{LFL}
\end{align}
\end{assumption}
For instance, this latter assumption holds true with Tikhonov regularization, gradient descent and Showalter's method with $b=1/2$.

Finally, when dealing with rates of convergence we will also need the following assumption on the qualification error.
\begin{assumption}\label{QualError}
There are constants $q,Q> 0$ such that
\begin{align}
    \mbox{ for all $(\lambda,t)\in(0,M^2]\times [0,\infty)$, we have } |r_t(\lambda)|\leq Q (\lambda t)^{-q}  \label{Qualif}\tag{{\bf QuErr}} ,
\end{align} 
with $r_t(\lambda)=1-g_t(\lambda)\lambda$.
\end{assumption}
\begin{remark}
Combining \eqref{Qualif} with \eqref{BdF}, we have $r_t(\lambda)\leq 1\wedge Q(t\lambda)^{-q}$ and thus also $r_t(\lambda)\leq 1\wedge Q(t\lambda)^{-p}$ for each $p\leq q$, provided that $Q\geq 1$.
\end{remark}
It is well-known that Tikhonov regularization and gradient descent satisfy \eqref{Qualif} with respectively $q=1$ and $q$ arbitrary; see e.g. \cite{MR3833647} for more discussion.

\medskip

Let us also introduce the $g$-effective dimension, which generalizes the classical notion of effective dimension \citep{zhang2003effective} to the case where $g$ is not limited to the Tikhonov regularization. 
\begin{definition}[$g$-Effective dimension]\label{gDffectiveDim}
For every $t\geq 0$ and any regularizer $g$, the (population) $g$-effective dimension is defined by $\geffdim(t)=\operatorname{tr}(\Sigma g_t(\Sigma))$, while the empirical effective dimension is $\geffdim_n(t)=\operatorname{tr}(\Sigma_n g_t(\Sigma_n))$.
\end{definition}
With Tikhonov regularization, that is $g_t(\lambda) = (\lambda +1/t)^{-1}$, both the population and empirical $g$-effective dimension simply reduce to the usual population and empirical effective dimensions respectively given by $\effdim(t)=\operatorname{tr}(\Sigma(\Sigma+1/t)^{-1})$ and $\effdim_n(t)=\operatorname{tr}(\Sigma_n(\Sigma_n+1/t)^{-1})$. Note that most cited references consider the parameterization $\eta=t^{-1}$, i.e. they write $g_{\eta}(\lambda)$ and $\mathcal{N}(\eta)$ instead of $g_{t}(\lambda)$ and $\mathcal{N}(t)$ in the present paper.
Interestingly, it turns out that the effective and $g$-effective dimensions are closely related up to multiplicative constants as established by the next result.
\begin{lemma}\label{LemEffDimGenDim} 
Let $g$ be a regularizer satisfying \eqref{LFL}. Then for each $t\geq 0$,
\begin{align*}
     b\mathcal{N}_n(t)\leq \geffdim_n(t)\leq2(B\vee 1)\mathcal{N}_n(t).
\end{align*}
\end{lemma}
\begin{proof}[Proof of Lemma~\ref{LemEffDimGenDim}]
By \eqref{BdF} and \eqref{LFU} we have
\[
\geffdim_n(t)\leq (B\vee 1)\sum_{j=1}^n
1\wedge \hat\lambda_jt\leq 2(B\vee 1)\sum_{j=1}^n \frac{\hat\lambda_jt}{\hat\lambda_jt+1} = 2(B\vee 1)\mathcal{N}_n(t),
\]
which gives the upper bound. The lower bound follows from \eqref{LFL} and the fact that $\lambda t/(\lambda t+1)\leq 1\wedge \lambda t$.
\end{proof}

\subsection{Convergence rates in related works}\label{sec.optimal.rates}

The use of kernel-based spectral regularization in random regression problems (also known as ``learning from examples'') has been extensively studied in the literature; see e.g. \cite{MR2186447,smale2007learning,caponnetto2007optimal} for Tikhonov regularization, \cite{yao2007early,MR3564933} for gradient descent methods and \cite{MR2297015,MR3833647,lin2018optimal} for general spectral regularization schemes. Existing bounds are mostly established for the $L^2(\rho)$-error and the $\mathcal{H}$-error under \eqref{Assume.SC} and a polynomial upper bound on the eigenvalues of $L_\rho$. They are usually used to construct deterministic early stopping rules (depending on the smoothness $r$ and the eigenvalue decay); see e.g. \citep{yao2007early} for gradient descent, \citep{MR3564933} for conjugate gradient descent and \citep{pillaud2018statistical} for stochastic gradient descent. 

\medskip

Surprisingly, while the inner case $r\geq 1/2$ is now well understood \cite{MR3833647,lin2018optimal}, there remain some unsolved issues related to the outer case. The main difficulties arise in case of the so-called hard learning problems for which the optimal rates are achieved for very small regularization parameters (resp. a very large number of iterations, considerably exceeding the number of observations). In this direction, some improvements have been established e.g. in \cite{FS19,pillaud2018statistical}, based on more precise concentration inequalities for the eigenvalues of the kernel matrix (see Theorem~\ref{ThmDPHardProblemsExt}).

\medskip

Progress has also been made in the study of data-driven regularization parameter selection rules. Lepskii's balancing principle has been extended to the learning framework in  \citep{MR2657949,MR3114700,BMM19}. While the estimators from \citep{MR2657949,MR3114700} are only adaptive with respect to the smoothness $r$, the estimator from \citep{BMM19} achieves faster rates by also being adaptive with respect to the eigenvalue decay of the kernel integral operator. In slightly different directions, \citep{PG19} studies the Goldenshluger-Lepskii method in a reproducing kernel framework, and \citep{MR3431429} studies model selection for principal component regression in a functional regression model. While all these methods share good oracle properties (and thus minimax adaption over suitable smoothness classes), they all put no attention on computational issues. In fact, they require that all estimators up to some threshold have to be computed before a parameter with close-to-optimal performance is chosen.

\medskip

In contrast, the question of data-driven early stopping rules remains widely open. \citep{MR3190843} suggest an early stopping rule for gradient descent that is adaptive to the decay rate of the eigenvalues but not to the smoothness $r$ (assumed to be $r=1/2$). They study the solution of a fixed-point equation corresponding to a bias-variance trade-off of the empirical norm and show that this rule leads to optimal rates for the prediction error. These results have been extended in \citep{wei2019early} to the $L^2$-boosting based on different loss functions. Our goal is to develop \emph{data-driven} stopping rules based on the discrepancy principle which are \emph{statistically adaptive} with respect to both the smoothness parameter $r$ and the eigenvalue decay.

\subsection{Early stopping and discrepancy principle: Motivation}\label{sec.motivation.DP}
As explained in the introduction, our goal is to make use of the discrepancy principle to find a value $t$ having small excess risk. 
The discrepancy principle (DP) has been extensively studied in the context of inverse problems with deterministic noise, where it is also called Morozov’s discrepancy principle, see e.g. \cite{MR1408680}. Using $\hat{\mathbf{f}}^{(t)}=S_ng_t(\Sigma_n)S_n^*\mathbf{Y}=K_ng_t(K_n)\mathbf{Y}$ from Definition~\ref{DefFilter} with regularizer $g$ from Definition~\ref{Def}, it is based on a comparison of the empirical risk $\|\mathbf{Y}-\hat{\mathbf{f}}^{(t)}\|_n^2$ (also squared discrepancy, squared residual) with the noise level $\mathbf{E}_\epsilon\|\boldsymbol{\epsilon}\|_n^2=\sigma^2$, where $\mathbf{E}_{\boldsymbol{\epsilon}}(\cdot)=\mathbb{E}(\cdot | X_1,\dots,X_n)$ denotes the expectation with respect to $(X_1,Y_1),\dots,(X_n,Y_n)$ conditional on the design $X_1,\dots,X_n$). It then advocates taking a value $t$ for which both quantities are of comparable size. 

The discrepancy principle can also be motivated by considering the expected empirical risk $\mathbf{E}_\epsilon{ \|\mathbf{Y}-\hat{\mathbf{f}}^{(t)} \|_n^2} = \mathbf{E}_\epsilon\|r_t(K_n)\mathbf{Y}\|_n^2$. 
The first step consists in  noticing that we have the following bias-variance decomposition of the excess risk
\begin{align*}
\mathbf{E}_\epsilon\|\mathbf{f}-\hat{\mathbf{f}}^{(t)}\|_n^2&=\|r_t(K_n)\mathbf{f}\|_n^2 + \frac{\sigma^2}{n}\operatorname{tr}(g_t^2(K_n)K_n^2).
\end{align*}
Using \eqref{BdF} this identity implies
\begin{align}\label{EqBiasVarProxyNonSmooth}
\mathbf{E}_\epsilon\|\mathbf{f}-\hat{\mathbf{f}}^{(t)}\|_n^2&\leq \|r_t(K_n)\mathbf{f}\|_n^2+\frac{\sigma^2}{n}\mathcal{N}_n^g(t)
\end{align}
with $g$-effective dimension $\mathcal{N}_n^g(t)$.

The second step exploits Lemma \ref{EqBasicIneqDP} below, which reveals a close relation to \eqref{EqBiasVarProxyNonSmooth} by showing that
\begin{align}
 &\|r_t(K_n)\mathbf{f}\|_n^2-2\frac{\sigma^2}{n}\mathcal{N}_n^g(t)\nonumber\\
 &\leq \mathbf{E}_\epsilon\|r_t(K_n)\mathbf{Y}\|_n^2-\sigma^2\leq \|r_t(K_n)\mathbf{f}\|_n^2-\frac{\sigma^2}{n}\mathcal{N}_n^g(t).\label{EqBasicIneq}
\end{align}
In particular, by defining $t_0\geq 0$ such that 
\begin{align}\label{expected.ESR.justification}
    t_0 = \inf\acc{ t\geq 0 \mid \mathbf{E}_\epsilon\|r_t(K_n)\mathbf{Y}\|_n^2=\sigma^2 },
\end{align} it follows from \eqref{EqBiasVarProxyNonSmooth} and \eqref{EqBasicIneq} that
\begin{align}\label{illustration.oracle.ineq}
\mathbf{E}_\epsilon\|{\mathbf{f}}-\hat{\mathbf{f}}^{(t_0)}\|_n^2&\leq 3\min_{t\geq 0}\Big\{\|r_t(K_n)\mathbf{f}\|_n^2+\frac{\sigma^2}{n}\mathcal{N}_n^g(t)\Big\},
\end{align} 
where we also used that $\|r_t(K_n)\mathbf{f}\|_n^2$ and $\geffdim_n(t)$ are respectively non-increasing and non-decreasing with respect to $t\geq 0$ (see Figure~\ref{fig.Illustration.ESR}). Let us mention that Ineq. \eqref{illustration.oracle.ineq} is called an \emph{oracle-type} inequality in what follows.
Similarly, we also have the next lower bound
\begin{align*}
&\mathbf{E}_\epsilon\|{\mathbf{f}}-\hat{\mathbf{f}}^{(t_0)}\|_n^2\geq \|r_{t_0}(K_n)\mathbf{f}\|_n^2\geq
\frac{1}{2}\min_{t\geq 0}\Big\{\|r_t(K_n)\mathbf{f}\|_n^2+\frac{\sigma^2}{n}\mathcal{N}_n^g(t)\Big\}.
\end{align*}  

The third step relies on the important consequence that these upper and lower bounds indicate that $t_0$ defined by Eq.~\eqref{expected.ESR.justification} is an optimal choice (up to the proxy variance term and constants) one can make for \emph{stopping early} the estimation process.
In particular, this justifies the introduction of the following early stopping rule based on the discrepancy principle (DP), which should be seen as the empirical counterpart of Eq.~\eqref{expected.ESR.justification}. 
\begin{definition}[DP stopping rule]
For any estimator $\hat f^{(t)}=g_t(\Sigma_n)S_n^{*}\mathbf{Y}$ given by Definition~\ref{DefFilter}, the DP-based stopping rule $\tauDP$ is defined by
\begin{equation}\label{EqDP}
\DP=\DP(\mathbf{Y},\sigma^2,T)=\inf\{t\geq 0:\|\mathbf{Y}-S_n\hat{f}^{(t)}\|_n^2\leq \sigma^2\}\wedge T,
\end{equation}
with the ``emergency stop'' $T\in[0,\infty]$.
\end{definition}
In the context of statistical inverse problems, see also \citet{blanchard2012discrepancy} with the conjugate gradient and \citet{MR3859376} with the spectral cut-off.
The above definition depends on the knowledge of two parameters, the emergency stop $T$ and the true noise level $\sigma^2$. In principle, it is also possible to use an estimator of $\sigma^2$. But such extensions are not pursued here for avoiding further technicalities.
From Definition~\ref{Def}, the fact that $\lim_{t\to +\infty} \lambda g_t(\lambda) = 1$ implies that the empirical risk $\|\mathbf{Y}-S_n\hat{f}^{(t)}\|_n^2=\|r_t(K_n)\mathbf{Y}\|_n^2$ converges to zero as $t\to +\infty$. This entails that the choice $T=\infty$ is admissible as well since we will interrupt the iterations after a finite number of them.

\subsection{Further notation}\label{sec.further.notations}
The abbreviation $\mathbf{E}_{\boldsymbol{\epsilon}}(\cdot)=\mathbb{E}(\cdot | X_1,\dots,X_n)$ denotes the expectation with respect to $(X_1,Y_1),\dots,(X_n,Y_n)$ conditional on the design $X_1,\dots,X_n$. This means a slight abuse of notation because in the present context, the distribution of $\epsilon_i$ is allowed to depend on $X_i$. We also write $\mathbf{P}_{\boldsymbol{\epsilon}}(\cdot)=\mathbb{P}(\cdot | X_1,\dots,X_n)$.

Given a bounded operator $A$ on $\mathcal{H}$ or a matrix $A\in\mathbb{R}^{n\times n}$, we write $\|A\|_{\operatorname{op}}$ for the operator norm. Given a Hilbert-Schmidt operator $A$ on $\mathcal{H}$ or a matrix $A\in\mathbb{R}^{n\times n}$, we write $\|A\|_{\operatorname{HS}}$ for the Hilbert-Schmidt or Frobenius norm. Given a trace class operator $A$ on $\mathcal{H}$ or a matrix $A\in\mathbb{R}^{n\times n}$, we denote the trace of $A$ by $\operatorname{tr}(A)$. 

Throughout the paper, we use the letters $c,C$ for constants that may change from line to line. They are allowed to depend on $A,B,b,Q,R,M$ and $r$. Apart from these dependencies, the constants are absolute and can be made explicit by considering the proofs. In Section \ref{SecRD} they are also allowed to depend on $L,\alpha$ (introduced therein) and $\sigma^2$. 
Finally for any $a,b\in\mathbb{R}$, we write $a\vee b=\max(a,b)$ and $a\wedge b=\min(a,b)$. For $a\geq 0$ we denote by $\lfloor a\rfloor$ the largest natural number that is smaller than or equal to $a$.
%

\section{DP and oracle inequality: Fixed-design}\label{sec.DP.Oracle.inequality.fixed.design}
The goal of this section is to assess the statistical performance of the final estimator $\hat f^{(\DP)}$, where $\DP$ is the early stopping rule defined by Eq.~\eqref{EqDP} and derived from the discrepancy principle (DP). 
We start by introducing new deviation inequalities for $\DP$ and for bias and variance terms (Propositions~\ref{ConcIneqVarDP} and~\ref{ConcIneqBiasDP}), leading then to oracle-type inequalities (Proposition~\ref{ThmCondDesignProxy} and Theorem~\ref{ThmCondDesignOracle}).

\subsection{Preliminary results}\label{sec.preliminary.results}
\subsubsection{Deviation inequalities for DP and main arguments}\label{sec.scetch.proof}
Our main deviation inequalities for the early stopping rules are developed in Section~\ref{secProofFixedDesign}. For the sake of simplifications, let us specialize them to the classical discrepancy principle $\DP$ with $T=\infty$.
For this, we abbreviate the squared bias and the \emph{proxy variance}  as
\begin{align} \label{bias.proxyvariance}
b_t^2=\|r_t(K_n)\mathbf{f}\|_n^2\quad \mbox{and} \quad v_t=\frac{\sigma^2}{n}\geffdim_n(t),
\end{align}
where $\geffdim_n(t)$ denotes the empirical $g$-effective dimension from Definition~\ref{gDffectiveDim}.
Moreover, we introduce the important balancing stopping rule
\begin{align*}
    \tstar=\inf\{t\geq 0:b_{t}^2= v_{t}\}.
\end{align*}
If such a $t$ does not exist, then we set $\tstar=\infty$. This can only happen if $v_t=0$ for every $t\geq 0$ (see the properties below), meaning that we can set $b^2_{\tstar},v_{\tstar}=0$ in this case. We start with a right-deviation inequality for $\DP$ that can be alternatively expressed in terms of the proxy variance $v_t$.
\begin{proposition}\label{ConcIneqVarDP}
If Assumption \eqref{SubGN} holds, then there is a constant $c>0$ depending only on $A$ such that for every $t>\tstar$,
\begin{align*}
    \mathbf{P}_\epsilon(\DP>t)\leq 2\exp\Big(-cn\Big(\frac{y}{\sigma^2}\wedge \frac{y^2}{\sigma^4}\Big)\Big),\qquad y=v_t -
   v_{\tstar}.
\end{align*}
In particular, for every $y>0$ we have
\begin{align*}
    \mathbf{P}_\epsilon(v_{\DP}>v_{\tstar}+y)\leq 2\exp\Big(-cn\Big(\frac{y}{\sigma^2}\wedge \frac{y^2}{\sigma^4}\Big)\Big).
\end{align*}
\end{proposition}
Both deviation inequalities are even equivalent if the proxy variance is strictly increasing. Proposition \ref{ConcIneqVarDP} is a simplified version of Proposition~\ref{ConcIneqVar} below. The proof can be based on exploring Figure~\ref{fig.Illustration.ESR} in combination with concentration inequalities for the empirical risk. Here is an outline of the argument.  Let us also mention that $t\mapsto b_t^2$ is continuous and non-increasing, while $t\mapsto v_t$ is continuous and non-decreasing. The definition of $\DP$ yields $\mathbf{P}_\epsilon(\DP>t)=\mathbf{P}_\epsilon(\|r_t(K_n)\mathbf{Y}\|_n^2>\sigma^2)$. Subtracting $\mathbf{E}_\epsilon\|r_t(K_n)\mathbf{Y}\|_n^2$ on both sides and invoking the upper bound in \eqref{EqBasicIneq}, we arrive at
\begin{align*}
    \mathbf{P}_\epsilon(v_{\DP}>y)\leq \mathbf{P}_\epsilon(\|r_t(K_n)\mathbf{Y}\|_n^2-\mathbf{E}_\epsilon\|r_t(K_n)\mathbf{Y}\|_n^2>v_t-b_t^2).
\end{align*}
By definition we have $v_t =  v_{\tstar}+y$. Moreover, from Figure~\ref{fig.Illustration.ESR} and the assumption on $y$, we get $b_t^2\leq b_{\tstar}^2= v_{\tstar}$. Hence, we conclude that 
\begin{align*}
    \mathbf{P}_\epsilon(v_{\DP}>y)&\leq \mathbf{P}_\epsilon(\|r_t(K_n)\mathbf{Y}\|_n^2-\mathbf{E}_\epsilon\|r_t(K_n)\mathbf{Y}\|_n^2>y),
\end{align*}
and Proposition \ref{ConcIneqVarDP} follows from the Hanson-Wright inequality (see Lemma \ref{LemResidualConc} below) and the fact that $b_t^2\leq v_t\leq \sigma^2$ since $t\geq \tstar$.

\begin{figure}[h!]
\hspace*{-1cm}
\begin{subfigure}{.5\textwidth}
\centering
    \includegraphics[width=\textwidth]{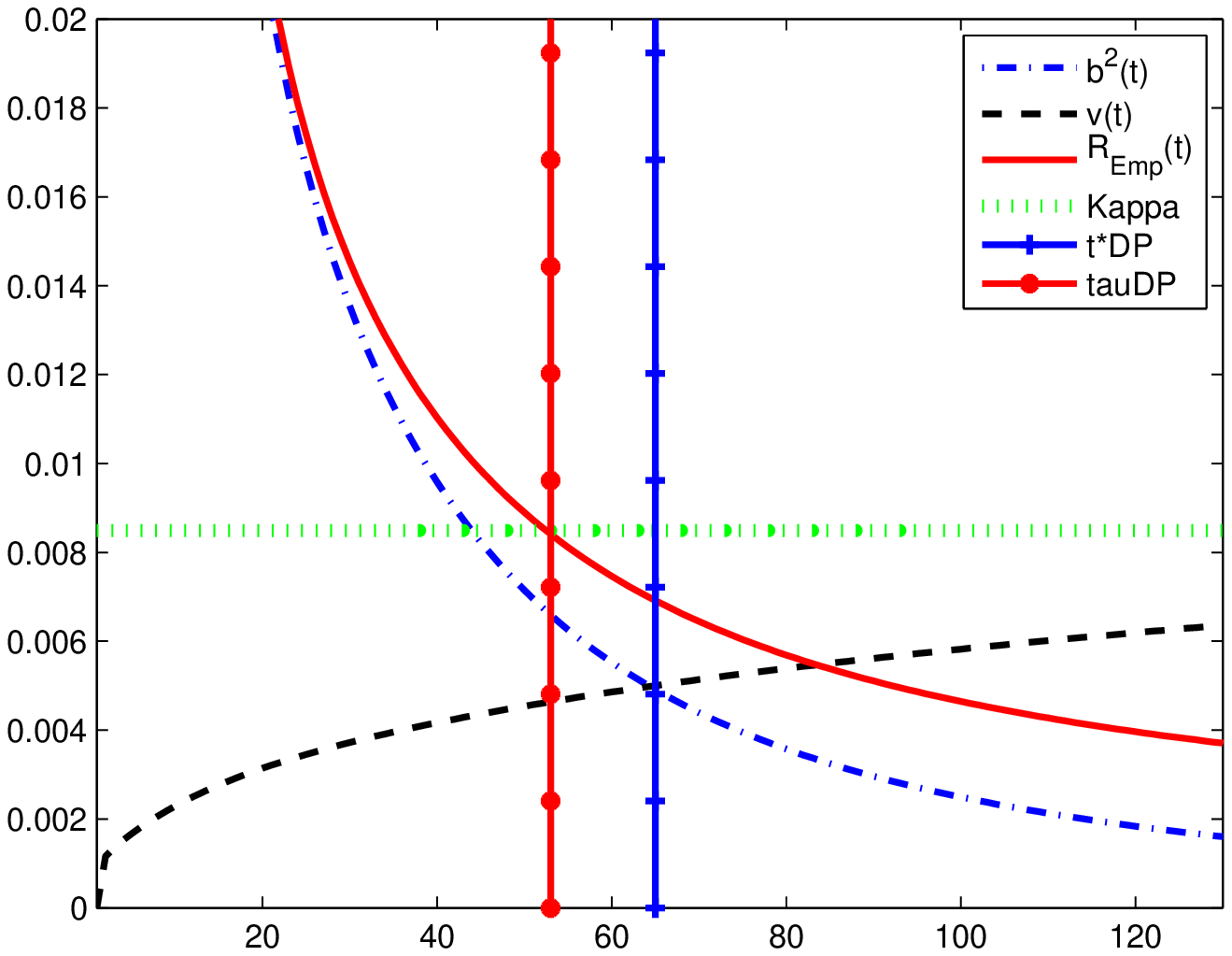}
\caption{The horizontal line defines $\kappa=\sigma^2$. The red plain decreasing curve crosses the horizontal line at $\DP$. The increasing curve crosses the blue dotted-dashed curve of the bias at $\tstar$.}
\label{fig.Illustration.ESR}
\end{subfigure}
\hspace*{1cm}
\begin{subfigure}{.5\textwidth}
\centering    
    \includegraphics[width=\textwidth]{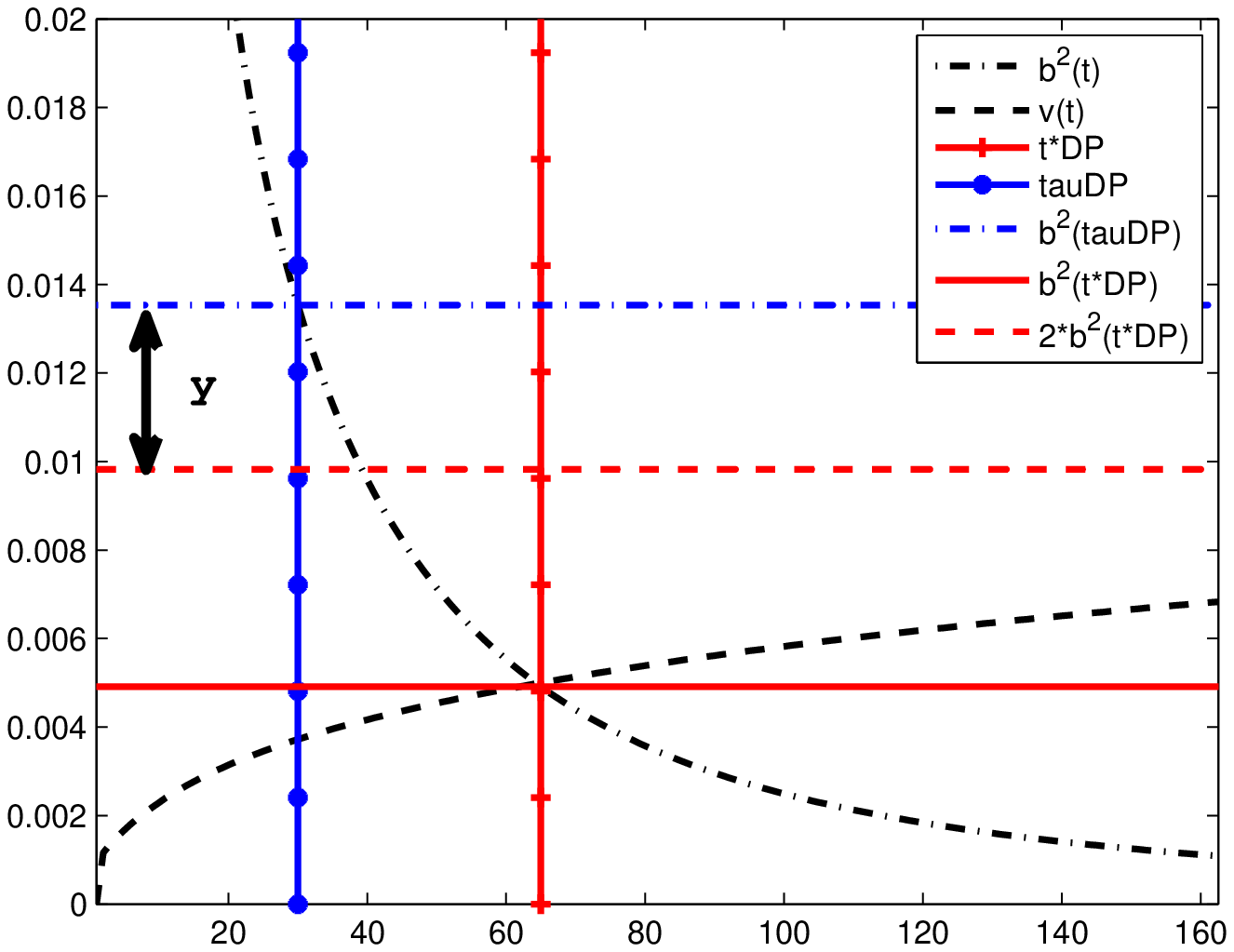}
\caption{Illustration of Proposition~\ref{ConcIneqBiasDP}. The red dashed horizontal line highlights the $2b^2(\tstar)$ threshold to which $b^2(\DP)$ is compared.\hfill~\\}
\label{fig.Deviation.Bias}
\end{subfigure}
\caption{Comparison of $\DP$ and the balancing stopping time $\tstar$.}
\end{figure}
Next, we present a left-deviation inequality for $\DP$ expressed in terms of the squared bias.
\begin{proposition}\label{ConcIneqBiasDP}
Suppose that Assumptions~\eqref{SubGN} and~\eqref{BdK} hold true. Then, for every $y>0$, we have
\begin{align*}
&\mathbf{P}_\epsilon(b_{\DP}^2> 2b_{\tstar}^2+y)\leq 2\exp\Big(-cn\Big(\frac{y}{\sigma^2}\wedge \frac{y^2}{\sigma^4}\Big)\Big),
\end{align*}
where $c>0$ is a constant depending only on $A$.
\end{proposition}
Proposition \ref{ConcIneqBiasDP} is a simplified version of Proposition~\ref{ConcIneqBias} below, and follows similarly as Proposition \ref{ConcIneqVarDP} by exploiting the lower bound in \eqref{EqBasicIneq} this time. As illustrated in Figure \ref{fig.Deviation.Bias}, let $t<\tstar$ be defined by $b_t^2=2b_{\tstar}^2+y$ (if such a $t$ does not exist, then the claim is trivial). Then the definition of $\DP$ yields $\mathbf{P}_\epsilon(b_{\DP}^2>b_{t}^2)\leq \mathbf{P}_\epsilon(\|r_t(K_n)\mathbf{Y}\|_n^2\leq \sigma^2)$. Subtracting $\mathbf{E}_\epsilon\|r_t(K_n)\mathbf{Y}\|_n^2$ on both sides and invoking the lower bound in \eqref{EqBasicIneq}, we arrive at
\begin{align*}
    \mathbf{P}_\epsilon(b_{\DP}^2>2b_{\tstar}^2+y)\leq \mathbf{P}_\epsilon(\|r_t(K_n)\mathbf{Y}\|_n^2-\mathbf{E}_\epsilon\|r_t(K_n)\mathbf{Y}\|_n^2\leq 2v_t-b_t^2).
\end{align*}
By definition we have $b_t^2=2b_{\tstar}^2+y$. Moreover, from Figure \ref{fig.Deviation.Bias} and the assumption on $y$, we get $v_t\leq v_{\tstar}= b_{\tstar}^2$. Hence, we conclude that 
\begin{align*}
    \mathbf{P}_\epsilon(b_{\DP}^2>2b_{\tstar}^2+y)&\leq \mathbf{P}_\epsilon(\|r_t(K_n)\mathbf{Y}\|_n^2-\mathbf{E}_\epsilon\|r_t(K_n)\mathbf{Y}\|_n^2\leq -y),
\end{align*}
and Proposition \ref{ConcIneqBiasDP} follows from the Hanson-Wright inequality (see Lemma \ref{LemResidualConc} below) and the fact that $b_t^2\leq 2v_{\tstar}+y\leq 2\sigma^2+y$.  

\subsubsection{Non-asymptotic performance quantification}

We are now in position to formulate our first upper bound for the estimation error in the empirical norm. It quantifies the statistical performance of the stopping rule based on the classical discrepancy principle (DP), namely $\tauDP$, in terms of an oracle-type inequality with high probability. 
\begin{proposition} \label{ThmCondDesignProxy} 
Suppose that Assumptions \eqref{SubGN} and \eqref{BdK} hold. Then the early stopping rule $\DP$ based on the standard discrepancy principle \eqref{EqDP} satisfies for each $T\in[0,\infty]$,
\begin{multline*}
\mathbf{P}_{\epsilon}\Big(\|\f-\hat \f^{(\DP)}\|_n^2>C\Big(\min_{0\leq t\leq T}\Big\{\|r_t(K_n)\f\|_n^2+\frac{\sigma^2}{n}\mathcal{N}^g_n(t)\Big\}+\frac{\sigma^2\sqrt{u}} {\sqrt{n}}+\frac{\sigma^2u}{n}\Big)\Big)\\
\leq 5 e^{-u},\quad u>0,
\end{multline*}
where $C$ is a constant depending only on $A$.
\end{proposition}
A proof of Proposition~\ref{ThmCondDesignProxy} is given in Section~\ref{sec.proof.oracle.FD}. 
The above result is established for spectral filter estimators with regularizer $g$, under mild assumptions on the noise (only required to be sub-Gaussian).
Deriving this result under such mild assumptions has been made possible by introducing the proxy variance $v_t=\sigma^2\mathcal{N}^g_n(t)/n $ (from Eq.~\eqref{bias.proxyvariance}) instead of the more classical variance term in the r.h.s. of the inequality.
Nevertheless, it is still possible to upper bound the proxy-variance by the classical one at the price of an additional assumption as will be done in the next section (Theorem~\ref{ThmCondDesignOracle}).

\subsection{Main oracle inequality}\label{sec.oracle.DP.fixed.design}
As explained earlier, the purpose of the present section is to establish an oracle inequality for $\DP$. Compared with Proposition~\ref{ThmCondDesignProxy}, this is possible at the price of an additional assumption that we first motivate.

The desired derivation is made possible by connecting the proxy variance (that is, the $g$-effective dimension) to the classical variance. The key ingredient is that the $g$-effective dimension is typically dominated by the eigenvalues satisfying $t\hat\lambda_j>1$ as highlighted by the proof of Lemma~\ref{LemProxyVariance}. For such eigenvalues, \eqref{LFL} yields $b\leq \hat \lambda_jg_t(\hat\lambda_j)\leq 1$, which leads to conclude that the proxy and true variances only differ by a constant. 
This argument can be made rigorous by means of the next (sufficient) condition.
\begin{assumption} There is a constant $E>0$ such for each $k\geq 0$ satisfying $\hat\lambda_kT\geq 1$, we have
\begin{align}
 \hat\lambda_{k+1}^{-1}\sum_{j> k}\hat\lambda_j\leq E(k\vee 1) \tag{{\bf EVBound}} \label{Assume.EVBound}.
\end{align} 
\end{assumption}
Considering this ratio between the tail series of eigenvalues and the $k$th largest one has already been made in the literature \citep[see Definition~3 in][for instance where this ratio is named the ``effective rank"]{bartlett2019benign}.  
It is noticeable that \eqref{Assume.EVBound} encompasses two classical assumptions on the decay rate of the eigenvalues, respectively called polynomial \eqref{polyDecayLowUp} and exponential \eqref{expoDecayLowUp} decay. 
The next two examples are provided for illustrative purposes only. A more general result will be proved under milder constraints on the empirical eigenvalues by means of  \eqref{Assume.EffRank} (see Section~\ref{sec.SDP.inner.case} for more details).
%
%
%
\begin{example}[Polynomial eigenvalues decay]
If there exist numeric constants $\ell,L>0$, and $\alpha>1$ such that 
    \begin{align}
\ell j^{-\alpha}    \leq \hat\lambda_j \leq L j^{-\alpha},\qquad 1\leq j\leq n ,\tag{{\bf PolDecTS}} \label{polyDecayLowUp}
    \end{align}
    then \eqref{Assume.EVBound} holds true with $E= 1+2 L\ell^{-1}(\alpha+1)^{-1}$.
\end{example}
\begin{example}[Exponential eigenvalues decay]
If there exist numeric constants $\ell,L>0$, and $\alpha\in]0,1]$ such that 
    \begin{align}
\ell e^{-j^{\alpha}} \leq \hat\lambda_j \leq L e^{-j^{\alpha}},\qquad 1\leq j\leq n ,\tag{{\bf ExpDecTS}} \label{expoDecayLowUp}
    \end{align}
    then \eqref{Assume.EVBound} holds true with $$E = 1+\frac{2L}{\ell\alpha} \int_{0}^\infty\paren{1 + v}^{1/\alpha-1} e^{-v}\, dv .$$
\end{example}

We are now in position to explain how $\geffdim_n(t)$ (resp. the proxy variance) connects to $\operatorname{tr}(g_t^2(K_n)K_n^2)$ (resp. the variance) by means of \eqref{Assume.EVBound}.
\begin{lemma}\label{LemProxyVariance} 
Suppose that Assumptions \eqref{LFL} and \eqref{Assume.EVBound} hold. Then there is a constant $C>0$ depending only on $B,b$ and $E$ such that 
\[
\forall 0\leq t\leq T,\quad \geffdim_n(t)\leq C(\operatorname{tr}(g_t^2(K_n)K_n^2)+1).
\]
\end{lemma}
For the sake of comparison, let us mention that Lemma~\ref{LemProxyVariance} shows that the constant $C_{l^1,l^2}$ from Proposition~2.5 in \citep{MR3829522} does exist under mild assumptions on the decay rate of the eigenvalues.

\begin{proof}[Proof of Lemma~\ref{LemProxyVariance}]
If $t\hat\lambda_1<1$, then \eqref{LFU} and \eqref{Assume.EVBound} imply
\begin{align*}
    \geffdim_n(t)\leq Bt\sum_{j\geq 1}\hat\lambda_j\leq B\hat\lambda_1^{-1}\sum_{j\geq 1}\hat\lambda_j\leq BE,
\end{align*}
giving the claim with $C=BE$. On the other hand, if $t\hat\lambda_1\geq 1$, then let $k\geq 1$ be defined by $t\hat\lambda_{k+1} < 1\leq t\hat\lambda_{k}$.
Applying \eqref{LFU}, we have
\begin{align}
\geffdim_n(t)&=\sum_{j=1}^n\hat\lambda_jg_t(\hat\lambda_j)=\sum_{j\leq k}\hat\lambda_jg_t(\hat\lambda_j)+\sum_{j> k}\hat\lambda_jg_t(\hat\lambda_j)\nonumber\\
&\leq \sum_{j\leq k}\hat\lambda_jg_t(\hat\lambda_j)+Bt\sum_{j> k}\hat\lambda_j.\label{EqVarBound}
\end{align}
Now by the definition of $k$, \eqref{Assume.EVBound} and \eqref{LFL}, we have
\begin{align*}
t\sum_{j> k}\hat\lambda_j\leq \hat\lambda_{k+1}^{-1}\sum_{j> k}\hat\lambda_j\leq Ek\leq Eb^{-1}\sum_{j\leq k}\hat\lambda_jg_t(\hat\lambda_j).
\end{align*}
Inserting this into \eqref{EqVarBound}, we get
\begin{align*}
\sum_{j=1}^n\hat\lambda_jg_t(\hat\lambda_j)\leq C\sum_{j\leq k}\hat\lambda_jg_t(\hat\lambda_j)\leq b^{-1}C\sum_{j=1}^n\hat\lambda_j^2g_t^2(\hat\lambda_j)
\end{align*}
with $C=(1+b^{-1}BE)$.
\end{proof}

Combining \eqref{Assume.EVBound} and Lemma~\ref{LemProxyVariance} illustrates the way Proposition~\ref{ThmCondDesignProxy} can be transferred into a classical oracle inequality that is, involving bias and variance terms in the r.h.s., which is achieved by the next result.
\begin{thm} \label{ThmCondDesignOracle} Suppose that Assumptions~\eqref{SubGN},~\eqref{BdK} and \eqref{Assume.EVBound} hold and that the regularizer $g$ satisfies \eqref{LFL}. Then the early stopping rule $\DP$ based on the standard discrepancy principle \eqref{EqDP} satisfies for every $u>1$ the bound
\begin{align*}
\mathbf{P}_{\epsilon}\Big(\|\f-\hat \f^{(\DP)}\|_n^2>C\Big(\min_{0\leq t\leq T}\mathbf{E}_\epsilon\|\f-\hat \f^{(t)}\|_n^2+\frac{\sigma^2\sqrt{u}} {\sqrt{n}}+\frac{\sigma^2u}{n}\Big)\Big) \leq 5e^{-u},
\end{align*}
where $C$ is a constant depending only on $A$, $b$, $B$ and $E$.
\end{thm}
The  proof of Theorem~\ref{ThmCondDesignOracle} is deferred to Section~\ref{sec.proof.oracle.FD}.
Theorem~\ref{ThmCondDesignOracle} yields a non-asymptotic result, which contrasts for instance with the one of \cite{MR3564933} where conjugate gradient descent and minimum discrepancy principle are analyzed.
The above inequality is established with high probability, and it provides the precise sub-Gaussian and sub-exponential factors. This is a technical improvement compared to existing approaches where similar oracle inequalities in expectation are derived \citep{MR3859376,MR3829522}. 

The oracle performance in the r.h.s. of Theorem~\ref{ThmCondDesignOracle} is given through the expected excess risk (rather than the excess risk). This could be made at the price of an additional $\log T$ term, accounting for the uniform control of the discrepancy between the excess risk and its expectation over the first $T$ iterations. 

Let us also notice that Theorem~\ref{ThmCondDesignOracle} does not depend on any smoothness assumption on $f$. Making additional smoothness assumptions would immediately lead to a specific bound on $\min_{0\leq t\leq T}\mathbf{E}_\epsilon\|\f-\hat \f^{(t)}\|_n^2$ expressed in terms of convergence rate. This will be done in the random design framework in Section~\ref{sec.outer.case}, where it is shown that the classical discrepancy principle leads to optimal convergence rates whenever the latter rate is slower than the $n^{-1/2}$-rate. Such a situation can happen in the outer case $r<1/2$.

In contrast, the $1/\sqrt{n}$-rate is not negligible whenever the minimal bias-variance trade-off is smaller than (or of same order as)  $1/\sqrt{n}$. This holds true e.g. in the inner case $r\geq 1/2$. Compared to \citep{MR3859376} and \citep{MR3829522}, the term $\sigma^2/\sqrt{n}$ corresponds to their term $\sqrt{D}\delta^2$ (with the analogy noise level $\delta^2=\sigma^2/n$ and discretization dimension $D=n$). Moreover, in \citep{MR3859376} it has been shown for the specific case of spectral cut-off that such terms can not be avoided for early stopping rules based on the residual filtration. Hence, we conclude that the classical minimum discrepancy principle turns out to be useless when estimating smooth functions. This motivates considering smoothing-based strategies in Section~\ref{sec.smoothing.residuals}.

\subsection{Discussion}
\label{subsec.discussion}
As earlier emphasized, the $\sigma^2/\sqrt{n}$ term in Theorem~\ref{ThmCondDesignOracle} cannot be improved.
The reason for this term is the high variability in the stopping rule $\DP$ and the empirical risk (see Figure~\ref{fig.Perf.Sobolev.inner.histo.DP}). 
To illustrate this further, let us consider the deviation inequality for $\DP$ from Proposition \ref{ConcIneqVarDP} applied with $t$ satisfying $\geffdim_n(t)= (1+\delta)\geffdim_n(\tstar)$ with $\delta>1$, leading to
\begin{align}\label{EqConcTauDiscussion}
&\mathbf{P}_\epsilon(\DP>t) \leq 2\exp\Big(-c\Big( \delta\geffdim_n(\tstar) \wedge \frac{(\delta\geffdim_n(\tstar))^2}{n}\Big)\Big).
\end{align}
If, for instance, \eqref{polyDecayLowUp} and \eqref{Assume.SC} hold, then $\mathcal{N}^g_n(\tstar)$ is typically of order $n^{1/(2\alpha r+1)}$, meaning that the above (non-improvable) concentration bound becomes vacuous for $n^{1/(2\alpha r+1)}\ll n^{1/2}$. This is the case if $r$ is larger than $1/(2\alpha)$. In such settings, the classical discrepancy principle will typically lead to stopping times that are too large with high probability. 
Interestingly, we prove in the random-design context of Section~\ref{sec.outer.case} that the discrepancy principle can nevertheless achieve state-of-the-art rates under the condition $r\leq 1/(2\alpha)$.

The limitation of $\DP$ in the present context can be also interpreted as the consequence of trying to estimate a part of the signal that is smaller than the level of noise $\sigma$. This can be easily observed by computing the singular value decomposition (SVD) of the normalized Gram matrix $K_n$, and by computing the residuals in this new basis.
Then a natural idea to overcome this problem is the smoothing of the residuals, then reducing the contribution of these ``small coordinates'' to the (smoothed) residuals. This strategy has been already explored in the literature (see for instance \citet{MR3564933} for the CGD).
Studying how $\DP$ can be improved when combined with the smoothing of the residuals is the purpose
of Section~\ref{sec.smoothing.residuals}.
%

\section{SDP and oracle inequality: Fixed-design}
\label{sec.smoothing.residuals}

We now turn to a modification of the discrepancy principle based on the smoothing of the residuals that is, on the smoothed empirical risk.

\subsection{Smoothing-based discrepancy principle}
As discussed in Section~\ref{subsec.discussion}, the main drawback of the discrepancy principle-based rule $\tauDP$ results from the large variance of the empirical risk, leading to the $\sigma^2/\sqrt{n}$ error term in Theorem~\ref{ThmCondDesignOracle}. 

The purpose of the present section is to show how this error term can be avoided by considering a modified stopping rule called $\tauSDP$ based on the smoothing of residuals that is, the smoothed empirical risk. 
This can be encoded by considering the so-called smoothed empirical risk $\|L_n(\mathbf{Y}-S_n\hat f^{(t)})\|_n^2$ for some (smoothing) matrix $L_n\in\mathbb{R}^{n\times n}$. In what follows, we will restrict ourselves to the case where $L_n=\tilde{g}^{1/2}_T(K_n)K_n^{1/2}$ with regularizer $\tilde g$ (satisfying Definition~\ref{Def}) and consider 
\begin{align}\label{EqSDP}
\SDP=\inf\Big\{t\geq 0:\|\tilde{g}^{1/2}_T(K_n)K_n^{1/2}(\mathbf{Y}-S_n\hat f^{(t)})\|_n^2\leq \frac{\sigma^2\mathcal{N}_n^{\tilde{g}}(T)}{n}\Big\}\wedge T
\end{align}
with $T>0$. 
the choice $\tilde{g}_T(\lambda)=(\lambda+T^{-1})^{-1}$ as Tikhonov regularization results in the early stopping rule earlier studied in \cite{blanchard2012discrepancy} in the statistical inverse problem setting. Different choices for $L_n$ include $L_n=K_n^{s/2}$, $s\leq 1$, have been studied in \citep{B19} for the spectral cut-off filter algorithm.

Then the goal in what follows is to assess the statistical performance of the final estimator $\hat f^{(\tauSDP)}$, where $\tauSDP$ is obtained by the so-called \emph{smoothed discrepancy principle} (SDP).

\subsection{Main results}

The present section follows the same structure as above Section~\ref{sec.DP.Oracle.inequality.fixed.design} with firstly describing key deviation inequalities for $\SDP$ and the related smoothed bias and variance terms, and secondly formulating our main improved oracle inequality for $\SDP$.

\subsubsection{Deviation inequalities for the smoothed stopping rule}
Let us now explain how deviation inequalities in the case of the classical DP (Section~\ref{sec.scetch.proof}) can be extended to smoothed case. 
For simplicity of the present exposition, we restrict ourselves to $\SDP$ applied with the Tikhonov smoothing $\tilde g_t(\lambda)=(\lambda+t^{-1})^{-1}$. However, the next results are not limited to this choice.

Following the analysis of the classical discrepancy principle in Section~\ref{sec.scetch.proof}, it is easy to see that the expected smoothed empirical risk satisfies a basic inequality similar to \eqref{EqBasicIneq}. 
In fact introducing the \emph{smoothed $g$-effective dimension} $\tildengeffdim(t) = \operatorname{tr}((K_n+T^{-1})^{-1}K_ng_t(K_n)K_n)$, we have
\begin{align*}
 &\|r_t(K_n)\Af\|_n^2-2\frac{\sigma^2}{n}\tildengeffdim(t)\\
 &\leq \mathbf{E}_\epsilon\|r_t(K_n)\AY\|_n^2-\frac{\sigma^2}{n}\mathcal{N}_n(T) \leq \|r_t(K_n)\Af\|_n^2-\frac{\sigma^2}{n}\tildengeffdim(t),\quad t\geq 0,
\end{align*}
where $\tilde{\mathbf{a}}=(K_n+T^{-1})^{-1/2}K_n^{-1/2}\mathbf{a}$, for every $\mathbf{a}\in\mathbb{R}^n$.
This allows us to carry out the same basic comparison between $\SDP$ and the \emph{smoothed balancing stopping rule} 
\begin{align}\label{Deftildetstar}
    \tildetstar=\inf\Big\{t\geq 1:\|r_t(K_n)\Af\|_n^2\leq\frac{\sigma^2}{n}\tildengeffdim(t)\Big\}
\end{align}
(with $\tildetstar=\infty$ if such a $t$ does not exist).
Our first result in this line is the next deviation inequality for $\SDP$, which should be seen as the smoothing-based counterpart of Proposition~\ref{ConcIneqVarDP}.
\begin{proposition}\label{cor.tau.SDPy} 
If \eqref{SubGN} holds, then there is a constant $c>0$ depending only on $A$ such that for every $t>\tildetstar$, 
\begin{align*}
    \mathbf{P}_\epsilon(\SDP>t)\leq 2\exp\Big(-c\Big( y\wedge \frac{y^2}{\effdim_n(T)}\Big)\Big),\qquad y=\tildengeffdim(t)-\tildengeffdim(\tildetstar).
\end{align*}
In particular, for every $y>0$, we have
\begin{align*}
&\mathbf{P}_\epsilon(\tildengeffdim(\SDP)>\tildengeffdim(\tildetstar)+y) \leq 2\exp\Big(-c\Big( y \wedge \frac{y^2}{\effdim_n(T)}\Big)\Big).
\end{align*}
\end{proposition}
This is a simplified version of the deviation bound established in Proposition~\ref{prop.deviation.SDP}. 

Let us make a few comments mainly emphasizing the differences with Proposition~\ref{ConcIneqVarDP} established for $\DP$. 
Firstly, the former $n$ at the denominator of the exponent is now replaced by the empirical effective dimension $\effdim_n(T)$, which allows for taking into account the decay rate of the eigenvalues of $K_n$. In particular, the condition for having this probability meaningful (that is, close to 0) is no longer $\sqrt{n}\ll y$ but instead $\sqrt{\effdim_n(T)}\ll y$, which is typically much weaker if one can exploit some knowledge on the decay rate of the eigenvalues.
Secondly, the $g$-effective dimension in Proposition~\ref{ConcIneqVarDP} is now replaced by its smoothed version $\tildengeffdim(t)$. Since $\tildengeffdim(t)\leq \mathcal{N}_n^g(t)$, this leads to a slightly weaker deviations in terms of $y$.

\medskip

Let us emphasize that this deviation inequality of the $\tildengeffdim(t)$ serves for controlling the variance of $\hat{f}^{(t)}$. This results from the key observation that the term $\operatorname{tr}(g_t^2(K_n)K_n^2)$ (appearing in the variance of $\hat f^{(t)}$) can be bounded by a constant times $\tildengeffdim(t)$ (while in Section~\ref{sec.motivation.DP}, we only used that it is bounded by the $g$-effective dimension).

Similarly, the squared bias $\|r_t(K_n)\mathbf{f}\|_n^2$ can be also related to its smoothed version $\|r_t(K_n)\Af\|_n^2$, where the latter term is dealt with in the following simplified version of the deviation bound in Proposition~\ref{ConcIneqBias}.
\begin{proposition}\label{cor.tau.bias.SDP} 
Suppose that Assumptions \eqref{SubGN} and \eqref{BdK} hold. Then, for every $y>0$ such that $2\|r_{\Ltstar}(K_n)\Af\|_n^2+\sigma^2n^{-1}y>\|r_{T}(K_n)\Af\|_n^2$, we have
\begin{align*}
&\mathbf{P}_\epsilon\Big(\|r_{\SDP}(K_n)\Af\|_n^2> 2\|r_{\Ltstar}(K_n)\Af\|_n^2+\frac{\sigma^2}{n}y\Big)\leq 2\exp\Big(-c\Big(y\wedge \frac{y^2}{\effdim_n(T)}\Big)\Big).
\end{align*}
\end{proposition}

\subsubsection{Improved oracle inequality}\label{sec.oracle.SDP.fixed.design}
We are now in position to state an improved oracle inequality for the inner case that holds for the smoothed discrepancy principle (SDP), namely $\SDP$.
\begin{thm} \label{ThmCondDesignProxySmoothedNorm} Suppose that \eqref{SubGN}, \eqref{BdK}, \eqref{Assume.EVBound} and \eqref{Assume.SC} hold with $s=r-1/2\geq 0$ and the the regularizer $g$ satisfies \eqref{LFL}. Moreover, suppose that $\|(\Sigma+T^{-1})^{-1/2}(\Sigma_n-\Sigma)(\Sigma+T^{-1})^{-1/2}\|_{\operatorname{op}}\leq 1/2$ holds. Then early stopping rule $\SDP$ based on the smoothed discrepancy principle from \eqref{EqSDP} with regularizer $\tilde g$ such that \eqref{LFL} holds satisfies the bound
\begin{multline*}
\mathbf{P}_{\epsilon}\Big(\|\mathbf{f}-\hat {\mathbf{f}}^{(\tauSDP)}\|_n^2>C\Big(\min_{0\leq t\leq T}\mathbf{E}_\epsilon\|\f-\hat \f^{(t)}\|_n^2+\frac{\sigma^2\sqrt{u\mathcal{N}_n(T)}} {n} +\frac{\sigma^2u}{n}\\
+T^{-(1+2s)}+T^{-1}\|\Sigma_n-\Sigma\|_{\operatorname{op}}^{2\wedge 2s}\Big)\Big)\leq 5e^{-u},\quad u>1,
\end{multline*}
where the term $T^{-1}\|\Sigma_n-\Sigma\|_{\operatorname{op}}^{2\wedge 2s}$ can be dropped if $s\leq 1/2$.
\end{thm}
A proof of Theorem~\ref{ThmCondDesignProxySmoothedNorm} is given in Section \ref{sec.proof.oracle.FD}. Comparing this result to the oracle inequality in Theorem~\ref{ThmCondDesignOracle}, we see that we replaced the term $\sigma^2/\sqrt{n}$ by $\sigma^2\sqrt{\mathcal{N}_n(T)}/n$. Under \eqref{polyDecayLowUp}, for instance, we have $\mathcal{N}_n(T) \leq C T^{1/\alpha}$, meaning that $\sqrt{\mathcal{N}_n(T)}/n \leq 1/\sqrt{n}$ as long as $T \leq n^{\alpha}$. 

The event $\|(\Sigma+T^{-1})^{-1/2}(\Sigma_n-\Sigma)(\Sigma+T^{-1})^{-1/2}\|_{\operatorname{op}}\leq 1/2$ is needed to apply the source condition \eqref{Assume.SC} (formulated in terms of the population covariance operator)  in the empirical world. It can be further weakened (there is e.g. no event in the case $s=0$; see the proof of Lemma \ref{LemChangeBias}), but in its present form it is exactly the event needed to transfer the results from the fixed to the random design framework. This is the purpose of the next section. In particular, we will turn the above oracle inequality into a rate optimality statement, showing that the smoothed discrepancy principle is adaptive over a certain range of smoothness parameters and polynomial decay rates.
%

\section{The random design framework}\label{SecRD}

In this section we transfer our oracle inequalities from the fixed to the random design framework by means of a change-of-norm (or change of measure) argument exposed in Section~\ref{sec.change.norm}. 
The purpose of Section~\ref{sec.outer.case} is the analysis of the stopping rule based on the discrepancy principle (DP) in the \emph{outer case}, while Section~\ref{sec.SDP.inner.case} rather addresses its smoothed version (SDP) in the \emph{inner case}.

To keep the exposition as simple as possible in what follows, we focus on results given in terms of expectations from now on. Similar results expressed ``with high probability'' can be derived from the technical material developed in Sections~\ref{secProofFixedDesign} and~\ref{SecProofRandomDesign}, but at the price of more involved expressions.

\subsection{Change of norm argument}\label{sec.change.norm}

The first step in our analysis is a change of norm argument formulated by the next result, which controls the difference between the $L^2(\rho)$-norm ($\norm{\cdot}_\rho$) and its empirical version, namely the $n$-th norm ($\norm{\cdot}_n$). 
\begin{lemma}\label{LemChangeNorm1.main} 
Let $\delta\in(0,1)$ and $T> 0$. Then we have
\begin{align*}
\forall h\in\mathcal{H},\quad |\|S_n h\|_n^2-\|S_\rho h\|_{\rho}^2|\leq \delta (\|S_\rho h\|_{\rho}^2+T^{-1}\|h\|_{\mathcal{H}}^2)
\end{align*}
if and only if 
\begin{align*}
    \|(\Sigma+T^{-1})^{-1/2}(\Sigma_n-\Sigma)(\Sigma+T^{-1})^{-1/2}\|_{\operatorname{op}}\leq \delta.
\end{align*}
\end{lemma}
Lemma~\ref{LemChangeNorm1.main} establishes the equivalence between the uniform control of the difference between the squared $\rho$-~and $n$-th norms and deriving an upper bound on the operator norm of the normalized difference between the empirical and population covariance operators.
In particular if one of the assertion holds, then 
    \begin{align*}
\forall h\in\mathcal{H},\quad \norm{S_\rho h}_\rho^2 \leq \frac{1}{1-\delta}\norm{S_n h}_n^2 + \frac{\delta}{1-\delta}  \frac{\normH{h}^2}{T} 
    \end{align*} 
    gives rise to a natural strategy for upper bounding the $\rho$-norm of any function in $\H$. It consists first in upper bounding its $n$-th norm (which was the purpose of Sections~\ref{sec.oracle.DP.fixed.design} and~\ref{sec.oracle.SDP.fixed.design}), and then in controlling its $\H$-norm.

\medskip

\begin{proof}[Proof of Lemma~\ref{LemChangeNorm1.main}]
Using the identities $\|S_n h\|_n^2-\|S_\rho h\|_{\rho}^2=\langle (\Sigma_n-\Sigma) h,h \rangle_{\mathcal{H}}$ and $\|S_\rho h\|_{\rho}^2+T^{-1}\|h\|_{\mathcal{H}}^2=\|(\Sigma+T^{-1})^{1/2} h\|_{\mathcal{H}}^2$, the first assertion is equivalent to 
\begin{align*}
 \forall h\in\mathcal{H},\quad   |\langle (\Sigma_n-\Sigma) h,h \rangle_{\mathcal{H}}|\leq \delta\|(\Sigma+T^{-1})^{1/2} h\|_{\mathcal{H}}^2.
\end{align*}
Since $(\Sigma+T^{-1})^{1/2}$ is self-adjoint and strictly positive definite, this is the case if and only if 
\begin{align*}
 \forall h\in\mathcal{H},\quad   |\langle (\Sigma+T^{-1})^{-1/2}(\Sigma_n-\Sigma)(\Sigma+T^{-1})^{-1/2} h,h \rangle_{\mathcal{H}}|\leq \delta\| h\|_{\mathcal{H}}^2.
\end{align*}
This gives the claim.
\end{proof}

\subsection{DP performance: Outer case}
\label{sec.outer.case}

\subsubsection{Main result}
We now turn to the classical discrepancy principle for which we formulate a result in the outer case.
\begin{thm}\label{ThmDPHardProblems}
Suppose that \eqref{SubGN} and \eqref{BdK} hold. Suppose that the source condition \eqref{Assume.SC} holds with $r<1/2$ and that $f$ is bounded. Moreover, suppose that the regularizer $g$ satisfies \eqref{Qualif} with $q\geq r$. Then there are constants $c,C>0$ such that the standard discrepancy principle $\tauDP$ with emergency stop $T=cn/\log n$, $n\geq 2$, satisfies
\begin{align*}
    \mathbb{E}\|f-S_\rho\hat f^{(\tau_{DP})}\|_\rho^2\leq C\Big(\min_{0< t\leq c\frac{n}{\log n}}\Big\{t^{-2r}+\frac{\mathcal{N}(t)}{n}\Big\}+n^{-1/2}\Big).
\end{align*}
\end{thm}
The proof of Theorem \ref{ThmDPHardProblems} can be found in Section~\ref{sec.proof.DP.outer.case.random.design}. 
Unlike the results from Sections~\ref{sec.oracle.DP.fixed.design} and~\ref{sec.oracle.SDP.fixed.design} in the fixed design case, there is an additional constraint on the emergency stop $T$ that has to be smaller than $c n/\log n$. This constraint is related to the control of the probability of the event $\{\| (\Sigma+T^{-1})^{-1/2} (\Sigma_n-\Sigma)(\Sigma+T^{-1})^{-1/2}\|_{\operatorname{op}}\leq 1/2 \}$.

Without any further assumption on the decay rate of the eigenvalues, the effective dimension $\effdim(t)$ can be upper bounded by $M^2 t$; see e.g. Appendix \ref{appendix.effective.dim}. Theorem \ref{ThmDPHardProblems} thus gives
\begin{align*}
&\mathbb{E}\|f-S_\rho\hat f^{(\DP)}\|_\rho^2\leq C\max\Big(n^{-\frac{2r}{2r+1}},\Big(\frac{\log n}{n}\Big)^{2r},n^{-1/2}\Big)\leq Cn^{-\frac{2r}{2r+1}}.
\end{align*}
As a consequence, the classical discrepancy principle leads to optimal rates of convergence throughout the whole range $r\in(0,1/2)$ of the outer case (cf. \cite{FS19}). 

\subsubsection{Discussion and extensions for polynomial decay}\label{SecExample}
For some $\ED>0$ and $ \alpha>1$, suppose that
\begin{align}\label{EqPolDecUB}
\forall j\geq 1,\qquad \lambda_j\leq \ED j^{-\alpha}.\tag{{\bf PolDec}} 
\end{align}
By Lemma \ref{EqEffDim}(i) we have $\mathcal{N}(t)\leq Ct^{1/\alpha}$ for all $t>0$. Specialized to \eqref{EqPolDecUB}, Theorem \ref{ThmDPHardProblems} thus gives
\begin{align*}
&\mathbb{E}\|f-S_\rho\hat f^{(\DP)}\|_\rho^2\leq C\max\Big(n^{-\frac{2r}{2r+1/\alpha}}, \Big(\frac{\log n}{n}\Big)^{2r},n^{-1/2}\Big).
\end{align*}
In other words, we obtain up to some additional $\log n$ factors the following rates of convergence:
\begin{align*}
\begin{cases}
    n^{-\frac{2r}{2r+1/\alpha}},&\quad  \text{ if } 2r+1/\alpha>1,r\leq 1/(2\alpha),\\
    n^{-2r},&\quad \text{ if } 2r+1/\alpha\leq 1, r\leq 1/4,\\
    n^{-1/2},&\quad \text{ if } r>1/4,r>1/(2\alpha).
    \end{cases} 
\end{align*}
We see that the classical discrepancy principle achieves the optimal rates of convergence in the hard learning scenario if $1/2-1/(2\alpha)< r\leq 1/(2\alpha)$. 

In what follows we compare these rates to results from the literature, and we show how Theorem \ref{ThmDPHardProblems} can be improved under an additional condition on the kernel. Ignoring $\log n$ factors, the rate
\begin{align}\label{EqRates}
    \begin{cases}
         n^{-\frac{2r}{2r+1/\alpha}},&\quad 2r+1/\alpha \geq 1,\\
         n^{-2r},&\quad 2r+1/\alpha \leq  1.
    \end{cases} 
\end{align}
is the state-of-the-art result for the outer case assuming only \eqref{BdK}; see e.g. Corollary 4.4 in \citep{lin2018optimal}. There are possible improvements under stronger boundedness assumptions. In fact, if there is a $\mu< 1$ such that $\|\Sigma^{\mu/2-1/2}k_X\|_{\mathcal{H}}\leq C_\mu M$ almost surely, then one can achieve up to $\log n$ factors the rate
\begin{align}\label{EqRatesExt}
    \begin{cases}
         n^{-\frac{2r}{2r+1/\alpha}},&\quad 2r+1/\alpha \geq \mu,\\
         n^{-\frac{2r}{\mu}},&\quad 2r+1/\alpha \leq  \mu,
    \end{cases} 
\end{align}
see e.g. \citep{pillaud2018statistical} and \citep{FS19}. Such improvements are also possible in our case, which is the purpose of the next result proved in Section~\ref{sec.proof.DP.outer.case.random.design}.
\begin{thm}\label{ThmDPHardProblemsExt}
Suppose that \eqref{SubGN}, \eqref{BdK}, \eqref{Assume.SC} and \eqref{EqPolDecUB} holds with $r<1/2$ and that $f$ is bounded. Suppose that there is a $\mu\in[0,1)$ and a constant $C_\mu>0$ such that $\|\Sigma^{\mu/2-1/2}k_X\|_{\mathcal{H}}\leq C_\mu M$. Finally, suppose that the regularizer $g$ satisfies \eqref{Qualif} with $q\geq r$. Then there are constants $c,C>0$ such that the standard discrepancy principle $\tauDP$ with emergency stop $T=c_1(n/\log n)^{1/\mu}$, $n\geq 2$, satisfies
\begin{align*}
    \mathbb{E}\|f-S_\rho\hat f^{(\DP)}\|_\rho^2\leq C\Big(\min_{0<t\leq c(\frac{n}{\log n})^{1/\mu}}\Big\{t^{-2r}+\frac{t^{1/\alpha}}{n}\Big\}+n^{-1/2}\Big).
\end{align*}
\end{thm}
Let us first notice that introducing the stronger assumption involving the parameter $0\leq \mu<1$ allows to enlarge the emergency stop $T$ and thus the range of values of $t$ over which the minimum in the r.h.s. is computed since $1/\mu>1$.
By the arguments from above we also see that the classical discrepancy principle achieves the optimal rates of convergence in the hard learning scenario if $2r+1/\alpha \geq \mu$ and $\mu/2-1/(2\alpha)< r\leq 1/(2\alpha)$. In the setting of Sobolev spaces any $\mu>1/\alpha$ is admissible (see Example 2 in \citep{pillaud2018statistical}), leading to the adaptation interval $r\in (0,1/(2\alpha)]$.

\subsection{SDP performance: Inner case}
\label{sec.SDP.inner.case}

In the present section, we establish two inequalities in the inner case for $\SDP$. 
The main difference between these results lies in the use of different emergency stopping times $T$. In the first one (Theorem~\ref{CorExtThmSDPBoundedK}), a deterministic emergency stop of size at most $n/\log n$ is used, while the second result (Theorem~\ref{CorExtThmSDPBoundedKEmpT}) allows for using a more sophisticated \emph{data-driven} emergency stop defined as the solution of a fixed-point equation, which gives rise to an optimal (leading to statistical adaptivity) early stopping rule that can be applied in practice.

\subsubsection{Main result}
\label{sec.main.result.random.design.innercase}

The transfer from the fixed design to the random design cases requires first an additional assumption on the effective rank, which is the population version of the former \eqref{Assume.EVBound} assumption earlier introduced in the fixed design case.
\begin{assumption} There exists a constant $E'>0$ such that, for each $k\geq 0$, we have
\begin{align}
 \lambda_{k+1}^{-1}\sum_{j> k}\lambda_j\leq E'(k\vee 1). \tag{{\bf EffRank}} \label{Assume.EffRank}
\end{align} 
\end{assumption}
This assumption is a population version of \eqref{Assume.EVBound} and Lemma~\ref{lem.EffRank.RandomEmpirical} specifies an event on which it indeed implies \eqref{Assume.EVBound}. Similarly as in Section~\ref{sec.oracle.DP.fixed.design}, \eqref{Assume.EVBound} is needed to bound the proxy variance term in terms of the smoothed proxy variance term (cf. Lemma~\ref{trace.g.effective.dimension.relaxed.EVBound}).
Under this additional assumption the smoothed discrepancy principle from Section \ref{sec.smoothing.residuals} achieved the following bound.
\begin{thm}\label{CorExtThmSDPBoundedK}
Suppose that Assumptions \eqref{SubGN}, \eqref{Assume.SC}, \eqref{BdK} and \eqref{Assume.EffRank} hold with $s=r-1/2\geq 0$. Moreover, suppose that the regularizer $g$ satisfies \eqref{Qualif} with $q\geq r$. Then there are constants $c,C>0$ such that the smoothed discrepancy principle $\tauSDP$ from \eqref{EqSDP} with $\tilde{g}_t(\lambda)=(\lambda+t^{-1})^{-1}$ and $T\leq cn/(\log n)$, $n\geq 2$, achieves the bound
\begin{align*}
    \mathbb{E}\|f-S_\rho\hat f^{(\tauSDP)}\|_\rho^2
\leq C\Big(\min_{0< t\leq T}\Big\{\frac{1}{t^{2r}}+\frac{\mathcal{N}(t)}{n}\Big\}+ \frac{\sqrt{\mathcal{N}(T)}}{n}\Big).
\end{align*}
\end{thm}
The proof of Theorem \ref{CorExtThmSDPBoundedK} can be found in Section~\ref{proof.Theorem.oracle.random.T.Deterministic}. Note that the condition $q\geq r$ on the qualification error of $g$ can be dropped by introducing slower rates depending also on $q$.

\medskip

Without any further assumption on the decay rate of the eigenvalues (except of \eqref{Assume.EffRank}), Theorem \ref{CorExtThmSDPBoundedK} gives
\begin{align*}
&\mathbb{E}\|f-S_\rho\hat f^{(\DP)}\|_\rho^2\leq C\max\Big(n^{-\frac{2r}{2r+1}},T^{-2r},\frac{\sqrt{T}}{n}\Big).
\end{align*}
Let us now assume that a lower bound $r_0\geq 1/2$ is known on the smoothness parameter $r$, which means that \eqref{Assume.SC} holds with $r\geq r_0$. Then using this side information, the choice $T=n^{1/(2r_0+1)}$ (that becomes smaller than $n/\log n$ as $n$ grows) leads to 
\begin{align*}
&\mathbb{E}\|f-S_\rho \hat f^{(\tauSDP)}\|_\rho^2\leq C\max\Big(n^{-\frac{2r}{2r+1}},n^{-\frac{4r_0+1}{4r_0+2}}\Big).
\end{align*}
This entails that the smoothed discrepancy principle $\SDP$ reaches optimal rates of convergence throughout the range
\begin{align*}
    r\in\Big[r_0,2r_0+\frac{1}{2}\Big].
\end{align*}
For instance with $r_0=1/2$ (that is the inner case without additional smoothness information), $\SDP$ is optimal over the range $r\in\croch{1/2,3/2}$.

Instead of choosing $T=n^{1/(2r_0+1)}$, one might also define $T$ as the solution to the fixed-point equation $ c_0T^{-2r_0} =\mathcal{N}(T)/n$ with $c_0=1$, which corresponds to a bias-variance trade-off in the case $r=r_0$. This would lead to the same adaptation interval $[r_0,2r_0+1/2]$. 
Such and related fixed-point equations play a central role in empirical risk minimization problems; see e.g. \citep{MR2166554} and \citep{MR2329442}, and it is easy to see, using the proof of Lemma \ref{LemEffDimGenDim} and Proposition 3.3 in \citep{MR2829871}, that the effective dimension $\mathcal{N}(t)$ can be bounded from below and above in terms of local Rademacher averages.

\medskip

With an additional assumption such as a polynomial eigenvalue decay, the previous analysis can be further applied. If \eqref{EqPolDecUB} and \eqref{Assume.SC} hold with $r\geq r_0$, then the choice $T^{2r_0}\mathcal{N}(T)= n$ leads to
\begin{align*}
&\mathbb{E}\|f-S_\rho \hat f^{(\tauSDP)}\|_\rho^2\leq C\max\Big(n^{-\frac{2r\alpha}{2r\alpha+1}},n^{-\frac{4\alpha r_0+1}{4\alpha r_0+2}}\Big),
\end{align*}
meaning that the smoothed discrepancy principle leads to optimal rates of convergence throughout the range
 \begin{align}\label{EqRangeS0}
r\in\Big[r_0,2r_0+\frac{1}{2\alpha}\Big].
\end{align}
Let us emphasize that this range of values is narrower than the previous one derived without any assumption on the eigenvalue decay ($\alpha>1$). This owes to the fact that, by specifying an eigenvalue decay assumption (that is, by choosing a given kernel), we restrict the smoothness of the functions in the induced Hilbert space that can be well approximated.

\subsubsection{Improvement towards data-driven emergency stops}
In previous Section~\ref{sec.main.result.random.design.innercase}, we have chosen a (deterministic) $T$ as the solution of the equation $t^{2r_0}\mathcal{N}(t)=  c_0n$ by taking advantage of the prior knowledge of a lower bound $r_0$ on the smoothness parameter. Without such an a priori knowledge on $r$, the equation $T\mathcal{N}(T)=  c_0n$ provides a natural choice for $T$. Yet, such a choice is not achievable in practice since $\mathcal{N}(t)$ is not known.

In this section we show that similar bounds hold true if $T=T(X_1,\dots,X_n)$ is allowed to depend on the covariates $X_1,\dots,X_n$ (but not on the responses). This is possible since all results established in the fixed design case continue to hold. The following result focuses on the choice $T\mathcal{N}_n(T)= c_0n$.
\begin{thm}\label{CorExtThmSDPBoundedKEmpT}
Suppose that Assumptions \eqref{SubGN}, \eqref{Assume.SC}, \eqref{BdK} and \eqref{Assume.EffRank} hold with $s=r-1/2\geq 0$. Moreover, suppose that the regularizer $g$ satisfy \eqref{Qualif} with $q\geq r$. Let $\hat T>0$ be the solution of $\hat T\mathcal{N}_n(\hat T)=n$ (set $\hat T=\infty$ if such a solution does not exist). 
Then the smoothed discrepancy principle $\tauSDP$ from \eqref{EqSDP} with $\tilde{g}_t(\lambda)=(\lambda+t^{-1})^{-1}$ and $T=\min(\hat T,cn/\log n)$, $n\geq 2$ and $c>0$ sufficiently small, achieves the bound
\begin{align*}
  &\mathbb{E}\|f-S_\rho\hat f^{(\tauSDP)}\|_\rho^2\\
  &\leq C\Big(\min_{t>0}\Big\{t^{-2r}+\frac{\effdim(t)}{n}\Big\}+\sqrt{\frac{1}{n}\min_{t>0}\Big\{ t^{-1}+\frac{\mathcal{N}(t)}{n}\Big\}}+\frac{\log n}{n}\Big).
\end{align*}
\end{thm}
The proof of Theorem \ref{CorExtThmSDPBoundedKEmpT} can be found in Section~\ref{proof.Theorem.oracle.random.T.Random}.
Compared to the statement in Theorem~\ref{CorExtThmSDPBoundedK}, the term $\sqrt{\mathcal{N}(T)}/n$ has disappeared. Actually it has been replaced by the square-root on the r.h.s. of the above inequality due to the control of $\sqrt{\mathcal{N}(T)}$ with the present (random) choice of $T=\min(\hat T,cn/\log n)$. As can be easily checked from the proof, the control of this term is also responsible for the additional $(\log n)/n$, which does not really influence our conclusion regarding convergence rates.

Let us also remark that the above definition $\hat T$ with $c_0=1$ does not take into account constants such as the variance $\sigma^2$ or $\norm{f}_{\H}$ for instance that should arise from the upper bounds on the variance or bias terms. Obviously introducing these constants in the fixed-point equation would not modify our conclusion regarding the convergence rates and the statistical adaptivity property, which is the main achievement of the present analysis. In practice, one could replace these constants in the upper bound on the bias term by upper bounds with high probability derived from the empirical risk evaluated at $0$.

\paragraph{Illustration on two classical eigenvalue decay assumptions}

Since the interpretation in terms of convergence rates is not easy from the statement in Theorem~\ref{CorExtThmSDPBoundedKEmpT}, let us now consider two illustrative examples allowing for drawing further insightful conclusions.

\begin{example}[Polynomial decay]
Under the assumptions of Theorem \ref{CorExtThmSDPBoundedKEmpT} and \eqref{EqPolDecUB}, we get \begin{align}\label{EqChoiceTInnerCase}
&\mathbb{E}\|f-S_\rho\hat f^{(\tauSDP)}\|_\rho^2\leq C\max\Big(n^{-\frac{2\alpha r}{2\alpha r+1}},n^{-\frac{2\alpha+1}{2\alpha+2}}\Big).
\end{align}  
This means that, by including the \emph{data-driven} choice of $\hat T$, the smoothed discrepancy principle $\SDP$ still reaches statistical adaptivity (that is, automatically enjoys optimal rates of convergence) throughout the range $r\in[1/2,1+1/(2\alpha)]$. 

Note that this choice for $\hat T$ corresponds to the stopping rule defined by Eq.~(6) in \citep{MR3190843}. 
The striking remark is that \citep{MR3190843} establishes the rate $n^{-\alpha/(\alpha+1)}$, while we obtain an estimator automatically achieving the optimal rate $n^{-\paren{2\alpha r}/\paren{2\alpha r+1}}$ throughout $r\in[1/2,1+1/(2\alpha)]$ and the rate $n^{-(\alpha+1/2)/(\alpha+1)}$ otherwise.
This proves that $\SDP$ is uniformly better than the stopping rule of \citep{MR3190843} in the inner case ($r\geq 1/2$) and under a polynomial decay assumption.
\end{example}

\medskip

\begin{example}[Exponential decay]
For some $\ED>0$ and $ \alpha>1$, suppose that $\lambda_j\leq  e^{- \ED j^\alpha}$ for every $j\geq 1$. Applying Theorem \ref{CorExtThmSDPBoundedKEmpT} and  Lemma \ref{EqEffDim}(ii), we get
\begin{align*}
&\mathbb{E}\|f-S_\rho\hat f^{(\SDP)}\|_\rho^2\leq C\frac{(\log n)^{1/\alpha}}{n}.
\end{align*}
\end{example}

\section{Simulation experiments}\label{sec.simulations.experiments}
The goal of the present section is to illustrate the main behaviors of the stopping rules under consideration, as predicted from a theoretical perspective, respectively in Sections~\ref{sec.outer.case} and~\ref{sec.SDP.inner.case}.
\subsection{Simulation design: Generating synthetic data}
The present simulation experiments are carried out with the Landweber algorithm (that is, gradient descent with constant step-size $\eta>0$ along the iterations) as described in Section~\ref{SecSpectralFilter}.
The sample size $n$ varies within $\acc{200,400,600,800,1\,000}$ and the number of replicates in all the experiments is $N=200$. 
In all the simulation experiments, when applying the smoothed discrepancy principle rule $\SDP$ (see Eq.~\eqref{EqSDP}).

The data are drawn from the model described by \eqref{def.model} with the variance $\sigma^2$ of the Gaussian noise to be equal to 1, and where the deterministic vector $(x_1,\ldots,x_n)$ is defined by $x_i=i/n$ for $1\leq i\leq n$. 
Two distinct settings have been considered with specific tuning of the related parameters.
\begin{itemize}
    \item Outer case (see also Section~\ref{sec.outer.case}):
The regression function $f$ to be estimated is given for all $x\in[0,1]$ by
    \begin{align*}
    f(x)  = 2 \1_{[0.15,0.3[}(x) - \1_{[0.3,0.5[}(x) + \1_{[0.5, 0.85[}(x) - \1_{[0.85, 1[}(x).
    \end{align*}
    \sloppy
The results are only reported for the Sobolev kernel ($k_S(x,y)=\min(x,y)$, for $x,y\in[0,1]$). The maximum number of iterations, called $T_{\max}$ is respectively equal to 500 if $n\leq 400$, 1\,000 if $n=600$, 2\,000 if $n=800$, and 3\,000 if $n=1\,000$. The step-size of the Landweber algorithm is $\eta=2.4$, and the emergency stopping time $T$ is chosen such that $T = 2 n/\log n$  for $\SDP$ (see Theorem~\ref{ThmDPHardProblems}), and $T=T_{\max}$ for $\DP$.  
    
    \item Inner case (see also Section~\ref{sec.SDP.inner.case}):
    \begin{align*}
    f(x)  = \frac{1+x}{2} \sin(2\pi x(1+x)) .
    \end{align*}
For the inner case, two reproducing kernels are used: the Sobolev kernel (see above) and the Gaussian kernel ($k_G(x,y) = \exp\paren{ (x-y)^2/w^2}$, with bandwidth $w=0.02$). The maximum number of iterations is $T_{\max}=500$. The step-size of the Landweber algorithm is respectively set at $\eta=2.4$ for the Sobolev kernel, and $\eta=0.5$ for the Gaussian kernel. The emergency stopping time $T$ is chosen such that $T = 4 \sqrt{n}$ for $\SDP$ (see the discussion following Theorem~\ref{CorExtThmSDPBoundedK} with $r_0=1/2$) and $T=T_{\max}$ for $\DP$. 
\end{itemize}

For any given stopping rule $\hat{t}$, its performance is measured by means of the squared empirical norm $\|\mathbf{f} - \hat{\mathbf{f}}^{(\hat{t})}\|_n^2$ averaged over the $N=200$ replications, which is called the (averaged) ``loss'' for short in what follows.

\subsection{The outer case}

Figure~\ref{fig.example.outercase} displays an example of signal generated from the outer case framework.
The piecewise-constant regression function (red curve) makes the estimation problem a difficult task as long as one is limited to using functions from the reproducing kernel Hilbert space (RKHS) generated by the Sobolev kernel $k_S$. This justifies calling this situation the outer case. 
\begin{figure}[h!]
\hspace*{-1cm}
\begin{subfigure}{.5\textwidth}
\centering
    \includegraphics[width=\textwidth]{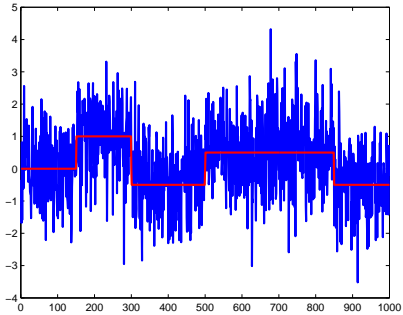}
\caption{Realization of the Outer case model.\hfill ~\\}
\label{fig.example.outercase}
\end{subfigure}
\hspace*{1cm}
\begin{subfigure}{.5\textwidth}
\centering    
    \includegraphics[width=\textwidth]{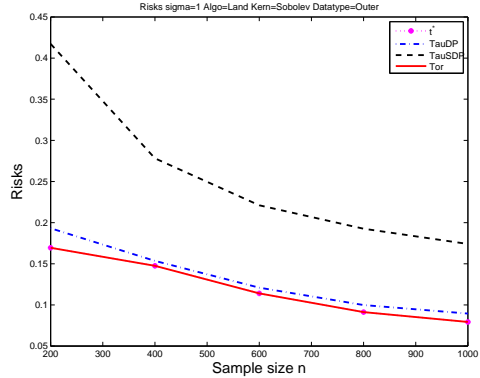}
\caption{Averaged losses of $t_{or}$, $\DP$, and $\SDP$ in the Outer case. The number of replications is $N=200$.}
\label{fig.Perf.outercase}
\end{subfigure}
\caption{Left: Instance of signal generated from the Outer case model. Right: averaged loss performances versus the increasing sample size.}
\end{figure}
For increasing sample sizes, Figure~\ref{fig.Perf.outercase} displays the empirical performances (measured in terms of the averaged loss) of several stopping rules, namely $\DP$, and $\SDP$. They are compared to the performance of the so-called \emph{oracle stopping rule} denoted by $t_{or}$ and defined as a global minimum location of the risk that is,
\begin{align}\label{oracle.stopping.rule}
    t_{or} = \operatorname{argmin}_{0<t\leq T_{\max}}\limits \mathbf{E}_{\epsilon}\| \mathbf{f} - \hat{\mathbf{f}}^{(t)}\|_n^2.
\end{align}

Although all the performances improve as the sample size grows, the performance of $\DP$ still remains uniformly closer to that of $t_{or}$, than the one of $\SDP$.
Keeping in mind that $\SDP$ is known to improve upon $\DP$ in the case of smooth regression functions (inner case that is, $r\geq 1/2$), it confirms that the present situation is by contrast a true instance of an outer case ($r<1/2$), meaning  that $f$ is outside the RKHS.

More precisely, since $f$ lies outside the RKHS, the expected number of iterations required for achieving a reliable estimator of $f$ is large. This is what we observe with the oracle stopping rule $t_{or}$ which remains close to the maximum number of iterations $T_{\max}$ as $n$ grows.
One main feature in designing $\SDP$ is the smoothing of the residuals as a means for avoiding too large values of the stopping rule (compared to $\DP$). Therefore the present situation is one typical instance where the trend of $\DP$ to take large values (unlike $\SDP$) makes this stopping rule a better candidate.

\subsection{The inner case}

Figure~\ref{fig.example.innercase} displays an example of signal generated in the inner case.
By contrast with the previous example (outer case), the smoothness of the regression function $f$ allows for using both the Gaussian and the Sobolev kernels, respectively denoted by $k_G$ and $k_S$. Their respective performance are summarized in Figures~\ref{fig.Perf.Sobolev.inner} and~\ref{fig.Perf.Gaus.inner}, where the different curves display the averaged loss for several stopping rules, namely $\tstar$, $\DP$, $\SDP$, and the oracle stopping rule $t_{or}$ (see Eq.~\eqref{oracle.stopping.rule}). 
\begin{figure}[h!]
\hspace*{3cm}\begin{subfigure}{.5\textwidth}
\centering
\includegraphics[width=\textwidth]{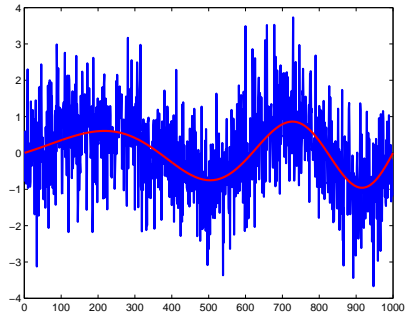}
\caption{Realization of the Inner case model.}    
\label{fig.example.innercase}
\end{subfigure}\\
\hspace*{-1cm}
\begin{subfigure}{.5\textwidth}
\centering
    \includegraphics[width=\textwidth]{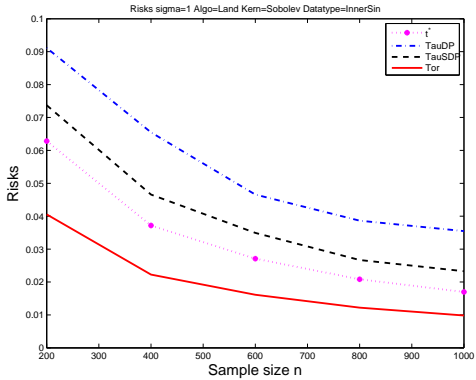}
\caption{Sobolev kernel.}
\label{fig.Perf.Sobolev.inner}
\end{subfigure}
\hspace*{1cm}
\begin{subfigure}{.5\textwidth}
\centering    
    \includegraphics[width=\textwidth]{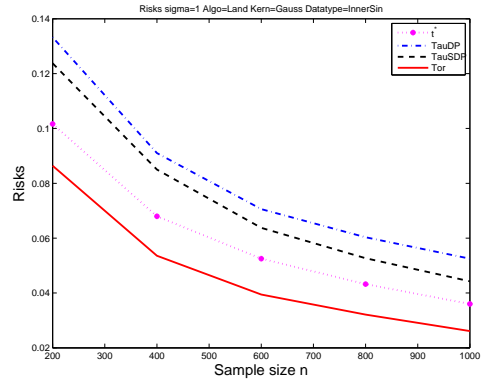}
\caption{Gaussian kernel.}
\label{fig.Perf.Gaus.inner}
\end{subfigure}
\caption{Averaged losses of $t_{or}$, $\tstar$, $\DP$, and $\SDP$ in the Inner case. The number of replications is $N=200$.}
\end{figure}

All the curves from Figures~\ref{fig.Perf.Sobolev.inner} and~\ref{fig.Perf.Gaus.inner} decrease as $n$ grows.
The best performance is uniformly achieved by $\tstar$, which is the stopping rule reaching the trade-off between the bias and the (proxy-)variance term (see also Figure~\ref{fig.Illustration.ESR}). From an asymptotic perspective, this is the best choice one can make in the present early stopping context. In particular, the data-drive stopping rules such as $\DP$ and $\SDP$ are estimating $\tstar$. It is then consistent that their respective performances are worse than that of $\tstar$.

For both the kernels $k_G$ and $k_S$, the worst performance is achieved by $\DP$. This sub-optimal behaviour in terms of averaged loss results from the higher variability of $\DP$ compared to $\SDP$, as it can be observed from the histograms of Figures~\ref{fig.Perf.Sobolev.inner.histo.DP} and ~\ref{fig.Perf.Sobolev.inner.histo.SDP} obtained with $n=800$ and $T_{\max}=500$.
In particular, this emphasizes that the residual smoothing encoded within the $\SDP$ stopping rule induces a considerable variance reduction, which avoids stopping too late (and then wasting time).
\begin{figure}[h!]
\hspace*{-1cm}
\begin{subfigure}{.5\textwidth}
\centering
    \includegraphics[width=\textwidth]{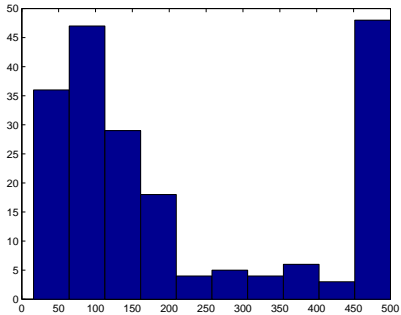}
\caption{Empirical distribution of $\DP$.}
\label{fig.Perf.Sobolev.inner.histo.DP}
\end{subfigure}
\hspace*{1cm}
\begin{subfigure}{.5\textwidth}
\centering    
    \includegraphics[width=\textwidth]{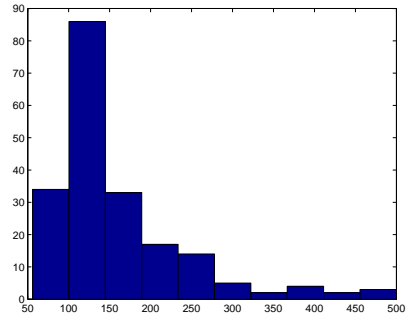}
\caption{Empirical distribution of $\SDP$.}
\label{fig.Perf.Sobolev.inner.histo.SDP}
\end{subfigure}
\caption{Empirical distribution of $\DP$ and $\SDP$ over $N=200$ replications for the Sobolev kernel with $n=800$ in the Inner case.}
\end{figure}

\section{Proofs for fixed-design results}\label{secProofFixedDesign}
In this section, we analyze the discrepancy principle conditional on the design. 

\subsection{A unified framework}
\label{preliminary.results}
A linearly transformed model is now introduced for simultaneously dealing with the smoothed and non-smoothed cases.
\[\AY=L_n\Ybf=L_n\mathbf{f}+L_n\boldsymbol{\epsilon}=\Af+\Ae,\qquad L_n\in\mathbb{R}^{n\times n},\]
with $L_n$ satisfying $\|L_n\|_{\operatorname{op}}\leq 1$. The new noise variable $\Ae$ is mean-zero and has covariance matrix $\sigma^2L_nL_n^T$. 
For a regularizer $g$ in the sense of Definition~\ref{Def}, our main goal is to analyze the stopping rule $\tau$ defined by
\begin{align}\label{EqDPPreSmooth}
    \tau=\inf\Big\{t\geq 0:\|\AY -K_ng_t(K_n)\AY\|_n^2=\|r_t(K_n)\AY\|_n^2\leq \frac{\sigma^2\operatorname{tr}(L_nL_n^T)}{n}\Big\}\wedge T
\end{align}
with $T\in(0,\infty]$. For $L_n=I_n$ the stopping rule in \eqref{EqDPPreSmooth} coincides with $\DP$ from \eqref{EqDP}, while for $L_n=\tilde g_T^{1/2}(K_n)K_n^{1/2}$ with regularizer $\tilde{g}$ it coincides with $\SDP$ from \eqref{EqSDP}. 

Moreover, the stopping rule \eqref{EqDPPreSmooth} can be interpreted as applying the classical discrepancy principle to the smoothed data $\AY$ and the class of estimators $K_ng_t(K_n)\AY=S_ng_t(\Sigma_n)S_n^*\AY$ where spectral regularization is applied to the smoothed data.
\begin{definition}\label{gDffectiveDimSmoothed}
For every $t\geq 0$ and every regularizer $g$, we define the smoothed $g$-effective dimension by
\begin{align*}
    \tildengeffdim(t)=\operatorname{tr}(L_nL_n^TK_n g_t(K_n)).
\end{align*}
\end{definition}
\begin{lemma}[Basic inequality]\label{EqBasicIneqDP}
Assumption~\eqref{BdF} yields, for every $t\geq 0$,
\begin{align*}
& \|r_t(K_n)\Af\|_n^2-2\frac{\sigma^2}{n}\tildengeffdim(t) \\
&\leq \mathbf{E}_{\epsilon}\|r_t(K_n)\AY\|_n^2-\frac{\sigma^2}{n}\operatorname{tr}(L_nL_n^T) \leq \|r_t(K_n)\Af\|_n^2-\frac{\sigma^2}{n}\tildengeffdim(t).
\end{align*}
\end{lemma}
Since $g$ is a regularizer, the term
\[
\|r_t(K_n)\Af\|_n^2=\sum_{j=1}^nr_t^2(\hat\lambda_j)\langle \hat v_j,  \Af\rangle_n^2
\]
is continuous and non-increasing in $t\geq 0$, while the term
\begin{align*}
\frac{\sigma^2}{n} \tildengeffdim(t)=\frac{\sigma^2}{n}\operatorname{tr}(L_nL_n^Tg_t(K_n)K_n)   = \frac{\sigma^2}{n} \sum_{j=1}^n\| L_n^T\hat v_j\|_2^2\hat\lambda_jg_t(\hat\lambda_j)
\end{align*}
is continuous and non-decreasing in $t\geq 0$. Moreover, by Definition~\ref{Def}, the term $\|r_t(K_n)\Af\|_n^2$ converges to zero as $t\to +\infty$, while the term $\sigma^2n^{-1}\tildengeffdim(t)$ is equal to zero for $t=0$.
Hence, we can define the following balancing stopping rule
\begin{align}\label{DeftstarL}
    \Ltstar=\inf\Big\{t\geq 0:\|r_t(K_n)\Af\|_n^2=\frac{\sigma^2}{n}\tildengeffdim(t)\Big\}.
\end{align}
If such a $t$ does not exist, then we set $\tildetstar=\infty$. By the above properties, this can only happen if $\tildengeffdim(t)=0$ for every $t\geq 0$ meaning that we can set $\|r_{\tildetstar}(K_n)\Af\|_n^2=0$ and $\sigma^2n^{-1}\tildengeffdim(\tildetstar)=0$ in this case.

\begin{proof}[Proof of Lemma~\ref{EqBasicIneqDP}]
We have 
\begin{align*}
\|r_t(K_n)\AY\|_n^2=\|r_t(K_n)\Af+r_t(K_n)\Ae\|_n^2
\end{align*}
and thus, using $\Ae=L_n\boldsymbol{\epsilon}$ and $r_t(K_n)=I-K_ng_t(K_n)$,
\begin{align*}
\mathbf{E}_{\epsilon}\|r_t(K_n)\AY\|_n^2 &=\|r_t(K_n)\Af\|_n^2+\frac{\sigma^2}{n}\operatorname{tr}(L_nL_n^T(I-K_ng_t(K_n))^2)\\
&=\|r_t(K_n)\Af\|_n^2+\frac{\sigma^2}{n}\operatorname{tr}(L_nL_n^T)\\
&-2\frac{\sigma^2}{n}\operatorname{tr}(L_nL_n^Tg_t(K_n)K_n)+\frac{\sigma^2}{n}\operatorname{tr}(L_nL_n^Tg_t^2(K_n)K_n^2).
\end{align*}
The lower bound follows from the fact that the last term is non-negative, while the upper bound follows from \eqref{BdF}.
\end{proof}

\subsection{Deviation inequality for the variance and bias parts}\label{sec.deviation.inequalities}
The results of this section are improvements over previous results from \citep{MR3859376} and \citep{MR3829522}, providing more precise sub-Gaussian and sub-exponential factors. Surprisingly, these improvements result from different arguments based on a more basic comparison between the discrepancy principle and its reference balancing stopping rule $\tildetstar$.

\subsubsection{Deviation inequality for the variance part}
Our first main result is a deviation inequality for $\tau$ from \eqref{EqDPPreSmooth}.
\begin{proposition}\label{prop.deviation.SDP}
Suppose that Assumptions \eqref{SubGN} and \eqref{BdF} hold.
Then, for every $t> \Ltstar$, we have
\begin{align*}
    \mathbf{P}_\epsilon(\tau>t)\leq 2\exp\Big(-c\Big( y \wedge \frac{y^2}{\tr\paren{ L_nL_n^\top}}\Big)\Big),\qquad y=\tildengeffdim(t) - \tildengeffdim(\Ltstar),
\end{align*}
where $c>0$ is a constant depending only on $A$.
\end{proposition}
\begin{proof}[Proof of Proposition~\ref{prop.deviation.SDP}]
Inserting the definition of the discrepancy principle in \eqref{EqDPPreSmooth}, we have
\begin{align}
\mathbf{P}_\epsilon(\tau> t)\leq\mathbf{P}_\epsilon\Big(\|r_t(K_n)\AY\|_n^2> \frac{\sigma^2}{n}\operatorname{tr}(L_nL_n^T)\Big).\label{ProofConcIneqVar}
\end{align}
By Lemma \ref{EqBasicIneqDP}, we have
\begin{align*}
    &\frac{\sigma^2}{n}\operatorname{tr}(L_nL_n^T)-\mathbf{E}_{\epsilon}\|r_t(K_n)\AY\|_n^2\geq \frac{\sigma^2}{n}\tildengeffdim(t)-\|r_t(K_n)\Af\|_n^2.
\end{align*}
Since $t>\Ltstar$ implies that
\[
\|r_t(K_n)\Af\|_n^2\leq \|r_{\Ltstar}(K_n)\Af\|_n^2= \frac{\sigma^2}{n}\tildengeffdim(\Ltstar),
\]
we arrive at
\begin{align*}
&\frac{\sigma^2}{n}\operatorname{tr}(L_nL_n^T)-\mathbf{E}_\epsilon\|r_t(K_n)\AY\|_n^2\geq  \frac{\sigma^2}{n}\tildengeffdim(t)-\frac{\sigma^2}{n}\tildengeffdim(\Ltstar) = \frac{\sigma^2}{n}y.
\end{align*}
Inserting this into \eqref{ProofConcIneqVar}, we get
\begin{align*}
\mathbf{P}_\epsilon(\tau> t)\leq \mathbf{P}_\epsilon\Big(\|r_t(K_n)\AY\|_n^2-\mathbf{E}_\epsilon\|r_t(K_n)\AY\|_n^2> \frac{\sigma^2}{n}y\Big).
\end{align*}
Applying Lemma \ref{LemResidualConc}, using also that $t> \Ltstar$ and \eqref{BdF} imply $\|r_t(K_n)\Af\|_n^2\leq \sigma^2n^{-1}\tildengeffdim(t) \leq \sigma^2n^{-1} \operatorname{tr}(L_nL_n^T)$, the claim follows.
\end{proof}

Our next main result is a deviation inequality
for the variance part.
\begin{proposition}\label{ConcIneqVar} Suppose that \eqref{SubGN} holds true. Then, for every $y>0$, we have
\begin{align*}
      &\mathbf{P}_\epsilon\Big(\|K_n^{1/2}g_{\tau}^{1/2}(K_n)\Ae\|^2_n>\frac{\sigma^2}{n}\tildengeffdim(\Ltstar)+\frac{\sigma^2}{n}2y\Big)\Big)\nonumber\\
      &\leq 3\exp\Big(-c\Big( y \wedge \frac{y^2}{\tr\paren{ L_nL_n^\top}}\Big)\Big)
\end{align*}
with constant $c>0$ depending only on $A$.
\end{proposition}
\begin{proof}[Proof of Proposition \ref{ConcIneqVar}]  
By Definition~\ref{Def}, the term $\|K_n^{1/2}g_t^{1/2}(K_n)\Ae\|^2_n$ is non-decreasing in $t\geq 0$.
Now, if 
\begin{align}\label{EqSimpleCase}
    \tildengeffdim(\Ltstar)+y>\tildengeffdim(T),
\end{align}
then
\begin{align*}
    &\mathbf{P}_\epsilon\Big(\|K_n^{1/2}g_{\tau}^{1/2}(K_n)\Ae\|^2_n>\frac{\sigma^2}{n}\tildengeffdim(\Ltstar)+2\frac{\sigma^2}{n}y\Big)\\
    &\leq \mathbf{P}_\epsilon\Big(\|K_n^{1/2}g_{T}^{1/2}(K_n)\Ae\|^2_n>\frac{\sigma^2}{n}\tildengeffdim(T)+\frac{\sigma^2}{n}y\Big),
\end{align*}
and the claim follows from Lemma \ref{LemVarHansonWright}. On the other hand, if \eqref{EqSimpleCase} does not hold, then we can define $\Ltstar<t\leq T$ by 
\begin{align}\label{EqDefy}
\tildengeffdim(t)=\tildengeffdim(\Ltstar)+y.
\end{align}
In this case we have
\begin{align*}
   &\mathbf{P}_\epsilon\Big(\|K_n^{1/2}g_{\tau}^{1/2}(K_n)\Ae\|^2_n>\frac{\sigma^2}{n}\tildengeffdim(\Ltstar)+2\frac{\sigma^2}{n}y\Big)\\
   &\leq \mathbf{P}_\epsilon\Big(\{\tau\leq t\}\cap \Big\{\|K_n^{1/2}g_{\tau}^{1/2}(K_n)\Ae\|^2_n>\frac{\sigma^2}{n}\tildengeffdim(\Ltstar)+2\frac{\sigma^2}{n}y\Big\}\Big)+\mathbf{P}_\epsilon(\tau> t)\\
   &\leq \mathbf{P}_\epsilon\Big(\|K_n^{1/2}g_{t}^{1/2}(K_n)\Ae\|^2_n>\frac{\sigma^2}{n}\tildengeffdim(t)+\frac{\sigma^2}{n}y\Big)+\mathbf{P}_\epsilon(\tau> t),
\end{align*}
and the claim follows from applying Lemma \ref{LemVarHansonWright} to the second last term and Proposition \ref{prop.deviation.SDP} to the last term, using that $t>\tildetstar$ and $y=\tildengeffdim(t)-\tildengeffdim(\Ltstar)$.
\end{proof}

\subsubsection{Deviation inequality for the bias part}
\begin{proposition}\label{ConcIneqBias} Suppose that Assumptions~\eqref{SubGN} and~\eqref{BdK} hold true. Then, for every $y>0$ such that $2\|r_{\Ltstar }(K_n)\Af\|_n^2+\sigma^2n^{-1}y>\|r_{T}(K_n)\Af\|_n^2$, we have
\begin{align}\label{EqConcIneqBias}
&\mathbf{P}_\epsilon\Big(\|r_\tau(K_n)\Af\|_n^2> 2\|r_{\Ltstar}(K_n)\Af\|_n^2+\frac{\sigma^2}{n}y\Big)\nonumber\\
&\leq 2\exp\Big(-c\Big(\frac{y^2}{\operatorname{tr}(L_nL_n^T)}\wedge y\Big)\Big)
\end{align}
with constant $c>0$ depending only on $A$.
\end{proposition}
\begin{proof}[Proof of Proposition \ref{ConcIneqBias}]
From $2\|r_{\Ltstar}(K_n)\Af\|_n^2+\sigma^2n^{-1}y>\|r_{T}(K_n)\Af\|_n^2$ it follows that, under the event considered in \eqref{EqConcIneqBias}, the stopping rule $\tau$ has to be smaller than $T$. This means that in the definition of $\tau$ in \eqref{EqDPPreSmooth}, we can ignore the minimum with $T$ in what follows. 

If $\|r_0(K_n)\Af\|_n^2<2\|r_{\Ltstar}(K_n)\Af\|_n^2+\sigma^2n^{-1}y$, then the claim is clear because the probability on the left-hand side of \eqref{EqConcIneqBias} is equal to zero. Otherwise, we define $0\leq t< \Ltstar$ by 
\begin{align*}
2\|r_{\Ltstar}(K_n)\Af\|_n^2+\frac{\sigma^2}{n}y=\|r_t(K_n)\Af\|_n^2,
\end{align*}
leading to
\begin{align}
&\mathbf{P}_\epsilon\Big(\|r_\tau(K_n)\Af\|_n^2> 2\|r_{\Ltstar}(K_n)\Af\|_n^2+\frac{\sigma^2}{n}y\Big)\nonumber\\
&\leq \mathbf{P}_\epsilon(\tau < t)\leq\mathbf{P}_\epsilon\Big(\|r_t(K_n)\AY\|_n^2\leq \frac{\sigma^2}{n}\operatorname{tr}(L_nL_n^T)\Big).\label{ProofConcIneqBias}
\end{align}
By Lemma \ref{EqBasicIneqDP}, we have
\begin{align*}
    &\frac{\sigma^2}{n}\operatorname{tr}(L_nL_n^T)-\mathbf{E}_{\epsilon}\|r_t(K_n)\AY\|_n^2\leq 2\frac{\sigma^2}{n}\tildengeffdim(t)-\|r_t(K_n)\Af\|_n^2.
\end{align*}
Since $t< \Ltstar$, \eqref{DeftstarL} implies 
\begin{align*}
2\frac{\sigma^2}{n}\tildengeffdim(t)&\leq 2\frac{\sigma^2}{n}\tildengeffdim(\Ltstar)=2\|r_{\Ltstar}(K_n)\Af\|_n^2.
\end{align*}
Thus we get
\begin{align*}
&\frac{\sigma^2}{n}\operatorname{tr}(L_nL_n^T)-\mathbf{E}_\epsilon\|r_t(K_n)\AY\|_n^2\leq 2\|r_{\Ltstar}(K_n)\Af\|_n^2-\|r_t(K_n)\Af\|_n^2=-\frac{\sigma^2}{n}y.
\end{align*}
Inserting this into \eqref{ProofConcIneqBias}, we get
\begin{align*}
\mathbf{P}_\epsilon\Big(\|r_\tau(K_n)\Af\|_n^2> \frac{\sigma^2}{n}y\Big)\leq \mathbf{P}_\epsilon\Big(\|r_t(K_n)\AY\|_n^2-\mathbf{E}_\epsilon\|r_t(K_n)\AY\|_n^2\leq  -\frac{\sigma^2}{n}y\Big),
\end{align*}
Using that 
\begin{align*}
    \|r_t(K_n)\Af\|_n^2&=2\|r_{\Ltstar}(K_n)\Af\|_n^2+\frac{\sigma^2}{n}y\\
    &=2\frac{\sigma^2}{n}\tildengeffdim(\Ltstar)+\frac{\sigma^2}{n}y\leq 2\frac{\sigma^2}{n}\operatorname{tr}(L_nL_n^T)+\frac{\sigma^2}{n}y,
\end{align*}
the claim now follows from Lemma \ref{LemResidualConc}.
\end{proof}

\subsection{Proofs of oracle inequalities (fixed-design)}\label{sec.proof.oracle.FD}
The present section gathers proofs of main oracle inequalities established in the fixed-design setting. They mainly follow from the results from Section~\ref{sec.deviation.inequalities}. In each of these proofs, notations are used according to the context where the theorem has been stated.
\begin{proof}[Proof of Proposition~\ref{ThmCondDesignProxy}]
The proof follows from Sections~\ref{preliminary.results} and~\ref{sec.deviation.inequalities} applied with $L_n=I_n$, in which case $\DP$ coincides with $\tau$ from \eqref{EqDPPreSmooth} and $\tstar$ coincides with $\tildetstar$ from \eqref{DeftstarL}.

By \eqref{BdF} and using that $(a+b)^2\leq 2a^2+2b^2$, we have
\begin{align}
\|\mathbf{f}-\hat{\mathbf{f}}^{(\DP)}\|_n^2 & \leq 2\|r_{\DP}(K_n)\mathbf{f}\|_n^2
+2\|K_ng_{\DP}(K_n)\boldsymbol{\epsilon}\|_n^2\nonumber\\
&\leq 2\|r_{\DP}(K_n)\mathbf{f}\|_n^2
+2\|K_n^{1/2}g_{\DP}^{1/2}(K_n)\boldsymbol{\epsilon}\|_n^2.\label{EqThmCondDesignProxy1}
\end{align}
Proposition~\ref{ConcIneqVar} with $L_n=I_n$ yields that, for every $u>0$,
\begin{multline}
    \mathbf{P}_\epsilon\Big(\|K_n^{1/2}g^{1/2}_{\DP}(K_n)\boldsymbol{\epsilon}\|_n^2>\frac{\sigma^2}{n}\mathcal{N}_n^g(\tstar)+C\Big(\frac{\sigma^2\sqrt{u}}{\sqrt{n}}+\frac{\sigma^2u}{n}\Big)\Big)\leq 3e^{-u}. \label{EqThmCondDesignProxy3}
\end{multline}
On the other hand, from Proposition~\ref{ConcIneqBias} with $L_n=I_n$, it follows that 
\begin{align}
    &\mathbf{P}_\epsilon\Big(\|r_{\DP}(K_n)\mathbf{f}\|_n^2>2\|r_{\tstar\wedge T}(K_n)\mathbf{f}\|_n^2+C\Big(\frac{\sigma^2\sqrt{u}}{\sqrt{n}}+\frac{\sigma^2u}{n}\Big)\Big)\leq 2e^{-u}.\label{EqThmCondDesignProxy2}
\end{align}
By the definition of $\tildetstar$, we have 
\begin{align}
&\|r_{\tildetstar\wedge T}(K_n)\mathbf{f}\|_n^2+\frac{\sigma^2}{n}\mathcal{N}_n^g(\tstar)\leq 2\min_{0\leq t\leq T}\Big\{\|r_t(K_n)\mathbf{f}\|_n^2+\frac{\sigma^2}{n}\mathcal{N}_n^g(t)\Big\}\label{Eqtstarbound1}.
\end{align}
Using \eqref{EqThmCondDesignProxy1} and \eqref{Eqtstarbound1} combined with \eqref{EqThmCondDesignProxy3} and \eqref{EqThmCondDesignProxy2}, and the union bound, the claim now follows.
\end{proof}

\begin{proof}[Proof of Theorem \ref{ThmCondDesignOracle}]
The claim follows from inserting Lemma~\ref{LemProxyVariance} into Theorem \ref{ThmCondDesignOracle}.
\end{proof}

\begin{proof}[Proof of Theorem \ref{ThmCondDesignProxySmoothedNorm}]
The result follows from Sections~\ref{preliminary.results} and~\ref{sec.deviation.inequalities} applied with $L_n=\tilde{g}_T^{1/2}(K_n)K_n^{1/2}$, in which case $\SDP$ from~\eqref{EqSDP} coincides with $\tau$ from \eqref{EqDPPreSmooth}. 

A key remark is that, since the regularizer $\tilde{g}$ satisfies \eqref{LFL}, we have $\lambda g_T(\lambda)\leq (B\vee 1)b^{-1}\lambda\tilde{g}_T(\lambda)$. Thus $\SDP\leq T$ implies
\begin{align}
&\|\mathbf{f}-\hat{\mathbf{f}}^{(\SDP)}\|_n^2\nonumber\\
&\leq 2\|r_{\SDP}(K_n)\mathbf{f}\|_n^2+2\|K_ng_{\SDP}(K_n)\boldsymbol{\epsilon}\|_n^2\nonumber\\
&\leq 2\|r_{\SDP}(K_n)\mathbf{f}\|_n^2
+ 2 (B\vee 1)b^{-1}\|K_n^{1/2}g_{\SDP}^{1/2}(K_n)K_n^{1/2}\tilde{g}_T^{1/2}(K_n)\boldsymbol{\epsilon}\|_n^2\nonumber\\
&=2\|r_{\SDP}(K_n)\mathbf{f}\|_n^2 + 2(B\vee 1)b^{-1}\|K_n^{1/2}g_{\SDP}^{1/2}(K_n)\Ae\|_n^2,\label{EqThmCondDesignProxySmoothedNorm}
\end{align} 
where $\boldsymbol{\epsilon}$ has been replaced by $\Ae$ in the last inequality. Invoking the first claim of Lemma~\ref{LemChangeBias} we get
\begin{align}
&\|\mathbf{f}-\hat{\mathbf{f}}^{(\SDP)}\|_n^2\label{EqThmCondDesignProxySmoothedNorm1}\\&\leq C\Big(\|r_{\SDP}(K_n)\Af\|_n^2 + \|K_n^{1/2}g_{\SDP}^{1/2}(K_n)\Ae\|_n^2+\frac{1}{T^{2s+1}}+\frac{\|\Sigma_n-\Sigma\|_{\operatorname{op}}^{2\wedge 2s}}{T}\Big),\nonumber
\end{align} 
where the last term $CT^{-1}\|\Sigma_n-\Sigma\|_{\operatorname{op}}^{2\wedge 2s}$ can be dropped if $s\leq 1/2$.
On the one hand, Proposition~\ref{ConcIneqBias} with $L_n=\tilde{g}_T^{1/2}(K_n)K_n^{1/2}$ and Lemma  yields that
\begin{multline}
    \mathbf{P}_\epsilon\Big(\|r_{\SDP}(K_n)\Af\|_n^2>2\|r_{\tildetstar\wedge T}(K_n)\Af\|_n^2+C\frac{\sigma^2}{n}\Big(\sqrt{u\effdim^{\tilde g}_n(T)}+u\Big)\Big)\\\leq 2e^{-u},\quad u>0.\label{EqThmCondDesignProxySmoothedNorm2}
\end{multline}
On the other hand, from Proposition \ref{ConcIneqVar} with $L_n=\tilde{g}_T^{1/2}(K_n)K_n^{1/2}$, we get
\begin{multline}
    \mathbf{P}_\epsilon\Big(\|K_n^{1/2}g^{1/2}_{\SDP}(K_n)\Ae\|_n^2 >\frac{\sigma^2}{n} \tildengeffdim(\tildetstar)+C\frac{\sigma^2}{n}\Big(\sqrt{u\effdim^{\tilde g}_n(T)}+u\Big)\Big)\\\leq 3e^{-u},\quad u>0.\label{EqThmCondDesignProxySmoothedNorm3}
\end{multline}
The definition of $\tildetstar$ and \eqref{BdF} lead to
\begin{align*}
&\|r_{\tildetstar\wedge T}(K_n)\Af\|_n^2+\frac{\sigma^2}{n}\tildengeffdim(\tildetstar)\leq 2\min_{0\leq t\leq T}\Big\{\|r_t(K_n)\Af\|_n^2+\frac{\sigma^2}{n}\tildengeffdim(t)\Big\}.
\end{align*}
Now, using \eqref{BdF}, we have $ \tildengeffdim(t) \leq \geffdim_n(t)$ and $\|r_t(K_n)\Af\|_n^2 \leq \norm{r_t(K_n)\mathbf{f}}_n^2$.
Thus combining everything together yields
\begin{align}
\|r_{\tildetstar\wedge T}(K_n)\Af\|_n^2+\frac{\sigma^2}{n}\tildengeffdim(\tildetstar) 
    \leq 2\min_{0\leq t \leq T}\Big\{\|r_t(K_n)\mathbf{f}\|_n^2+\frac{\sigma^2}{n}\mathcal{N}^{ g}_n(t)\Big\}\label{Eqtstarbound2}.
\end{align}
Using \eqref{EqThmCondDesignProxySmoothedNorm1} and \eqref{Eqtstarbound2} combined with \eqref{EqThmCondDesignProxySmoothedNorm2} and \eqref{EqThmCondDesignProxySmoothedNorm3}, and the union bound, we get for every $u>0$
\begin{multline}\label{EqCondDesignProxySmoothedNorm}
\mathbf{P}_{\epsilon}\Big(\|\mathbf{f}-\hat {\mathbf{f}}^{(\tauSDP)}\|_n^2>C\Big(\min_{0< t\leq T}\Big\{\|r_t(K_n)\mathbf{f}\|_n^2+\frac{\sigma^2}{n}\mathcal{N}_n(t)\Big\}\\
+\frac{\sigma^2\sqrt{u\mathcal{N}_n^{\tilde{g}}(T)}} {n} +\frac{\sigma^2u}{n}+\frac{1}{T^{2s+1}}+\frac{\|\Sigma_n-\Sigma\|_{\operatorname{op}}^{2\wedge 2s}}{T}\Big)\Big)\leq 5e^{-u}
\end{multline}
The desired inequality now follows from inserting Lemmas \ref{LemEffDimGenDim} and \ref{LemProxyVariance}.
\end{proof}

\subsection{Key technical results}
In order to prove Propositions~\ref{ConcIneqVar} and~\ref{ConcIneqBias}, we need the following two concentration inequalities, namely Lemmas~\ref{LemResidualConc} and~\ref{LemVarHansonWright}.

\begin{lemma}\label{LemResidualConc} Suppose that Assumption \eqref{SubGN} holds. Then, for every $t\geq 0$ and every $y>0$, we have
\begin{align*}
&\mathbf{P}_\epsilon(\|r_t(K_n)\AY\|_n^2-\mathbf{E}_\epsilon\|r_t(K_n)\AY\|_n^2>  y)\\
&\leq \exp\Big(-c\Big(\frac{n^2y^2}{\subgauss^4\operatorname{tr}(L_nL_n^T)}\wedge\frac{ny}{\subgauss^2}\Big)\Big)+\exp\bigg(-\frac{cny^2}{\subgauss^2\|r_t(K_n)\Af\|_n^2}\bigg)
\end{align*}
and the same upper bound holds for $\mathbf{P}_\epsilon(\|r_t(K_n)\AY\|_n^2-\mathbf{E}_\epsilon\|r_t(K_n)\AY\|_n^2<-  y)$.
\end{lemma}

\begin{proof}[Proof of Lemma \ref{LemResidualConc}]
We have 
\begin{align*}
&\|r_t(K_n)\AY\|_n^2=\|r_t(K_n)\Af\|_n^2+\langle r_t(K_n)\Af,r_t(K_n)\Ae\rangle_n+\|r_t(K_n)\Ae\|_n^2
\end{align*}
and thus
\begin{align*}
&\|r_t(K_n)\AY\|_n^2-\mathbf{E}_\epsilon\|r_t(K_n)\AY\|_n^2\\
&=\langle L_n^Tr_t^2(K_n)\Af,\boldsymbol{\epsilon}\rangle_n+\|r_t(K_n)L_n\boldsymbol{\epsilon}\|_n^2-\mathbf{E}_\epsilon\|r_t(K_n)L_n\boldsymbol{\epsilon}\|_n^2.
\end{align*}
By \eqref{SubGN} and a general Hoeffding inequality for sub-Gaussian random variables (cf. \cite[Theorem 2.6.3]{MR3837109}), we have for all $y>0$, 
\begin{align*}
\mathbf{P}_\epsilon(\langle L_n^Tr_t^2(K_n)\Af,\boldsymbol{\epsilon}\rangle_n>y)&\leq \exp\bigg(-\frac{cn^2y^2}{\sigma^2\|L_n^Tr_t^2(K_n)\Af\|_2^2}\bigg)\\
&\leq \exp\bigg(-\frac{cny^2}{\sigma^2\|r_t(K_n)\Af\|_n^2}\bigg),
\end{align*}
where we used the fact that $\|L_n^T\|_{\operatorname{op}}=\|L_n\|_{\operatorname{op}}\leq 1$ and \eqref{BdF} in the second inequality.
Moreover, an application of the Hanson-Wright inequality (cf. \cite[Theorem 6.2.1]{MR3837109}) gives for all $y>0$,
\begin{align*}
&\mathbf{P}_\epsilon(\|r_t(K_n)L_n\boldsymbol{\epsilon}\|_n^2-\mathbf{E}_\epsilon\|r_t(K_n)L_n\boldsymbol{\epsilon}\|_n^2>y)\\
&\leq \exp\bigg(-c\bigg(\frac{n^2y^2}{\sigma^4\|L_n^Tr_t^2(K_n)L_n\|_{\operatorname{HS}}^2}\wedge\frac{ny}{\sigma^2\|L_n^Tr_t^2(K_n)L_n\|_{\operatorname{op}}}\bigg)\bigg).
\end{align*}
By Assumption \eqref{BdF} and the fact that $\|L_n\|_{\operatorname{op}}=\|L_n^T\|_{\operatorname{op}}\leq 1$, we have 
\begin{align*}
    \|L_n^Tr_t^2(K_n)L_n\|_{\operatorname{op}}\leq 1\quad\text{and}\quad\|L_n^Tr_t^2(K_n)L_n\|_{\operatorname{HS}}^2\leq \|L_n\|_{\operatorname{HS}}^2=\operatorname{tr}(L_nL_n^T).
\end{align*} 
We thus obtain that
\begin{align*}
&\mathbf{P}_\epsilon(\|r_t(K_n)L_n\boldsymbol{\epsilon}\|_n^2-\mathbf{E}_\epsilon\|r_t(K_n)L_n\boldsymbol{\epsilon}\|_n^2>y)\\
&\leq \exp\Big(-c\Big(\frac{n^2y^2}{\subgauss^4\operatorname{tr}(L_nL_n^T)}\wedge\frac{ny}{\subgauss^2}\Big)\Big).
\end{align*}
This completes the proof of the right-deviation inequality. The left-deviation inequality follows analogously.
\end{proof}

\medskip

\begin{lemma}\label{LemVarHansonWright}
Suppose that Assumption~\eqref{SubGN} holds. Then, for every $t\geq 0$ and every $y>0$, we have
\begin{align*}
    \mathbf{P}_\epsilon(\|K_n^{1/2}g_t^{1/2}(K_n)\Ae\|^2_n>\tildengeffdim(t)+y)
    &\leq \exp\Big(-c\Big(\frac{n^2y^2}{\subgauss^4\tildengeffdim(t)}\wedge\frac{ny}{\subgauss^2}\Big)\Big) \\
    &\leq\exp\Big(-c\Big(\frac{n^2y^2}{\subgauss^4\operatorname{tr}(L_nL_n^T)}\wedge\frac{ny}{\subgauss^2}\Big)\Big).
\end{align*}
\end{lemma}

\begin{proof}[Proof of Lemma \ref{LemVarHansonWright}]
First, note that $\tildengeffdim(t)=\mathbf{E}_\epsilon\|K_n^{1/2}g_t^{1/2}(K_n)\Ae\|^2_n$. Moreover, by the Hanson-Wright inequality (cf. \cite[Theorem 6.2.1]{MR3837109}), we have for all $y>0$,
\begin{align*}
    &\mathbf{P}_\epsilon\Big(\|K_n^{1/2}g_t^{1/2}(K_n)L_n\boldsymbol{\epsilon}\|^2_n>\mathbf{E}_\epsilon\|K_n^{1/2}g_t^{1/2}(K_n)L_n\boldsymbol{\epsilon}\|^2_n+y\Big)\\
    &\leq \exp\Big(-c\Big(\frac{n^2y^2}{\subgauss^4\|L_n^TK_ng_t(K_n)L_n\|_{\operatorname{HS}}^2}\wedge\frac{ny}{\subgauss^2\|L_n^TK_ng_t(K_n)L_n\|_{\operatorname{op}}} \Big)\Big).
\end{align*}
The claims now follow from inserting $\|L_n^TK_ng_t(K_n)L_n\|_{\operatorname{op}}\leq 1$ as well as $\|L_n^TK_ng_t(K_n)L_n\|_{\operatorname{HS}}^2\leq \operatorname{tr}(L_nL_n^TK_ng_t(K_n))\leq \operatorname{tr}(L_nL_n^T)$.
\end{proof}

\medskip

\begin{lemma}\label{LemChangeBias}
Let $L_n=\tilde{g}_T^{1/2}(K_n)K_n^{1/2}$ with regularizer $\tilde{g}$ satisfying \eqref{LFL}. If \eqref{Assume.SC} holds with $s=r-1/2\geq 0$ and if $\|(\Sigma+T^{-1})^{-1/2}(\Sigma_n-\Sigma)(\Sigma+T^{-1})^{-1/2}\|_{\operatorname{op}}\leq 1/2$, then there is a constant $C>0$ depending only on $s$, $R$ and $M$ such that for every $0<t\leq T$,
\begin{align*}
    \|r_t(K_n)\mathbf{f}\|_n^2\leq  \frac{1}{b}\|r_t(K_n)\Af\|_n^2+C\Big(\frac{1}{T^{2s+1}}+\frac{\|\Sigma_n-\Sigma\|_{\operatorname{op}}^{2\wedge 2s}}{T}\Big),
\end{align*}
where the last term in the upper bound  $CT^{-1}\|\Sigma_n-\Sigma\|_{\operatorname{op}}^{2\wedge 2s}$ can be dropped if $s\leq 1/2$. Moreover, if \eqref{Assume.SC} and \eqref{Qualif} hold with $s=r-1/2\geq 0$ and $r\geq q$ and if $\|(\Sigma+T^{-1})^{-1/2}(\Sigma_n-\Sigma)(\Sigma+T^{-1})^{-1/2}\|_{\operatorname{op}}\leq 1/2$, then we have for every $0<t\leq T$,
\begin{align*}
    \|r_t(K_n)\mathbf{f}\|_n^2\leq  C\Big(\frac{1}{t^{2s+1}}+\frac{\|\Sigma_n-\Sigma\|_{\operatorname{op}}^{2\wedge 2s}}{t}\Big),
\end{align*}
where the second term in the upper bound can be dropped if $s\leq 1/2$.
\end{lemma}

\begin{proof}
Using the identity $\mathbf{f}=S_nf$ and the singular value decomposition in \eqref{EqSVDSamplingOp}, we have
\begin{align*}
    \|r_{t}(K_n)\mathbf{f}\|_n^2 & =\sum_{j\geq 1}\hat\lambda_j r_{t}^2(\hat\lambda_j)\langle f,\hat u_j\rangle^2\\
   &\leq \frac{1}{b}\sum_{\hat\lambda_jT> 1} \lambda_j\tilde{g}_T(\hat\lambda_j)\hat\lambda_jr_{t}^2(\hat\lambda_j)\langle f,\hat u_j\rangle^2+\frac{1}{T}\sum_{\hat\lambda_jT\leq 1}\langle f,\hat u_j\rangle^2\\
   &= \frac{1}{b}\|r_t(K_n)\Af\|_n^2 + \frac{1}{T}\sum_{\hat\lambda_jT\leq 1}\langle f,\hat u_j\rangle^2,
\end{align*}
where we applied \eqref{LFL} and \eqref{BdF} in the inequality. To see the first claim, we have show that
\begin{align}\label{EqChangeBias2}
    \sum_{\hat\lambda_jT\leq 1}\langle f,\hat u_j\rangle^2\leq C(T^{-2s}+\|\Sigma_n-\Sigma\|_{\operatorname{op}}^{2\wedge 2s}),
\end{align}
where the second term $\|\Sigma_n-\Sigma\|_{\operatorname{op}}^{2\wedge 2s}$ can be dropped if $s\leq 1/2$.
By assumption $\|\Sigma-\Sigma_n\|_{\operatorname{op}}\leq (\lambda_1+T^{-1})/2$. By assumption, we have $f=\Sigma^sg$ with $\|g\|_{\mathcal{H}}\leq R$ and $s=r-1/2\geq 0$. Hence,
\begin{align*}
\sum_{\hat\lambda_jT\leq 1}\langle f,\hat u_j\rangle^2&\leq 2\sum_{\hat\lambda_jT< 1}\langle  \Sigma_n^sg,\hat u_j\rangle^2+2\sum_{\hat\lambda_jT< 1}\langle (\Sigma^s-\Sigma_n^s)g,\hat u_j\rangle^2\\
&\leq 2\sum_{\hat\lambda_jT< 1}\hat\lambda_j^{2s}\langle  g,\hat u_j\rangle^2+2\|( \Sigma^s-\Sigma_n^s)g\|_{\mathcal{H}}^2\\
&\leq 2T^{-2s}\|g\|_{\mathcal{H}}^2+C\|\Sigma-\Sigma_n\|_{\operatorname{op}}^{2\wedge 2s}\|g\|_{\mathcal{H}}^2,
\end{align*}
where we applied \eqref{EqPowerOperatorNorm1} and \eqref{EqPowerOperatorNorm2} in the last inequality and where $C>0$ is a constant depending only on $s$ and $M$. If $s\leq 1/2$, then we have
\begin{align*}
    &\sum_{\hat\lambda_jT\leq 1}\langle f,\hat u_j\rangle^2= \sum_{\hat\lambda_jT< 1}\langle (\Sigma_n+T^{-1})^{s}(\Sigma_n+T^{-1})^{-s}\Sigma^{s} g,\hat u_j\rangle^2\\
    &\leq (2T^{-1})^{2s}\|(\Sigma_n+T^{-1})^{-s}\Sigma^{s} \|_{\operatorname{op}}^2R^2\leq (2T^{-1})^{2s}\|(\Sigma_n+T^{-1})^{-1/2}\Sigma^{1/2} \|_{\operatorname{op}}^{2s}R^2,
\end{align*}
where we applied \eqref{EqPowerOperatorNorm3} in the last inequality and where $C>0$ is a constant depending only on $s$, $M$ and $R$. Hence, the second part of the claim follows from 
\begin{align}
   &\|(\Sigma_n+T^{-1})^{-1/2}\Sigma^{1/2} \|_{\operatorname{op}}^2\leq\|(\Sigma_n+T^{-1})^{-1/2}(\Sigma+T^{-1})^{1/2}\|_{\operatorname{op}}^2\nonumber\\
   &= \|(\Sigma+T^{-1})^{1/2}(\Sigma_n+T^{-1})^{-1}(\Sigma+T^{-1})^{1/2}\|_{\operatorname{op}}\nonumber\\
   &=\|((\Sigma+T^{-1})^{-1/2}(\Sigma_n-\Sigma)(\Sigma+T^{-1})^{-1/2}+1)^{-1}\|_{\operatorname{op}}\leq 2\label{EqChangeBias3}.
\end{align}
The proof of the last claim is very similar. Using \eqref{Qualif} with $r\geq q$, \eqref{EqPowerOperatorNorm1} and \eqref{EqPowerOperatorNorm2}, we get
\begin{align*}
   \|\Sigma_n^{1/2}r_t(\Sigma_n)f\|_{\mathcal{H}}^2 &\leq 2\|\Sigma_n^{1/2}r_t(\Sigma_n)\Sigma_n^sg\|_{\mathcal{H}}^2+2\|\Sigma_n^{1/2}r_t(\Sigma_n)(\Sigma_n^s-\Sigma^s)g\|_{\mathcal{H}}^2 \\
   &\leq C(t^{-1-2s}+t^{-1}\|\Sigma-\Sigma_n\|_{\operatorname{op}}^{2\wedge 2s}),
\end{align*}
and the second part of the last claim follows. On the other hand, if $s\leq 1/2$, then we have
\begin{align*}
    &\|\Sigma_n^{1/2}r_t(\Sigma_n)f\|_{\mathcal{H}}^2\leq \|\Sigma_n^{1/2}r_t(\Sigma_n)(\Sigma_n+T^{-1})^{s}(\Sigma_n+T^{-1})^{-s}\Sigma^{s} g\|_{\mathcal{H}}^2\\ &\leq C_1\|\Sigma_n^{1/2}r_t(\Sigma_n)(\Sigma_n+t^{-1})^{s}\|_{\operatorname{op}}^2\|(\Sigma_n+t^{-1})^{-1/2}(\Sigma+t^{-1})^{1/2}\|_{\operatorname{op}}^{2s} \leq C_2t^{-1-2s},
\end{align*}
where we applied \eqref{Qualif} and \eqref{EqChangeBias3}.
\end{proof}

\section{Proofs for random design results}\label{SecProofRandomDesign}

\subsection{Concentration inequalities}
In this section, we provide concentration and deviation inequalities needed to transfer our results from the fixed to the random design setting.
We start with a deviation inequality dealing with the change of norm event from Lemma \ref{LemChangeNorm1.main}. 
The next lemma follows from an extension of \citet{T15} obtained in \citet{MR3648301} and further simplified by \citet{MR3629418} (see Lemma~\ref{LemConcIneqTropp}).
\begin{lemma}\label{LemConcIneqBoundedK}
Suppose that \eqref{BdK} holds. For $t>0$, let $\mathcal{E}_t$ be the event defined by
\begin{align*}
    \mathcal{E}_t=\{\|(\Sigma+t^{-1})^{-1/2}(\Sigma_n-\Sigma)(\Sigma+t^{-1})^{-1/2}\|_{\operatorname{op}}\leq 1/2\}.
\end{align*} 
Then there are constants $c_1,c_2,C_1>0$ depending only on $M$ such that, for every $0< t \leq c_2n$, 
\begin{align*}
&\mathbb{P}(\mathcal{E}_t^c)\leq C_1t\exp(-c_1n/t).
\end{align*}

\end{lemma}

\begin{proof}[Proof of Lemma~\ref{LemConcIneqBoundedK}]
The proof consists in checking the assumptions of Lemma~\ref{LemConcIneqTropp} from the Appendix. This justifies introducing constants $R$, $V$, and $D$ from Lemma~\ref{LemConcIneqTropp}.
In particular
\[
\xi_i=(\Sigma+t^{-1})^{-1/2}k_{X_i}\otimes (\Sigma+t^{-1})^{-1/2}k_{X_i}-(\Sigma+t^{-1})^{-1}\Sigma.
\]
Then $\|\xi_1\|_{\operatorname{op}}\leq 2\|(\Sigma+t^{-1})^{-1/2}k_{X_1}\|_{\mathcal{H}}^2 \leq 2M^2t = R$. 
Moreover, we have
\begin{align*}
    \|\mathbb{E}\xi_1^2\|_{\operatorname{op}}&\leq \Big\|\mathbb{E}\Big((\Sigma+t^{-1})^{-1/2}k_{X}\otimes (\Sigma+t^{-1})^{-1/2}k_{X}\Big)^2\Big\|_{\operatorname{op}}\\
    &\leq \Big\|\mathbb{E}\langle (\Sigma+t^{-1})^{-1}k_{X_1},k_{X}\rangle_{\mathcal{H}}(\Sigma+t^{-1})^{-1/2}k_{X}\otimes (\Sigma+t^{-1})^{-1/2}k_{X}\Big\|_{\operatorname{op}}\\
    &\leq tM^2\Big\|\mathbb{E}(\Sigma+t^{-1})^{-1/2}k_{X}\otimes (\Sigma+t^{-1})^{-1/2}k_{X}\Big\|_{\operatorname{op}}\\
    &=tM^2\Big\|(\Sigma+t^{-1})^{-1}\Sigma\Big\|_{\operatorname{op}}\leq tM^2 = V.
\end{align*}
Similarly with $D = \mathcal{N}(t)$, we have 
\begin{align*}
    \operatorname{tr}(\mathbb{E}\xi_1^2)&\leq tM^2\operatorname{tr}((\Sigma+t^{-1})^{-1}\Sigma)=tM^2 \mathcal{N}(t) = V \cdot D .
\end{align*}
%
%
Then, for every $t>0$ such that $V^{1/2}n^{-1/2} + (3n)^{-1}R\leq 1/2$,
\begin{align*}
\P\croch{ \norm{\frac{1}{n}\sum_{i=1}^n \xi_i}_{\operatorname{op}} \geq \frac{1}{2} } 
& \leq 4 \mathcal{N}(t) \exp\croch{ -\frac{n}{8 \paren{ M^2 + (2/6)M^2 } t } } \\
& \leq 4 M^2 t \exp\croch{ -\frac{n}{(32/3)M^2 t } },
\end{align*}
where the last inequality results from \eqref{BdK}, which yields the claim with $C_1=4M^2$, $c_1=(32/3)M^2$, and $c_2 = (3/4)^2(\sqrt{7/3}-1)^2/M^2$.
\end{proof}
Next, we establish a concentration inequality for the empirical effective dimension. Interestingly, the event $\mathcal{E}_T$ again plays a key role.

\begin{lemma}\label{LemConcEffDim}
Suppose that \eqref{BdK} holds. Then there is a constant $C$ depending only on $M$ and $\lambda_1^{-1}$ such that, for every $1\leq t \leq T$, 
\begin{align*}
     \mathbb{P}\Big(\mathcal{E}_T\cap\Big\{\mathcal{N}_n(t)> C\mathcal{N}(t)\Big\}\Big)\leq e^{-n/t}.
\end{align*}
In particular, for every $1\leq t \leq T$, we have
\begin{align*}
    \mathbb{E}\1_{\mathcal{E}_T}\mathcal{N}_n(t)\leq C\mathcal{N}(t)+ne^{-n/t}.
\end{align*}
\end{lemma}
\begin{remark}\label{RemConcEffDim}
Lemma \ref{LemConcEffDim} deals only with the case $t\geq 1$. The reason for this is that for $0<t\leq 1$, the trivial bound $\mathcal{N}_n(t)\leq M^2t\leq M^2$ will be sufficient for our purposes.
\end{remark}
\begin{proof}[Proof of Lemma~\ref{LemConcEffDim}]
Setting 
\begin{align}\label{EqDefAT}
    A_{t}=(\Sigma+t^{-1})^{-1/2}(\Sigma_n-\Sigma)(\Sigma+t^{-1})^{-1/2},
\end{align}
we have
\begin{align*}
   &(\Sigma_n+t^{-1})^{-1}=(\Sigma+t^{-1})^{-1/2}(I+A_{t})^{-1}(\Sigma+t^{-1})^{-1/2}.
\end{align*}
Hence, 
\begin{align*}
        \mathcal{N}_n(t) & = \operatorname{tr}( \Sigma_n(\Sigma_n+t^{-1})^{-1}) \\
        & = \operatorname{tr}\croch{ \Sigma_n (\Sigma+t^{-1})^{-1/2}(I+A_{t})^{-1}(\Sigma+t^{-1})^{-1/2} }  \\
        & = \operatorname{tr}\croch{ (\Sigma+t^{-1})^{-1/2} \Sigma_n (\Sigma+t^{-1})^{-1/2}(I+A_{t})^{-1} }.
\end{align*}
Since $\mathcal{E}_T$ holds and $t\leq T$, we have $\norm{A_t}_{\operatorname{op}}\leq \norm{A_T}_{\operatorname{op}}\leq 1/2$
by using 
\begin{align*}
    A_{t}
    &=(\Sigma+t^{-1})^{-1/2}(\Sigma+T^{-1})^{1/2}A_T(\Sigma+T^{-1})^{1/2}(\Sigma+t^{-1})^{-1/2},
\end{align*}
which implies that $\|(I+A_t)^{-1}\|_{\operatorname{op}}\leq 2$.

Then, the von Neumann trace inequality applied to non-negative symmetric operators on the event $\mathcal{E}_T$ leads to
\begin{align}
        \mathcal{N}_n(t) & \leq \norm{(I+A_{t})^{-1}}_{\operatorname{op}} \operatorname{tr}\croch{ (\Sigma+t^{-1})^{-1/2} \Sigma_n (\Sigma+t^{-1})^{-1/2} }\nonumber\\
        & \leq 2 \operatorname{tr}\croch{ (\Sigma+t^{-1})^{-1/2} \Sigma_n (\Sigma+t^{-1})^{-1/2} }\nonumber\\
        & \leq 2 \croch{ \mathcal{N}(t) + \operatorname{tr}(A_{t}) }.\label{EqBoundEffDimAT}
\end{align}
Using the definition of the empirical covariance operator, we have
\begin{align*}
   \operatorname{tr}(A_{t})=\frac{1}{n}\sum_{i=1}^n\|(\Sigma+t^{-1})^{-1/2}k_{X_i}\|_{\mathcal{H}}^2-\mathbb{E}\|(\Sigma+t^{-1})^{-1/2}k_X\|_{\mathcal{H}}^2.
\end{align*}
In addition since $\|(\Sigma+t^{-1})^{-1/2}k_{X_1}\|_{\mathcal{H}}^2\leq M^2t$, and $\mathbb{E}\|(\Sigma+t^{-1})^{-1/2}k_{X_1}\|_{\mathcal{H}}^4\leq M^2t\mathcal{N}(t)$, Bernstein's inequality (cf. Theorem 2.10 in \cite{MR3185193})  yields
\begin{align*}
    \mathbb{P}\Big( \operatorname{tr}(A_{t})>\sqrt{ \frac{2uM^2t\mathcal{N}(t)}{n}}+\frac{M^2}{3}\frac{ut}{ n}\Big)\leq e^{-u}.
\end{align*}
Inserting
\begin{align}\label{EqSimplerConcIneq}
\sqrt{\frac{2uM^2t\mathcal{N}(t)}{n}}\leq \mathcal{N}(t)+M^2\frac{tu}{n},
\end{align}
we get for every $u>0$,
\begin{align*}
    \mathbb{P}\Big( \operatorname{tr}(A_{t})>\mathcal{N}(t)+\frac{4M^2}{3}\frac{ut}{ n}\Big)\leq e^{-u}.
\end{align*}
Finally setting $u=n/t$ and using $\mathcal{N}(t)\geq \lambda_1/(\lambda_1+1)$ for $t\geq 1$, it results
\begin{align*}
    \mathbb{P}\Big( \operatorname{tr}(A_{t})>\Big(1+\frac{4M^2}{3}\Big(1+\frac{1}{\lambda_1}\Big)\Big)\mathcal{N}(t)\Big)\leq e^{-n/t}.
\end{align*}
Combining this with \eqref{EqBoundEffDimAT}, the first claim follows with $C=4(1+2(1+\lambda_1^{-1})M^2/3)$. The second claim follows from the first one, using also that $\mathcal{N}_n(t)\leq n$.
\end{proof}
Finally, we establish the following deviation bound for remainder traces.
\begin{lemma}\label{LemConcIneqTraces}
Suppose that \eqref{BdK} holds. Then, for each $u>0$ and any $0\leq k \leq n$, we have
\begin{align*}
\mathbb{P}\Big(\sum_{j>k}\hat\lambda_j>2\sum_{j>k}\lambda_j+2M^2\frac{u}{n}\Big)\leq e^{-u}.
\end{align*}
In particular, defining
\begin{align*}
\mathcal{A}(t,K)=\Big\{\forall 0\leq k\leq K:\sum_{j>k}\hat\lambda_j\leq 2\sum_{j>k}\lambda_j+2M^2\Big(\frac{1}{t}+\frac{\log (K+1)}{n}\Big)\Big\}
\end{align*}
with $0\leq K \leq n$ and $t>0$, we have
\begin{align*}
    \mathbb{P}(\mathcal{A}(t,K))\geq 1-e^{-n/t}.
\end{align*}
\end{lemma}
\begin{proof}[Proof of Lemma~\ref{LemConcIneqTraces}]
Let $\Pi_k$ be the orthogonal projection from $\mathcal{H}$ onto the span of the (population) eigenvectors $(u_j:j>k)$. Then, by the variational characterization of partial traces, we have $\sum_{j>k}\lambda_j= \operatorname{tr}(\Pi_k\Sigma)$ and $\sum_{j>k}\hat\lambda_j\leq \operatorname{tr}(\Pi_k\hat\Sigma)$. We conclude that
\begin{align*}
\sum_{j>k}\hat\lambda_j-\sum_{j>k}\lambda_j\leq \operatorname{tr}(\Pi_k(\hat\Sigma -\Sigma)\Pi_k)=\frac{1}{n}\sum_{i=1}^n\|\Pi_kk_{X_i}\|_{\mathcal{H}}^2-\mathbb{E}\|\Pi_kk_{X}\|_{\mathcal{H}}^2.
\end{align*}
Since $\|\Pi_k k_{X_i}\|_{\mathcal{H}}^2\leq \|k_{X_i}\|_{\mathcal{H}}^2\leq M^2$, and $\mathbb{E}\|\Pi_k k_{X_i}\|_{\mathcal{H}}^4\leq M^2\mathbb{E}\|\Pi_k k_{X_i}\|_{\mathcal{H}}^2=M^2\sum_{j>k}\lambda_j$, Bernstein's inequality yields
\begin{align*}
\mathbb{P}\Big(\sum_{j>k}\hat\lambda_j>\sum_{j>k}\lambda_j+\sqrt{\frac{2uM^2(\sum_{j>k}\lambda_j)}{n}} +\frac{M^2}{n}u\Big)\leq e^{-u}.
\end{align*}
Inserting
\begin{align*}
\sqrt{\frac{2uM^2(\sum_{j>k}\lambda_j)}{n}}\leq \sum_{j>k}\lambda_j+\frac{M^2}{n}u,
\end{align*}
the first claim follows. The second claim follows from the first one with $u=n/t+\log(K+1)$ in combination with the union bound.
\end{proof}

\medskip

\subsection{Bounds for the variance and bias parts}
We also use the notation of Section \ref{secProofFixedDesign} with $L_n=\tilde{g}^{1/2}_T(K_n)K_n^{1/2}$. In particular, we abbreviate $\Af=\tilde{g}^{1/2}_T(K_n)K_n^{1/2}\mathbf{f}$ and $\Ae=\tilde{g}^{1/2}_T(K_n)K_n^{1/2}\boldsymbol{\epsilon}$. Moreover, we write $\tildengeffdim(t)=\operatorname{tr}(\tilde{g}_T(K_n)K_ng_{t}(K_n)K_n)$ for the smoothed $g$-effective dimension and $\tildetstar=\inf\{t\geq 1:\|r_t(K_n)\Af\|_n^2\leq\sigma^2n^{-1}\tildengeffdim(t)\}$ for the smoothed balancing stopping rule.

\subsubsection{A bound for the variance part}
\begin{proposition}\label{prop.variance.term.random}
Under the assumptions of Theorem~\ref{CorExtThmSDPBoundedK}, we have on the event $\mathcal{E}_T\cap \mathcal{A}(T,\lfloor M^2T\rfloor)$,
\begin{align*}
    \mathbf{P}_{\boldsymbol{\epsilon}}(\|S_\rho g_{\SDP}(\Sigma_n)S_n^*\boldsymbol{\epsilon}\|_\rho^2> y(u)) \leq 3e^{-u},\quad u>0,
\end{align*}
with 
\begin{align*}
    y(u) =C\frac{\sigma^2}{n}(\tildengeffdim(\tildetstar)+\sqrt{u\mathcal{N}_n(T)}+u+1).
\end{align*}
\end{proposition}
The proof of Proposition~\ref{prop.variance.term.random} will be based on a series of lemmas successively detailed in what follows.

\medskip

The following lemma provides a version of Assumption \eqref{Assume.EVBound} that is implied by the population variant \eqref{Assume.EffRank}.
\begin{lemma}\label{lem.EffRank.RandomEmpirical}
Suppose that \eqref{Assume.EffRank} and \eqref{BdK} hold. Let $T>0$ be such that $T\log(\lfloor M^2T\rfloor+1)\leq n$. Then, on the event $\mathcal{E}_T\cap \mathcal{A}(T,\lfloor M^2T\rfloor)$, we have
\begin{align*}
  \forall 0< t\leq T,\qquad  t\sum_{j:t\hat\lambda_j< 1}\hat\lambda_j\leq E(|\{j:t\hat\lambda_j\geq 1\}|\vee 1).
\end{align*}
with $E=6E'+4M^2$.
\end{lemma}

\begin{proof}[Proof of Lemma~\ref{lem.EffRank.RandomEmpirical}]
Firstly by \eqref{BdK} we have $k\hat\lambda_k\leq \sum_{j\leq k}\hat\lambda_j\leq\operatorname{tr}(\hat\Sigma)\leq M^2$ and thus $\hat\lambda_k\leq M^2k^{-1}$ for every $k\geq 1$. 

For $0< t\leq T$ define now $k\geq 0$ such that $t\hat\lambda_k\geq 1 > t\hat\lambda_{k+1}$ (with the convention that $k=0$ if $t\hat\lambda_1<1$). Then it follows from the above that $k\leq \lfloor M^2T\rfloor$.
Let us now consider the event $\mathcal{A}(T,\lfloor M^2T\rfloor)\cap \mathcal{E}_T$. We have
\begin{align}
t\sum_{j>k}\hat \lambda_j&\leq  2t\sum_{j>k}\lambda_j+2M^2+\frac{2M^2T\log(\lfloor M^2T\rfloor+1)}{n}\nonumber\\
&\leq 2tE'\lambda_{k+1}(k\vee 1)+4M^2,\label{Eq.EffRank.RandomEmpirical}
\end{align}
where we applied \eqref{Assume.EffRank} and $T\log(\lfloor M^2T\rfloor+1)\leq n$ in the second inequality. Using the lower bound in Lemma~\ref{LemEVConc}, we have $\lambda_{k+1}\leq 2\hat\lambda_{k+1}+1/T$. Inserting this into \eqref{Eq.EffRank.RandomEmpirical}, we get
\begin{align*}
    t\sum_{j>k}\hat \lambda_j\leq 4E'(k\vee 1)+ 2E'(k\vee 1)+4M^2\leq (6E'+4M^2)(k\vee 1),
\end{align*}
and the claim follows with $E=6E'+4M^2$.
\end{proof}

\medskip

\begin{lemma}\label{trace.g.effective.dimension.relaxed.EVBound} 
Suppose that \eqref{Assume.EffRank} and \eqref{BdK} hold. Let $T>0$ be such that $T\log(\lfloor M^2T\rfloor+1)\leq n$. Then, on the event $\mathcal{E}_T\cap \mathcal{A}(T,\lfloor M^2T\rfloor)$, we have
\begin{align*}
\forall 1\leq t\leq T,\qquad \geffdim_n(t)\leq C_1 \tildengeffdim(t)+C_2
\end{align*}
with $C=\tilde{b}^{-1}(1+b^{-1}EB)$, $C_2=BE$ and $E=6E'+4M^2$.
\end{lemma}
\begin{proof}[Proof of Lemma~\ref{trace.g.effective.dimension.relaxed.EVBound}]
If $t\hat\lambda_1<1$, then \eqref{LFU} and Lemma \ref{lem.EffRank.RandomEmpirical} imply
\begin{align*}
    \geffdim_n(t)\leq Bt\sum_{j\geq 1}\hat\lambda_j\leq BE,
\end{align*}
yielding the claim in this case. On the other hand, if $t\hat\lambda_1\geq 1$, then let $k\geq 1$ be defined by $t\hat\lambda_{k+1} < 1\leq t\hat\lambda_{k}$. By \eqref{EqVarBound}, we have
\begin{align}\label{EqChangeVariance}
\geffdim_n(t)\leq \sum_{j\leq k}\hat\lambda_jg_t(\hat\lambda_j)+Bt\sum_{j> k}\hat\lambda_j.
\end{align}
Now by the definition of $k$, Lemma \ref{lem.EffRank.RandomEmpirical} and \eqref{LFL}, we have
\begin{align*}
&t\sum_{j> k}\hat\lambda_j\leq Ek\leq Eb^{-1}\sum_{j\leq k}\hat\lambda_jg_t(\hat\lambda_j).
\end{align*}
Inserting this into \eqref{EqChangeVariance}, we get
\begin{align*}
\sum_{j=1}^n\hat\lambda_jg_t(\hat\lambda_j)&\leq C\sum_{j\leq k}\hat\lambda_jg_t(\hat\lambda_j)\leq \tilde{b}^{-1} C\sum_{j=1}^ng_t(\hat\lambda_j)\hat\lambda_j\tilde g_T(\hat\lambda_j)\hat\lambda_j
\end{align*}
with $C=(1+b^{-1}BE)$. 
\end{proof}

\begin{lemma}\label{deviation.ineq.variance.term.change.norm.SDP}
Suppose that \eqref{Assume.EffRank} and \eqref{BdK} hold. Let $T>0$ be such that $T\log(\lfloor M^2T\rfloor+1)\leq n$. Then, on the event $\mathcal{E}_T\cap \mathcal{A}(T,\lfloor M^2T\rfloor)$, we have
\begin{align*}
\bP\Big( \|K_n^{1/2}g^{1/2}_{\SDP}(K_n)\boldsymbol{\epsilon}\|_n^2 > \frac{\sigma^2}{n}\Lambda(y)\Big)\leq 3\exp\Big(-c\Big( y \wedge \frac{y^2}{\effdim_n(T)}\Big)\Big),\quad y>0,
\end{align*}
with
\begin{align*}
    \Lambda(y) & = C \tildengeffdim(\tildetstar) + 2y +BE,
\end{align*}
where $C=2(1+b^{-1}BE)$, $E=6E'+4M^2$ and $c>0$ is a constant depending only on $A$.
\end{lemma}
\begin{proof}[Proof of Lemma~\ref{deviation.ineq.variance.term.change.norm.SDP}]
If \eqref{EqSimpleCase} holds, that is if $\tildengeffdim(\tildetstar)+y>\tildengeffdim(T)$, then Lemma~\ref{trace.g.effective.dimension.relaxed.EVBound} implies that on $\mathcal{E}_T\cap \mathcal{A}(n/T,K)$,
\begin{align*}
    \Lambda(y)  \geq C \tildengeffdim(T) + y +BE \geq \geffdim_n(T) + y.
\end{align*}
Hence,
\begin{align*}
    &\bP\Big( \|K_n^{1/2}g^{1/2}_{\SDP}(K_n)\boldsymbol{\epsilon}\|_n^2 > \frac{\sigma^2}{n}\Lambda(y)\Big)\\
    &\leq \bP\Big( \|K_n^{1/2}g^{1/2}_{T}(K_n)\boldsymbol{\epsilon}\|_n^2 > \frac{\sigma^2}{n}\tildengeffdim(T) + \frac{\sigma^2}{n}y\Big)
\end{align*}
and the claim follows from Lemma \ref{LemVarHansonWright} and Lemma \ref{LemEffDimGenDim}. On the other hand, if \eqref{EqSimpleCase} does not hold, then we can define $\tildetstar<t\leq T$ by    $\tildengeffdim(t) = \tildengeffdim(\tildetstar)+y$.
On $\mathcal{E}_T\cap \mathcal{A}(n/T,K)$, Lemma~\ref{trace.g.effective.dimension.relaxed.EVBound} implies 
\begin{align*}
    \Lambda(y) & = C \tildengeffdim(t) + y +BE \geq \geffdim_n(t) + y.
\end{align*}
Hence,
\begin{align*}
& \bP\Big( \|K_n^{1/2}g^{1/2}_{\SDP}(K_n)\boldsymbol{\epsilon}\|_n^2 > \frac{\sigma^2}{n}\Lambda(y)\Big)\\
& \leq \bP\Big( \|K_n^{1/2}g^{1/2}_{t}(K_n)\boldsymbol{\epsilon}\|_n^2 > \frac{\sigma^2}{n}\Lambda(y)\Big)+\bP(\SDP > t)\\
&\leq \bP\Big( \|K_n^{1/2}g^{1/2}_{t}(K_n)\boldsymbol{\epsilon}\|_n^2 > \frac{\sigma^2}{n} \geffdim_n(t) + \frac{\sigma^2}{n} y\Big)+\bP(\SDP > t),
\end{align*}
and the claim follows from applying Lemma \ref{LemVarHansonWright} and Lemma \ref{LemEffDimGenDim} to the second last term and Proposition \ref{prop.deviation.SDP} and Lemma \ref{LemEffDimGenDim} to the last term, using that $t>\tildetstar$ and $y=\tildengeffdim(t)-\tildengeffdim(\Ltstar)$.
\end{proof}

\medskip

\begin{proof}[Proof of Proposition~\ref{prop.variance.term.random}]
First, by Lemma \ref{LemChangeNorm1.main}, we have on the event $\mathcal{E}_T$,
\begin{align*}
    \|S_\rho g_{\SDP}(\Sigma_n)S_n^*\boldsymbol{\epsilon}\|_\rho^2\leq 2\|S_n g_{\SDP}(\Sigma_n)S_n^*\boldsymbol{\epsilon}\|_n^2 + T^{-1}\|g_{\SDP}(\Sigma_n)S_n^*\boldsymbol{\epsilon}\|_{\mathcal{H}}^2
\end{align*}
Applying \eqref{LFU} and the fact that $\SDP\leq T$, and then \eqref{BdF}, we get
\begin{align}
    \|S_\rho g_{\SDP}(\Sigma_n)S_n^*\boldsymbol{\epsilon}\|_\rho^2&\leq  2 \|S_n g_{\SDP}(\Sigma_n)S_n^*\boldsymbol{\epsilon}\|_n^2+T^{-1}\|g_{\SDP}(\Sigma_n)S_n^*\boldsymbol{\epsilon}\|_{\mathcal{H}}^2\nonumber\\
    &\leq 2 \|S_n g_{\SDP}(\Sigma_n)S_n^*\boldsymbol{\epsilon}\|_n^2+B\|g^{1/2}_{\SDP}(\Sigma_n)S_n^*\boldsymbol{\epsilon}\|_{\mathcal{H}}^2\nonumber\\
    &=2 \|K_ng_{\SDP}(K_n)\boldsymbol{\epsilon}\|_n^2+B\|K_n^{1/2}g^{1/2}_{\SDP}(K_n)\boldsymbol{\epsilon}\|_n^2\nonumber\\
    &\leq (2+B)\|K_n^{1/2}g^{1/2}_{\SDP}(K_n)\boldsymbol{\epsilon}\|_n^2.\label{eq.variance.term.random}
\end{align}
Hence, on the event $\mathcal{E}_T$,
\begin{align*}
    \mathbf{P}_{\boldsymbol{\epsilon}}(\|S_\rho g_{\SDP}(\Sigma_n)S_n^*\boldsymbol{\epsilon}\|_\rho^2>y(u))\leq \mathbf{P}_{\boldsymbol{\epsilon}}((2+B)\|K_n^{1/2}g^{1/2}_{\SDP}(K_n)\boldsymbol{\epsilon}\|_n^2>y(u)),
\end{align*}
and the claim follows from Lemma \ref{deviation.ineq.variance.term.change.norm.SDP} applied with $y=C(\sqrt{\effdim_n(T)u}+u)$ and the fact that the assumption $T \leq cn/(\log n)$ with $c$ small enough implies that $T\log(\lfloor M^2T\rfloor+1)\leq n$.
\end{proof}

\subsubsection{A bound for the bias part}\label{SecRDBiasPart}
\begin{proposition}\label{lem.bias.term.random}
Under the assumptions of Theorem~\ref{CorExtThmSDPBoundedK}, we have on the event $\mathcal{E}_T\cap \mathcal{A}(T,\lfloor M^2T\rfloor)$,
\begin{align*}
    \mathbf{P}_{\boldsymbol{\epsilon}}(\|S_\rho r_{\SDP}(\Sigma_n)f\|_\rho^2> z(u))\leq 2e^{-u},\quad,u>0,
\end{align*}
with 
\begin{align*}
    z(u)=C\Big(\|r_{\tildetstar\wedge T}(K_n)\tilde{\mathbf{f}}\|_n^2+\frac{\sqrt{u\mathcal{N}_n(T)}+u}{n}+\frac{1}{T^{1+2s}}+\frac{\|\Sigma-\Sigma_n\|_{\operatorname{op}}^{2\wedge 2s}}{T}\Big),
\end{align*}
If $s\leq 1/2$, then the last term in the definition of $z(u)$ can be dropped.
\end{proposition}

\begin{proof}[Proof of Proposition~\ref{lem.bias.term.random}]
First, note that under $s=r-1/2\geq 0$ the regression function $f$ can be represented as a function in $\mathcal{H}$. By Lemma \ref{LemChangeNorm1.main}, we have on the event $\mathcal{E}_T$,
\begin{align*}
     \|S_\rho r_{\SDP}(\Sigma_n)f\|_\rho^2\leq 2\|S_n r_{\SDP}(\Sigma_n) f\|_n^2 + T^{-1}\|r_{\SDP}(\Sigma_n)f\|_{\mathcal{H}}^2.
\end{align*}
Using this and \eqref{BdF}, we get
\begin{align*}
    & \|S_\rho r_{\SDP}(\Sigma_n)f\|_\rho^2
   \leq \sum_{j\geq 1}(2\hat\lambda_j+T^{-1})r_{\SDP}^2(\hat\lambda_j)\langle f,\hat u_j\rangle^2\\
   &\leq 3b^{-1}\sum_{\hat\lambda_jT > 1}\hat\lambda_jr_{\SDP}^2(\hat\lambda_j)\tilde{g}_T(\hat\lambda_j)\hat\lambda_j\langle f,\hat u_j\rangle^2 + 3T^{-1}\sum_{\hat\lambda_jT\leq 1}\langle f,\hat u_j\rangle^2.
\end{align*}
Using \eqref{EqChangeBias2}, we get on $\mathcal{E}_T$,
\begin{align*}
    \|S_\rho r_{\SDP}(\Sigma_n)f\|_\rho^2\leq 3b^{-1}\|r_{\SDP}(K_n)\Af\|_n^2+ z(u)/2,
\end{align*}
provided that the constant $C$ in the definition of $z(u)$ is six times as big as the constant in \eqref{EqChangeBias2}. Hence, on the event $\mathcal{E}_T$,
\begin{align*}
    \mathbf{P}_{\boldsymbol{\epsilon}}(\|S_\rho r_{\SDP}(\Sigma_n)f\|_\rho^2 > z(u)) & \leq \mathbf{P}_{\boldsymbol{\epsilon}}(6b^{-1}\|r_{\SDP}(K_n)\Af\|_n^2 > z(u)),
\end{align*}
and the claim follows from \eqref{EqThmCondDesignProxySmoothedNorm3}, provided that $C$ in the definition of $z(u)$ is chosen large enough. 
\end{proof}

\subsection{Proofs of oracle inequalities (inner case)}\label{sec.proofs.oracle ineq.inner.case}
\subsubsection{Proof of Theorem \ref{CorExtThmSDPBoundedK}}
\label{proof.Theorem.oracle.random.T.Deterministic}
Since $s=r-1/2\geq 0$, $f$ can be represented as a function in $\mathcal{H}$. In particular, we can write $\mathbf{Y}=S_nf+\boldsymbol{\epsilon}$, leading to
\begin{align*}
   f-\hat f^{(\SDP)}&=f-g_{\SDP}(\Sigma_n)\Sigma_n f-g_{\SDP}(\Sigma_n)S_n^*\boldsymbol{\epsilon}\\
   &=r_{\SDP}(\Sigma_n)f-g_{\SDP}(\Sigma_n)S_n^*\boldsymbol{\epsilon}.
\end{align*}
Hence,
\begin{align*}
    \|S_\rho(f-\hat f^{(\SDP)})\|_\rho^2&\leq 2\|S_\rho r_{\SDP}(\Sigma_n)f\|_\rho^2+2\|S_\rho g_{\SDP}(\Sigma_n)S_n^*\boldsymbol{\epsilon}\|_\rho^2.
\end{align*}
The last but one term is addressed by Lemma~\ref{lem.bias.term.random} and the last one by Proposition~\ref{prop.variance.term.random}. Combining these estimates with \eqref{Eqtstarbound2}, introducing the event $\Omega_T=\mathcal{E}_T\cap \mathcal{A}(T,\lfloor M^2T\rfloor)$, we get on the event $\Omega_T$,
\begin{align*}
    \mathbf{P}_{\boldsymbol{\epsilon}}(\|S_\rho(f-\hat f^{(\SDP)})\|_\rho^2> x(u))\leq 5e^{-u},\quad u>0,
\end{align*}
with 
\begin{align*}
    x(u)=C\Big(&\min_{0< t\leq T}\Big\{\|r_t(K_n)\mathbf{f}\|_n^2+\frac{\mathcal{N}_n^g(t)}{n}\Big\}\\&+\frac{\sqrt{u\mathcal{N}_n(T)}+u+1}{n}+\frac{1}{T^{1+2s}}+\frac{\|\Sigma-\Sigma_n\|_{\operatorname{op}}^{2\wedge 2s}}{T}\Big),
\end{align*}
where the last term in the definition of $x(u)$ can be dropped if $s\leq 1/2$. 
Invoking the last claim in Lemma~\ref{LemChangeBias} and Lemma~\ref{LemEffDimGenDim}, we get on the event $\Omega_T$,
\begin{align*}
    \mathbf{P}_{\boldsymbol{\epsilon}}(\|S_\rho(f-\hat f^{(\SDP)})\|_\rho^2> \tilde x(u))\leq 5e^{-u},\quad u>0.
\end{align*}
with 
\begin{align*}
    \tilde x(u)=C\Big(&\min_{0< t\leq T}\Big\{\frac{1}{t^{1+2s}}+\frac{\mathcal{N}_n(t)}{n}+\frac{\|\Sigma-\Sigma_n\|_{\operatorname{op}}^{2\wedge 2s}}{t}\Big\}+\frac{\sqrt{u\mathcal{N}_n(T)}+u}{n}\Big),
\end{align*}
where the last term in the curly brackets in the definition of $\tilde x(u)$ can be dropped if $s\leq 1/2$.
Integrating this inequality on the event $\Omega_T$, we get
\begin{align*}
&\mathbb{E}\1_{\Omega_T}\|S_\rho(f-\hat f^{(\SDP)})\|_\rho^2 = \mathbb{E}\1_{\Omega_T} \mathbf{E}_{\boldsymbol{\epsilon}}\|S_\rho(f-\hat f^{(\SDP)})\|_\rho^2\\
&\leq C\Big(\min_{1\leq t\leq T}\Big\{\frac{1}{t^{1+2s}}+\frac{\mathbb{E}\1_{\Omega_T}\mathcal{N}_n(t)}{n} + \frac{\mathbb{E}\|\Sigma-\Sigma_n\|_{\operatorname{op}}^{2\wedge 2s}}{t}\Big\}+\frac{\mathbb{E}\1_{\Omega_T}\sqrt{\mathcal{N}_n(T)}+1}{n}\Big),
\end{align*}
where the last term in the curly brackets can be dropped if $s\leq 1/2$. Here, we have replaced the minimum over $0<t\leq T$ by $1\leq t\leq T$ since the range $t\in(0,1]$ does not yield any improvement.
Focusing now on $\mathbb{E}\1_{\Omega_T}\|\Sigma-\Sigma_n\|_{\operatorname{op}}^{2\wedge 2s}$, this latter term can be tackled by first
\begin{align*}
    \mathbb{E}\|\Sigma-\Sigma_n\|_{\operatorname{op}}^{2\wedge 2s}\leq (\mathbb{E}\|\Sigma-\Sigma_n\|_{\operatorname{op}}^2)^{1\wedge s}\leq (\mathbb{E}\|\Sigma-\Sigma_n\|_{\operatorname{HS}}^2)^{1\wedge s}.
\end{align*}
Then, since the random variables $k_{X_i}\otimes k_{X_i}-\Sigma$ are centered and independent, we have
\begin{align}\label{EqOrderCovOp}
\mathbb{E}\|\Sigma-\Sigma_n\|_{\operatorname{HS}}^2\leq \frac{1}{n}\mathbb{E}\|k_{X}\otimes k_{X}\|_{\operatorname{HS}}^2=\frac{1}{n}\mathbb{E}\|k_{X}\|_{\mathcal{H}}^4\leq \frac{M^4}{n}.
\end{align}
Using the Cauchy-Schwarz inequality, the second claim in Lemma \ref{LemConcEffDim} and the previous bound, we get
\begin{align*}
&\mathbb{E}\1_{\Omega_T}\|S_\rho(f-\hat f^{(\SDP)})\|_\rho^2\\
&\leq C\Big(\min_{1\leq t\leq T}\Big\{\frac{1}{t^{1+2s}}+\frac{1}{tn^{1\wedge s}}+\frac{\mathcal{N}(t)+ne^{-n/t}}{n}\Big\}
+ \frac{\sqrt{\mathcal{N}(T)+ne^{-n/T}}}{n}\Big),
\end{align*}
where the second term $t^{-1}n^{-(1\wedge s)}$ is only present for $s>1/2$. 

We now show that this term can also be dropped for $s>1/2$. If $s\geq 1$, this is clear using $t^{-1}n^{-1\wedge s}\leq n^{-1}$. Assume now that $s\in(1/2,1)$. If $t\leq \sqrt{n}$, then $t^{-1}n^{-s}\leq t^{-1-2s}$, while if $t> \sqrt{n}$ then $t^{-1}n^{-s}\leq n^{-1/2-s}\leq n^{-1}$. Moreover, the terms $n e^{-n/t}$ and $n e^{-n/T}$ can also be dropped using the condition $1\leq t\leq T\leq c_1n/(\log n)$ with $c_1$ small enough. We thus get 
\begin{align*}
&\mathbb{E}\1_{\Omega_T}\|S_\rho(f-\hat f^{(\SDP)})\|_\rho^2\leq C\Big(\min_{1\leq t\leq T}\Big\{\frac{1}{t^{1+2s}}+\frac{\mathcal{N}(t)}{n}\Big\}
+ \frac{\sqrt{\mathcal{N}(T)}}{n}\Big).
\end{align*}

The last part of the proof consists in analyzing the prediction error on the complement of the event $\Omega_T$.
Since $\|\hat f^{(t)}\|_{\mathcal{H}}^2$ is non-decreasing in $t\geq 1$, we have $\|\hat f^{(\SDP)}\|_{\mathcal{H}}^2\leq \|\hat f^{(T)}\|_{\mathcal{H}}^2$. Moreover, applying \eqref{BdF} and \eqref{LFU}, we get $\|\hat f^{(T)}\|_{\mathcal{H}}^2\leq B T\|\mathbf{Y}\|_n^2$. Hence,
$\|S_\rho\hat f^{(\SDP)}\|_{\rho}^2\leq \lambda_1B T\|\mathbf{Y}\|_n^2$ and
\begin{align}
    &\|S_\rho (f-\hat f^{(\SDP)})\|_\rho^2\leq 2\|S_\rho f\|_\rho^2+2\lambda_1B T\|\mathbf{Y}\|_n^2\nonumber\\
    &\leq 2\|S_\rho f\|_\rho^2+4M^2\lambda_1 BT\|f\|_{\mathcal{H}}^2+4\lambda_1 BT\|\boldsymbol{\epsilon}\|_n^2\leq C(1+T\|\boldsymbol{\epsilon}\|_n^2)\label{EqSmallEventBound1},
\end{align}
where we applied $\|S_n f\|_n^2=(1/n)\sum_{i=1}^n\langle f,k_{X_i}\rangle_{\mathcal{H}}^2\leq M^2\|f\|_{\mathcal{H}}^2$ in the second inequality. Using $T\leq c_1n/(\log n)$ with $c_1$ small enough, we get $\mathbb{P}(\Omega_T^c) \leq \P(A(n/T,3M^2T)^c) + \mathbb{P}(\mathcal{E}_T^c)\leq 2 C_1Te^{-c_2n/T}\leq 2C_2n^{-C_3}$ with $C_3>4$. Using the Cauchy-Schwarz inequality and \eqref{SubGN} it follows that
\begin{align}\label{EqSmallEventBound2}
    \mathbb{E}\1_{\Omega_T^c}\|S_\rho(f-\hat f^{(\SDP)})\|_\rho^2\leq Cn^{-1}
\end{align}
and the claim follows.

\subsubsection{Proof of Theorem \ref{CorExtThmSDPBoundedKEmpT}}
\label{proof.Theorem.oracle.random.T.Random}
We prove the result in the case $s\leq 1/2$, the other case follows similarly. From previous Section~\ref{proof.Theorem.oracle.random.T.Deterministic}, let us consider the event $\Omega_{T_1}=\mathcal{E}_{T_1}\cap \mathcal{A}(n/T_1,3M^2T_1)$, where $T_1=c_1 n/\log n$ and $c_1$ is sufficiently small such that $\mathbb{P}(\Omega_{T_1})\leq n^{-4}$ (such a choice is possible by Lemma~\ref{LemConcIneqBoundedK} and Lemma \ref{LemConcIneqTraces}). 

We first show that with $T = \min(T_1,\hat T)$, we have on the event $\Omega_{T_1}$
\begin{align}
\|r_{\tildetstar\wedge T}(K_n)\tilde{\mathbf{f}}\|_n^2 + \frac{\sigma^2}{n}\tildengeffdim(\tildetstar) & \leq C\Big(\min_{t>0}\Big\{t^{-2r}+\frac{\effdim_n(t)}{n}\Big\}+\frac{\log n}{n}\Big)\label{EqEqtstarEmpTMain}
\end{align}
By the definition of $\tildetstar$ (Eq.~\eqref{Deftildetstar}), Eq.~\eqref{Eqtstarbound2} and Lemma \ref{LemChangeBias}, we have on the event $\Omega_{T_1}$
\begin{align}
\|r_{\tildetstar\wedge T}(K_n)\tilde{\mathbf{f}}\|_n^2+\frac{\sigma^2}{n} \tildengeffdim(\tildetstar)   &\leq C\min_{0<  t\leq T}\Big\{t^{-2r}+\frac{\effdim_n(t)}{n}\Big\}.\label{EqTstarEmpT}
\end{align}
On the one hand, if $\hat T> T_1$, then $T=T_1$ and $T_1\mathcal{N}_n(T_1)<n$ and thus (since $2r\geq 1$)
\begin{align*}
\min_{0<  t\leq T}\Big\{t^{-2r}+\frac{\effdim_n(t)}{n}\Big\} \leq \frac{1}{T_1} + \frac{\effdim_n(T_1)}{n}<\frac{2}{T_1} \leq \frac{2}{c_1} \frac{\log n}{n} \cdot
\end{align*}
On the other hand, if $\hat T \leq  T_1$, then $T=\hat T$ and $t_n$ defined by $t_n^{2r}\mathcal{N}_n(t_n)=n$ satisfies either $1\leq  t_n\leq \hat T$ or $0<t_n<1$. In the former case the right-hand side of \eqref{EqTstarEmpT} is bounded by $2C\min_{t>0}\{t^{-2r}+n^{-1}\effdim_n(t)\}$, where the constraint that $t\leq T$ has been removed, while in the latter case the bound \eqref{EqEqtstarEmpTMain} is trivial since $2\min_{t>0}\{t^{-2r}+n^{-1}\effdim_n(t)\}\geq t_n^{-2r}+n^{-1}\effdim_n(t_n)\geq 1$ in this case. This completes the proof of \eqref{EqEqtstarEmpTMain}. 

Similarly, by the definition of $T$, we have
\begin{align*}
    \frac{\sqrt{\mathcal{N}_n(T)}}{n} = \frac{1}{\sqrt{n}} \sqrt{ \frac{\mathcal{N}_n(T)}{n} }\leq C\Big(\sqrt{\frac{1}{n}\min_{t>0}\Big\{ t^{-1}+\frac{\mathcal{N}_n(t)}{n}\Big\}}+\frac{\log n}{n}\Big)
\end{align*}
We can now proceed as in Proposition \ref{prop.variance.term.random} and Proposition \ref{lem.bias.term.random} to obtain on the event $\Omega_{T_1}$
\begin{align*}
    &\mathbf{E}_\epsilon\|S_\rho(f-\hat{f}^{(\SDP)})\|_{\rho}^2\\
    &\leq C\Big(\min_{t\geq 1}\Big\{t^{-2r}+\frac{1}{n}\effdim_n(t)\Big\}+\sqrt{\frac{1}{n}\min_{t>0}\Big\{ t^{-1}+\frac{\mathcal{N}_n(t)}{n}\Big\}}+\frac{\log n}{n}\Big).
\end{align*}
Here we used that $T$ does only depend on the design and is thus fixed conditional on the design. Hence, taking expectation and using Lemma \ref{LemConcEffDim} and Remark \ref{RemConcEffDim}, we conclude
\begin{align*}
    &\mathbb{E}\1_{\Omega_{T_1}}\|S_\rho(f-\hat{f}^{(\SDP)})\|_{\rho}^2  \\
    &\leq C\Big(\min_{t>0}\Big\{t^{-2r}+\frac{\effdim(t)}{n}\Big\}+\sqrt{\frac{1}{n}\min_{t>}\Big\{ t^{-1}+\frac{\mathcal{N}(t)}{n}\Big\}}+\frac{\log n}{n}\Big).
\end{align*} 

The claim follows from the final arguments in the proof of Theorem~\ref{CorExtThmSDPBoundedK}, showing that
\begin{align*}
    &\mathbb{E}\|S_\rho(f-\hat{f}^{(\SDP)})\|_{\rho}^2\leq \mathbb{E}\croch{\1_{\Omega_{T_1}}\|S_\rho(f-\hat{f}^{(\SDP)})\|_{\rho}^2 } + Cn^{-1}.
\end{align*}.

\subsection{Proofs of oracle inequalities (outer case)}
\label{sec.proof.DP.outer.case.random.design}
\subsubsection{Proof of Theorem \ref{ThmDPHardProblems}}
For simplicity, we prove Theorem \ref{ThmDPHardProblems} only in the case of Tikhonov regularization. Throughout the proof, we set $T=cn/(\log n)$ with $c$ sufficiently small such that 
\begin{align}\label{EqEventPolSmall}
	\mathbb{P}(\mathcal{E}_{T}^c)\leq n^{-C},\qquad C>4.
\end{align}
Such a choice is possible by Lemma \ref{LemConcIneqBoundedK}.

\begin{lemma}\label{LemBoundBiasTR} Suppose that \eqref{Assume.SC} holds with $0<r\leq 1/2$. For $t\geq 1$, let $f^{(t)}=(\Sigma + t^{-1})^{-1}S_\rho^* f\in\mathcal{H}$. Then we have
	\begin{itemize}
		\item[(i)] $\|f-S_\rho f^{(t)}\|_{\rho}^2\leq t^{-2r}R^2$,
		\item[(ii)] $\|f^{(t)}\|_{\mathcal{H}}^2\leq t^{-2r+1}R^2$.
	\end{itemize}
\end{lemma}
\begin{proof}[Sketch of proof of Lemma~\ref{LemBoundBiasTR}]
	Part (i) follows from Theorem 4 in \cite{MR2186447} applied with $\lambda=t^{-1}$. Part (ii) can be proved analogously; see e.g. Proposition 3 in \cite{C06}.
\end{proof}

\begin{lemma}\label{prop.variance.term.random.outer}
	Under the assumptions of Theorem~\ref{ThmDPHardProblems}, we have on $\mathcal{E}_T$,
	\begin{align*}
		\mathbf{E}_\epsilon\|S_\rho g_{\DP}(\Sigma_n) S_n^*\boldsymbol{\epsilon} \|_\rho^2 \leq C\Big(\min_{0< t\leq c\frac{n}{\log n}}\Big\{\|r_t(K_n)\mathbf{f}\|_n^2+\frac{\mathcal{N}_n(t)}{n}\Big\}+\frac{1} {\sqrt{n}}\Big).
	\end{align*}
\end{lemma}

\begin{proof}[Proof of Lemma~\ref{prop.variance.term.random.outer}]
	By \eqref{eq.variance.term.random} with $\SDP$ replaced by $\DP$, we have on the event $\mathcal{E}_T$,
	\begin{align*}
		\|S_\rho g_{\DP}(\Sigma_n)S_n^*\boldsymbol{\epsilon}\|_\rho^2\leq  (2+B)\|K_n^{1/2}g^{1/2}_{\DP}(K_n)\boldsymbol{\epsilon}\|_n^2.
	\end{align*}
	Applying \eqref{EqThmCondDesignProxy3}, we get on the event $\mathcal{E}_T$ and for every $u>0$,
	\begin{align*}
		& \mathbf{P}_\epsilon\Big(\|S_\rho g_{\DP}(\Sigma_n)S_n^*\boldsymbol{\epsilon}\|_\rho^2> C\Big(\frac{\mathcal{N}_n(\tstar)}{n}+\frac{\sqrt{u}}{\sqrt{n}}+\frac{u}{n}\Big)\Big) \\
		& \leq \mathbf{P}_\epsilon\Big((2+B)\|K_n^{1/2}g^{1/2}_{\DP}(K_n)\boldsymbol{\epsilon}\|_n^2> C\Big(\frac{\mathcal{N}_n(\tstar)}{n}+\frac{\sqrt{u}}{\sqrt{n}}+\frac{u}{n}\Big)\Big)  \leq 3e^{-u}
	\end{align*}
	with $C$ sufficiently large. Integrating this inequality and inserting \eqref{Eqtstarbound1}, the claim follows.
\end{proof}

\begin{lemma}\label{LemChangeNormOC}
	Under the assumptions of Theorem \ref{ThmDPHardProblems}, we have 
	\[
	\mathbb{E}\|f-S_\rho\hat f^{(\DP)}\|_\rho^2\leq C\Big(\mathbb{E}\1_{\mathcal{E}_{T}}\min_{0< t\leq c\frac{n}{\log n}}\Big\{\|r_t(K_n)\mathbf{f}\|_n^2+\frac{\mathcal{N}_n(t)}{n}\Big\}+\frac{1} {\sqrt{n}}+\Big(\frac{\log n}{n}\Big)^{2r}\Big).
	\]
\end{lemma}
\begin{proof}[Proof of Lemma~\ref{LemChangeNormOC}]
	We have
	\begin{align*}
		\mathbb{E}\|f-S_\rho\hat f^{(\DP)}\|_\rho^2&
		\leq   \mathbb{E}\1_{\mathcal{E}_{T}}\|f-S_\rho\hat f^{(\DP)}\|_\rho^2+2\mathbb{E}\1_{\mathcal{E}_{T}^c}\|f-S_\rho\hat f^{(\DP)}\|_\rho^2   \\
		&\leq \mathbb{E}\1_{\mathcal{E}_{T}}\|f-S_\rho\hat f^{(\DP)}\|_\rho^2 + Cn^{-1}, 
	\end{align*}
	where the second inequality follows by the same line of arguments as at the end of the proof of Theorem~\ref{CorExtThmSDPBoundedK}  (cf. \eqref{EqSmallEventBound1} and \eqref{EqSmallEventBound2}), using that $f$ is bounded this time which implies $\|g_{\DP}^{1/2}(K_n)K_n^{1/2}\mathbf{f} \|_n^2\leq \norm{f}_{\infty}^2$.
	
	Let us now introduce, for $t_1>0$ to be chosen later, 
	\begin{align*}
		f-S_\rho\hat f^{(\DP)}=f-S_\rho f^{(t_1)}+S_\rho f^{(t_1)}-S_\rho g_{\DP}(\Sigma_n)S_n^*\mathbf{f}-S_\rho g_{\DP}(\Sigma_n)S_n^*\boldsymbol{\epsilon},
	\end{align*}
	where $f^{(t_1)}=(\Sigma+t_1^{-1})^{-1}S_\rho^* f$.
	It results that
	\begin{align*}
		& \frac{1}{3}\mathbb{E}\1_{\mathcal{E}_{T}}\|f-S_\rho\hat f^{(\DP)}\|_\rho^2\\
		&\leq \|f-S_\rho f^{(t_1)}\|_\rho^2+\mathbb{E}\1_{\mathcal{E}_{T}}\|S_\rho g_{\DP}(\Sigma_n)S_n^*\boldsymbol{\epsilon}\|_\rho^2+\mathbb{E}\1_{\mathcal{E}_{T}}\|S_\rho f^{(t_1)}-S_\rho g_{\DP}(\Sigma_n)S_n^*\mathbf{f}\|_\rho^2\\
		&=:I_1+I_2+I_3.
	\end{align*}
	Form Lemma~\ref{LemBoundBiasTR}(i), we get $I_1\leq R^2 t_1^{-2r}$, and Lemma~\ref{prop.variance.term.random.outer} provides
	\begin{align*}
		& I_2 =\mathbb{E}\1_{\mathcal{E}_{T}}\mathbf{E}_\epsilon\|S_\rho g_{\DP}(\Sigma_n)S_n^*\boldsymbol{\epsilon}\|_\rho^2\\
		&\leq C\Big(\mathbb{E}\1_{\mathcal{E}_{T}}\min_{0< t\leq c\frac{n}{\log n}}\Big\{\|r_t(K_n)\mathbf{f}\|_n^2+\frac{\mathcal{N}_n(t)}{n}\Big\}+\frac{1} {\sqrt{n}}\Big).
	\end{align*}
	The remainder of this proof consists in considering the term $I_3$. 
	
	By the change of norm argument of Lemma~\ref{LemChangeNorm1.main} applied to functions belonging to  $\H$, on the event $\mathcal{E}_{T}$, we have
	\begin{align}
		&\|S_\rho f^{(t_1)}-S_\rho g_{\DP}(\Sigma_n)S_n^*\mathbf{f}\|_\rho^2 \nonumber\\
		&\leq \| \mathbf{f}^{(t_1)}-g_{\DP}(K_n)K_n\mathbf{f}\|_n^2+T^{-1}\| f^{(t_1)}-g_{\DP}(\Sigma_n)S_n^*\mathbf{f}\|_{\mathcal{H}}^2.\label{ineq.rho.norm.outer}
	\end{align}
	
	\textbf{Empirical norm in \eqref{ineq.rho.norm.outer}:} Integrating yields 
	\begin{align*}
		\mathbb{E}\1_{\mathcal{E}_{T}}\| \mathbf{f}^{(t_1)}-g_{\DP}(K_n)K_n\mathbf{f}\|_n^2
		&\leq 2\mathbb{E}\| \mathbf{f}-\mathbf{f}^{(t_1)}\|_n^2+2\mathbb{E}\1_{\mathcal{E}_{T}}\| r_{\DP}(K_n)\mathbf{f}\|_n^2\\
		& = 2\| f-S_\rho f^{(t_1)}\|_\rho^2+2\mathbb{E}\1_{\mathcal{E}_{T}}\| r_{\DP}(K_n)\mathbf{f}\|_n^2.
	\end{align*}
	The first term in the r.h.s. is addressed by Lemma~\ref{LemBoundBiasTR}(i), leading to the upper bound $2R^2 t_1^{-2r}$. 
	For the second one, integrating~\eqref{EqThmCondDesignProxy2} with $T=cn/\log n$ and inserting~\eqref{Eqtstarbound1}, we get 
	\begin{align*}
		\mathbb{E}\1_{\mathcal{E}_{T}}\| r_{\DP}(K_n)\mathbf{f}\|_n^2\leq C\Big(\mathbb{E}\1_{\mathcal{E}_{T}}\min_{0< t\leq c\frac{n}{\log n}}\Big\{\|r_t(K_n)\mathbf{f}\|_n^2+\frac{\mathcal{N}_n(t)}{n}\Big\}+\frac{1} {\sqrt{n}} \Big).
	\end{align*}
	
	\textbf{Hilbert norm in \eqref{ineq.rho.norm.outer}:} We have
	\begin{align}\label{EqHNEstim}
		\| f^{(t_1)}-g_{\DP}(\Sigma_n)S_n^*\mathbf{f}\|_{\mathcal{H}}^2&=\| f^{(t_1)}-g_{\DP}(\Sigma_n)\Sigma_nf^{(t_1)} + g_{\DP}(\Sigma_n)S_n^*(\mathbf{f}^{(t_1)}-\mathbf{f})\|_{\mathcal{H}}^2\nonumber\\
		&\leq 2\|r_{\DP}(\Sigma_n) f^{(t_1)}\|_{\mathcal{H}}^2 + 2\| g_{\DP}(\Sigma_n)S_n^*(\mathbf{f}^{(t_1)}-\mathbf{f})\|_{\mathcal{H}}^2\nonumber\\
		&\leq 2R^2 t_1^{1-2r} + 2BT\|\mathbf{f}^{(t_1)}-\mathbf{f}\|_n^2,
	\end{align}
	where we applied \eqref{BdF} and Lemma \ref{LemBoundBiasTR}(ii) to the first term and \eqref{BdF}, \eqref{LFU} and the inequality $\DP\leq T$ to the second term. 
	
	Collecting these bounds and using $T= c n/(\log n)$ and, we get
	\begin{align*}
		I_3\leq C\Big(\mathbb{E}\1_{\mathcal{E}_{T}}\min_{0< t\leq c\frac{n}{\log n}}\Big\{\|r_t(K_n)\mathbf{f}\|_n^2+\frac{\mathcal{N}_n(t)}{n}\Big\}+\frac{1} {\sqrt{n}} + t_1^{-2r}+\frac{\log n}{n} t_1^{1-2r}\Big).
	\end{align*}
	The claim now follows from these bounds for $I_1-I_3$ by setting $t_1 = cn/(\log n)$.
\end{proof}

\begin{lemma}\label{LemBoundBiasOC}
	For $t>0$ let $g_t(\lambda)=(\lambda+t^{-1})^{-1}$, and let $T=cn/(\log n)$. Suppose that \eqref{BdK} holds. Then we have
	\[
	\forall 0< t\leq T,\qquad
	\mathbb{E}\1_{\mathcal{E}_T}\|r_t(K_n)\mathbf{f}\|_n^2\leq C\Big(t^{-2r}+\frac{\mathcal{N}(t)}{n}\Big).
	\]
	Moreover, we have
	\begin{align*}
		\forall 0< t\leq T,\qquad
		\mathbb{E}\1_{\mathcal{E}_{T}}\mathcal{N}_n(t)\leq C_2(\mathcal{N}(t)+1).
	\end{align*}
\end{lemma}

\begin{proof}[Proof of Lemma~\ref{LemBoundBiasOC}]
	The second claim directly follows from Lemma \ref{LemConcEffDim} in combination with Remark \ref{RemConcEffDim}.
	
	For the first claim, set $f^{(t)}=(\Sigma+t^{-1})^{-1}S_\rho^* f$. By Lemma \ref{LemBoundBiasTR}, we have 
	\begin{align*}
		&\mathbb{E}\1_{\mathcal{E}_{T}}\|\mathbf{f}-S_n(\Sigma_n+t^{-1})^{-1}S_n^*\mathbf{f}\|_n^2\\
		&\leq 2\|f-S_\rho f^{(t)}\|_\rho^2+2\mathbb{E}\1_{\mathcal{E}_{T}}\|S_nf^{(t)}-S_n(\Sigma_n+t^{-1})^{-1}S_n^*\mathbf{f}\|_n^2\\
		&\leq 2R^2t^{-2r}+2\mathbb{E}\1_{\mathcal{E}_{T}}\|S_nf^{(t)}-S_n(\Sigma_n+t^{-1})^{-1}S_n^*\mathbf{f}\|_n^2.
	\end{align*}
	It remains to analyze the last term. Using Lemma~\ref{LemChangeNorm1.main} (change of norm), we have on $\mathcal{E}_{T}$,
	\begin{align*}
		&\|S_nf^{(t)}-S_n(\Sigma_n+t^{-1})^{-1}S_n^*\mathbf{f}\|_n^2\\
		&\leq 2\|S_\rho f^{(t)}-S_\rho(\Sigma_n+t^{-1})^{-1}S_n^*\mathbf{f}\|_\rho^2
		+C\frac{\log n}{n}\|f^{(t)}-(\Sigma_n+t^{-1})^{-1}S_n^*\mathbf{f}\|_{\mathcal{H}}^2.
	\end{align*}
	By \eqref{EqHNEstim} (where $\DP$ is replaced by $t$), the $\mathcal{H}$-norm is bounded by $C(t^{1-2r}+t\|\mathbf{f}-\mathbf{f}^{(t)}\|_n^2)$ and thus on $\mathcal{E}_{T}$,
	\begin{align*}
		&\|S_n f^{(t)}-S_n(\Sigma_n+t^{-1})^{-1}S_n^*\mathbf{f}\|_n^2\\
		&\leq 2\|S_\rho f^{(t)}-S_\rho(\Sigma_n+t^{-1})^{-1}S_n^*\mathbf{f}\|_\rho^2+C(t^{-2r}+\|\mathbf{f}-\mathbf{f}^{(t)}\|_n^2),
	\end{align*}
	where we also used that $t\leq T = cn/(\log n)$. 
	Since $\mathbb{E}\|\mathbf{f}-\mathbf{f}^{(t)}\|_n^2= \|f-S_\rho f^{(t)}\|_\rho^2\leq R^2t^{-2r}$, as can be seen from Lemma \ref{LemBoundBiasTR}(i), it remains to bound the term
	\begin{align*}
		&2\|S_\rho f^{(t)}-S_\rho(\Sigma_n+t^{-1})^{-1}S_n^*\mathbf{f}\|_\rho^2\\
		& \leq2\|(\Sigma+t^{-1})^{1/2}((\Sigma+t^{-1})^{-1}S_\rho^* f-(\Sigma_n+t^{-1})^{-1}S_n^*\mathbf{f})\|_{\mathcal{H}}^2,
	\end{align*}
	where we used $\|S_\rho h\|_\rho^2=\|\Sigma^{1/2} h\|_{\mathcal{H}}^2\leq \|(\Sigma+t^{-1})^{1/2} h\|_{\mathcal{H}}^2$, $h\in\mathcal{H}$, in the inequality. Inserting 
	\begin{align*}
		&(\Sigma+t^{-1})^{-1}S_\rho^* f-(\Sigma_n+t^{-1})^{-1}S_n^*\mathbf{f}\\
		&=(\Sigma_n+t^{-1})^{-1}(S_\rho^* f-S_n^*\mathbf{f})-(\Sigma_n+t^{-1})^{-1}(\Sigma_n-\Sigma)(\Sigma+t^{-1})^{-1}S_\rho^* f
	\end{align*}
	and 
	\begin{align*}
		&(\Sigma_n+t^{-1})^{-1}=(\Sigma+t^{-1})^{-1/2}(I+A_{t})^{-1}(\Sigma+t^{-1})^{-1/2}
	\end{align*}
	with $A_t$ from \eqref{EqDefAT}, we get
	\begin{align*}
		&\|(\Sigma+t^{-1})^{1/2}((\Sigma+t^{-1})^{-1}S_\rho^* f-(\Sigma_n+t^{-1})^{-1}S_n^*\mathbf{f})\|_{\mathcal{H}}^2\\
		&\leq 2\|(I+A_{t})^{-1}(\Sigma+t^{-1})^{-1/2}(S_\rho^* f-S_n^*\mathbf{f})\|_{\mathcal{H}}^2\\
		&+2\|(I+A_{t})^{-1}(\Sigma+t^{-1})^{-1/2}(\Sigma_n-\Sigma)f^{(t)}\|_{\mathcal{H}}^2.
	\end{align*}
	In the proof of Lemma \ref{LemConcEffDim}, we have shown that on the event $\mathcal{E}_{T}$ we have $\|(I+A_{t})^{-1}\|_{\operatorname{op}}\leq 2$. Hence, on $\mathcal{E}_{T}$,
	\begin{align*}
		&\|(\Sigma+t^{-1})^{1/2}((\Sigma+t^{-1})^{-1}S_\rho^* f-(\Sigma_n+t^{-1})^{-1}S_n^*\mathbf{f})\|_{\mathcal{H}}^2\\
		&\leq 4\|(\Sigma+t^{-1})^{-1/2}(S_\rho^* f-S_n^*\mathbf{f})\|_{\mathcal{H}}^2+4\|(\Sigma+t^{-1})^{-1/2}(\Sigma_n-\Sigma)f^{(t)}\|_{\mathcal{H}}^2.
	\end{align*}
	We conclude that
	\begin{align*}
		\mathbb{E}\1_{\mathcal{E}_{T}}\|r_t(K_n)\mathbf{f}\|_n^2&\leq  
		8\mathbb{E}\|(\Sigma+t^{-1})^{-1/2}(S_\rho^* f-S_n^*\mathbf{f})\|_{\mathcal{H}}^2\\
		&+8\mathbb{E}\|(\Sigma+t^{-1})^{-1/2}(\Sigma_n-\Sigma)f^{(t)}\|_{\mathcal{H}}^2+Ct^{-2r}.
	\end{align*}
	By construction $S_n^*\mathbf{f}-S_\rho^*f$ is a sum of independent, zero-mean random variables. To see the second claim, use that for every $h\in\mathcal{H}$, we have $\mathbb{E}f(X)\langle k_X,h \rangle_{\mathcal{H}}=\langle f, S_\rho h\rangle_\rho=\langle S_\rho^*f, h\rangle_{\mathcal{H}}$, and thus $\mathbb{E}f(X)k_X=S_\rho^*f$
	Now, using the fact that $f$ is bounded, we have
	\begin{align*}
		&\mathbb{E}\|(\Sigma+t^{-1})^{-1/2}(S_\rho^* f-S_n^*\mathbf{f})\|_{\mathcal{H}}^2\leq\frac{1}{n}\mathbb{E}\|(\Sigma+t^{-1})^{-1/2}k_Xf(X)\|_{\mathcal{H}}^2\\
		&\leq\frac{1}{n}\|f\|_\infty^2\mathbb{E}\|(\Sigma+t^{-1})^{-1/2}k_X\|_{\mathcal{H}}^2=\|f\|_\infty^2\frac{\mathcal{N}(t)}{n}.
	\end{align*}
	Similarly, we have
	\begin{align*}
		&\mathbb{E}\|(\Sigma+t^{-1})^{-1/2}(\Sigma_n-\Sigma)f^{(t)}\|_{\mathcal{H}}^2\\
		&\leq  \frac{1}{n}\mathbb{E}\|(\Sigma+t^{-1})^{-1/2}k_X\langle k_X ,f^{(t)}\rangle_{\mathcal{H}}\|_{\mathcal{H}}^2\\
		&\leq \frac{2}{n}\mathbb{E}\|(\Sigma+t^{-1})^{-1/2}k_X\|_{\mathcal{H}}^2((f(X))^2+(f^{(t)}(X)-f(X))^2).
	\end{align*} 
	Using that that $f$ is bounded, the fact that $\|(\Sigma+t^{-1})^{-1/2}k_X\|_{\mathcal{H}}^2\leq M^2t$ and Lemm \ref{LemBoundBiasTR}(i), we get
	\begin{align*}
		&\mathbb{E}\|(\Sigma+t^{-1})^{-1/2}(\Sigma_n-\Sigma)f^{(t)}\|_{\mathcal{H}}^2\\
		&\leq  \frac{2\|f\|_\infty}{n}\mathbb{E}\|(\Sigma+t^{-1})^{-1/2}k_X\|_{\mathcal{H}}^2+M^2t\|f-S_\rho f^{(t)}\|_\rho^2\\
		&\leq 2\|f\|_\infty\frac{\mathcal{N}(t)}{n}+R^2M^2\frac{t^{-2r+1}}{n}\leq  C\Big(\frac{\mathcal{N}(t)}{n}+t^{-2r}\Big),
	\end{align*} 
	where the last inequality follows from $t\leq c_1n/(\log n)$. This completes the proof.
\end{proof}

\begin{proof}[End of proof of Theorem  \ref{ThmDPHardProblems}]
	The claim follows from inserting Lemma \ref{LemBoundBiasOC} into Lemma \ref{LemChangeNormOC}.
\end{proof}
\subsubsection{Sketch of proof of Theorem \ref{ThmDPHardProblemsExt}} 
	The proof of Theorem \ref{ThmDPHardProblemsExt} follows from the arguments of the proof of Theorem \ref{ThmDPHardProblems}. The improvement is based on the fact that if additionally $\|\Sigma^{\mu/2-1/2}k_X\|_{\mathcal{H}}\leq C_\mu M$ holds for some $\mu\in[0,1)$, then one can improve the concentration and deviation bounds in Lemma \ref{LemConcIneqBoundedK} and Lemma \ref{LemConcEffDim} accordingly. First, Lemma \ref{LemConcIneqBoundedK} can be improved to $\mathbb{P}(\mathcal{E}_T^c)\leq C_1T^\mu\exp(-c_1n/T^\mu)$, since now $\|(\Sigma+t^{-1})^{-1/2}k_X\|_{\mathcal{H}}^2$ can be bounded by $C_\mu M^2t^\mu$. Similarly Lemma \ref{LemConcEffDim} can be improved to $\mathbb{E}\1_{\mathcal{E}_T}\mathcal{N}_n(t)\leq C\mathcal{N}(t)+2ne^{-n/t^\mu}$. In particular, setting $T=c(n/(\log n))^{1/\mu}$ with $c$ sufficiently small, we get  
	$\mathbb{P}(\mathcal{E}_{T}^c)\leq n^{-4}$ and $\mathbb{E}\1_{\mathcal{E}_{T}}\mathcal{N}_n(t)\leq C_2(\mathcal{N}(t)+1)$. We can now follow the same line of arguments from above to obtain Theorem \ref{ThmDPHardProblemsExt}. Only at the end of proof of Lemma \ref{LemBoundBiasOC}, we have to apply $\|(\Sigma+t^{-1})^{-1/2}k_X\|_{\mathcal{H}}^2\leq C_\mu M^2t^\mu$ once more.


\subsection*{Acknowledgement}
The authors thank Markus Rei\ss{} for helpful discussions and his great support for making this work possible.

\appendix

\section{Some useful operator bounds}
Let $A,B$ be two positive, compact operators $A$ and $B$ on $\mathcal{H}$. Then we have
\begin{align}\label{EqPowerOperatorNorm1}
    \|A^s-B^s\|_{\operatorname{op}}\leq \|A-B\|_{\operatorname{op}}^s,\quad 0\leq s \leq 1,
\end{align}
and
\begin{align}\label{EqPowerOperatorNorm2}
    \|A^s-B^s\|_{\operatorname{op}}\leq C_s(\|A\|_{\operatorname{op}}+\|A-B\|_{\operatorname{op}})^{s-1}\|A-B\|_{\operatorname{op}},\quad s>1.
\end{align}
Moreover, we have
\begin{align}\label{EqPowerOperatorNorm3}
    \|A^sB^s\|_{\operatorname{op}}\leq \|AB\|_{\operatorname{op}}^s,\quad 0\leq s \leq 1.
\end{align}
For a proof of the first and the third claim see Theorem X.1.1 and Theorem IX.2.1 in \cite{MR1477662}, for a proof of the second claim see e.g. \cite{MR3833647}.

\section{Effective dimension and eigenvalue bounds}\label{appendix.effective.dim}
The effective dimension $\mathcal{N}(t)$ of a positive self-adjoint trace-class operator $\Sigma$ is a continuous and non-decreasing function in $t\geq 0$. Moreover, under \eqref{BdK}, we have $\operatorname{tr}(\Sigma)=\mathbb{E}\|k_X\|_{\mathcal{H}}^2\leq M^2$, leading to $\mathcal{N}(t)\leq M^2t$ for all $t\geq 0$. Under additional assumption on the decay of the eigenvalues, this bound can be further improved.
\begin{lemma}\label{EqEffDim} (i) Suppose that for some $\alpha>1$ and $\ED>0$, we have $\lambda_j\leq \ED j^{-\alpha}$ for all $j\geq 1$. Then there is a constant $C>0$ depending only on $\alpha$ and $L$ such that $\mathcal{N}(t)\leq Ct^{1/\alpha}$ for all $t\geq L^{-1}$.

(ii) Suppose that for some $\alpha\in(0,1]$ and $\ED >0$, we have $\lambda_j\leq  e^{-\ED j^{\alpha}}$ for all $j\geq 1$. Then there is a constant $C>0$ depending only on $\alpha$ and $L$ such that $\mathcal{N}(t)\leq C(\log t)^{1/\alpha}$ for all $t\geq e^L$.
\end{lemma}
\begin{proof}
Part (i) is proved in Proposition 3 in \cite{caponnetto2007optimal}, see also Lemma 5.1 in \cite{MR3833647}. In order to get part (ii), we use that $\lambda/(\lambda+1/t)$ is increasing in $\lambda$, such that
\begin{align*}
\mathcal{N}(t)&\leq \sum_{j\geq 1}\frac{\ED e^{-\ED j^{\alpha}}}{\ED e^{-\ED j^{\alpha}}+1/t}.
\end{align*}
Defining $k\geq 1$ by $e^{-\ED (k+1)^{\alpha}}< 1/t\leq e^{-\ED k^{\alpha}}$ (using that $te^{-L}\geq 1$), we have 
\begin{align}
&\mathcal{N}(t)\leq \sum_{j\leq k}\frac{\ED e^{-\ED j^{\alpha}}}{\ED e^{-\ED j^{\alpha}}+1/t}+\sum_{j> k}\frac{\ED e^{-\ED j^{\alpha}}}{\ED e^{-\ED j^{\alpha}}+1/t}\nonumber\\
&\leq k+ t\sum_{j>k}e^{-\ED j^{\alpha}}\leq k+ Ct(k+1)^{1-\alpha}e^{-\ED (k+1)^{\alpha}}\leq k+C(k+1)^{1-\alpha},\label{EqEffDim1}
\end{align}
where we applied Equation (5.1) in \cite{MR4073556} in the third inequality. Now $1/t\leq e^{-\ED k^{\alpha}}$ implies $k\leq (L^{-1}\log t)^{1/\alpha}$ and inserting this into \eqref{EqEffDim1} gives the claim.
\end{proof}

\begin{lemma}\label{LemEVConc}
If $\|(\Sigma+T^{-1})^{-1/2}(\Sigma_n-\Sigma)(\Sigma+T^{-1})^{-1/2}\|_\infty\leq 1/2$ holds, then 
\[
\forall j\geq 1,\quad \lambda_j/2-1/(2T)\leq \hat\lambda_j\leq 3\lambda_j/2+1/(2T).
\]
\end{lemma}
\begin{proof}[Proof of Lemma~\ref{LemEVConc}]
We have 
\[
\|(\Sigma+T^{-1})^{-1/2}(\Sigma_n-\Sigma)(\Sigma+T^{-1})^{-1/2}\|_\infty\leq 1/2
\]
if and only if 
\[
 -(1/2)(\langle h,\Sigma h\rangle_{\mathcal{H}}+T^{-1})\leq \langle h,(\Sigma_n-\Sigma)h\rangle_{\mathcal{H}}\leq (1/2)(\langle h,\Sigma h\rangle_{\mathcal{H}}+T^{-1})
\]
for every $h\in\mathcal{H}$ such that $\|h\|_{\mathcal{H}}=1$. Rearranging the terms this is equivalent to
\[
(1/2)\langle h,\Sigma h\rangle_{\mathcal{H}}-1/(2T)\leq \langle h,\Sigma_n h\rangle_{\mathcal{H}}\leq (3/2)\langle h,\Sigma h\rangle_{\mathcal{H}}+1/(2T)
\]
for every $h\in\mathcal{H}$ such that $\|h\|_{\mathcal{H}}=1$. The claim now follows from the minimax characterization of eigenvalues. 
\end{proof}

\section{Concentration inequalities}
The following lemma is taken from \citep{MR3629418}. It is an extension of \citep{T15} from self-adjoint
matrices to self-adjoint Hilbert-Schmidt operators.
\begin{lemma}[From Lemma~5 in \cite{MR3629418}]\label{LemConcIneqTropp}
Let $\xi_1,\dots,\xi_n$ be a sequence of independently and identically distributed self-adjoint
Hilbert-Schmidt operators on a separable Hilbert space. Suppose that $\mathbb{E}\xi_1=0$ and $\|\xi_1\|_{\operatorname{op}}\leq R$ almost surely for some constant $R>0$. Moreover, suppose that there are constants $V,D > 0$ satisfying $\|\mathbb{E}\xi_1^2\|_{\operatorname{op}}\leq V$ and $\operatorname{tr}(\mathbb{E}\xi_1^2)\leq VD$. Then, for all $u \geq V^{1/2}n^{-1/2} + (3n)^{-1}R$,
\begin{align*}
    \mathbb{P}\Big(\Big\|\frac{1}{n}\sum_{i=1}^n\xi_i\Big\|_{\operatorname{op}}\geq u\Big)\leq 4D\exp\Big(-\frac{nu^2}{2V+(2/3)uR}\Big)
\end{align*}
\end{lemma}

\bibliography{ESRbiblio}

\end{document}